\documentclass[a4paper,11pt,reqno]{amsart}
\usepackage{amssymb,amsmath,hyperref,comment}
\title[On string functions of the generalized parafermionic theories]{On string functions of the generalized\\ parafermionic theories, mock theta functions, and\\false theta functions, II}

\author{Nikolay E. Borozenets}
\address{Department of Mathematics and Computer Science, Saint Petersburg State University, Saint Petersburg,  Russia, 199178}
\address{Igor Krichever Center for Advanced Studies, Skolkovo Institute of Science and Technology, Moscow, Russia, 121205}
\address{Faculty of Mathematics, National Research University Higher School of Economics, Moscow, Russia, 119048}
\email{nikolayborozenets.spbumcs@gmail.com}

\author{Eric T. Mortenson}
\address{Department of Mathematics and Computer Science, Saint Petersburg State University, Saint Petersburg,  Russia, 199178}
\email{etmortenson@gmail.com}

\renewcommand\theta{\vartheta}

\newcommand\sg{\operatorname{sg}}
\newtheorem{theorem}{Theorem}
\newtheorem{lemma}[theorem]{Lemma}
\newtheorem{corollary}[theorem]{Corollary}
\newtheorem{proposition}[theorem]{Proposition}
\theoremstyle{definition}

\newtheorem{remark}[theorem]{Remark}

\numberwithin{theorem}{section} 
\numberwithin{equation}{section}

\newcommand{\Z}{\mathbb{Z}}

\newcommand{\U}{{\text {\rm U}}}
\newcommand{\SL}{{\text {\rm SL}}}

\newcommand{\im}{\textnormal{Im}}

\makeatletter
\@namedef{subjclassname@2020}{\textup{2020} Mathematics Subject Classification}
\makeatother

\begin{document}

\date{5 April 2025}



\begin{abstract}
Kac and Wakimoto introduced the admissible highest weight representations  as a conjectural classification of all modular-invariant representations of the affine Kac--Moody algebras. For the affine Kac--Moody algebra $A_1^{(1)}$ their conjectural construction has been proved.  Using their construction, Ahn, Chung, and Tye introduced the generalized Fateev--Zamolodchikov parafermionic theories.  The characters of these parafermionic theories are string functions of admissible representations of $A_1^{(1)}$ up to a simple appropriate factor. Determining modular properties or explicitly calculating string functions and branching coefficients is an important yet wide-open problem.  Outside of initial works of Kac, Peterson, and Wakimoto, little is known.  Here we take a new approach by first developing a quasi-periodic notion of admissible string functions and then calculating the Zagier--Zwegers' polar-finite decomposition for the admissible characters.  As an application of the decomposition, we extend the results of our paper (Borozenets and Mortenson, 2024) for the affine Kac--Moody algebra $A_1^{(1)}$, in that we obtain families of new mock theta conjecture-like identities for $1/3$ and $2/3$-level string functions in terms of Ramanujan's mock theta functions $f_3(q)$ and $\omega_3(q)$. We also obtain an analogous family of new identities for the $1/5$-level string functions in terms of Ramanujan's four tenth-order mock theta functions. In addition, we give a heuristic argument for an expansion of the general positive-level admissible string functions in terms of Appell functions.
\end{abstract}

\maketitle

 \tableofcontents

\section{Introduction}

Kac and Wakimoto \cite{KW88pnas} initiated a program to describe the class of modular invariant representations of infinite-dimensional Lie algebras and superalgebras, that is, the representations with characters that possess modular behavior. This class is of special interest for many reasons. For modular invariant representations it is possible to establish many non-trivial asymptotic results, such as the calculation of asymptotic dimensions \cite{KW88pnas}, which allows one to compute asymptotic growth of weight multiplicities \cite{KP84}, or, in general, branching coefficients \cite{KW88advmath}. Also, modular invariant representations are of interest in conformal field theory and string theory as it is possible to use the modular bootstrap approach to determine the exact spectrum of the conformal theories on the torus with Lie-algebraic symmetry \cite{CIZ87, C86, GQ87, GW86, L89physlet}. For the affine Kac--Moody algebras, Kac and Wakimoto introduced the admissible highest weight representations as a conjectural classification of modular invariant representations \cite{KW88pnas,KW89,KW90}. Their conjecture is known to be true for the affine Kac--Moody algebra $A_1^{(1)}$, but for general affine Kac--Moody algebras it is still open \cite{KW24}.

Many conformal field theories, such as Belavin--Polyakov--Zamolodchikov minimal models \cite{BPZ} and Fateev--Zamolodchikov (FZ) parafermionic theories \cite{FZ85}, were interpreted in terms of  the Goddard--Kent--Olive (GKO) coset construction with affine Kac--Moody algebra $A_1^{(1)}$ as a main ingredient \cite{GQ87, GKO85}. Results of Kac and Wakimoto \cite{KW88pnas} were used subsequently by many authors to generalize the unitary GKO coset theories from integrable integer-levels to admissible fractional-levels, and to study non-unitary conformal field theories. Non-unitary conformal field theories are objects of recent research as no unified approach to study them has emerged so far, and as they appear in various contexts, for example, as scaling limits of the models of statistical physics \cite{C03}. Using the GKO coset realization Ahn, Chung, and Tye \cite{ACT}, defined the generalized FZ parafermionic theories, which appeared to be generically non-unitary and also non-rational theories as they have an infinite number of primary fields with all but a finite number of them having negative conformal dimensions. In this work we observe the non-rationality of these theories in the sense of characters and introduce the quasi-periodic relations for the admissible string functions in contrast to the periodic relations in the rational case. The FZ parafermionic theories are of special interest because by combining them with bosonic models one can obtain large classes of important conformal field theories \cite{ACT}.

Determining the modular properties and explicitly computating string functions or, in general, the branching functions is an important problem. Kac and Peterson \cite{KP84} derived the modular properties of string functions of the integrable highest weight representations on the whole modular group, and calculated them for certain cases in terms of theta functions. In \cite{KW88advmath} Kac and Wakimoto generalized these results and computed certain branching functions using methods of modular and conformal invariance. 

The fractional admissible string functions as discussed in our paper \cite{BoMo2024} are expected to be mixed mock modular forms for positive-level and mixed false theta functions for negative levels. In \cite{BoMo2024}, we not only deduced the mock modular transformation properties of the $1/2$-level string functions and computed the string functions in terms of second-order mock theta functions, but we also found compact formulas for general negative-level admissible string functions in terms of false theta functions. In order to carry out our calculations, we appealed to mock modularity \cite{Zw2} and general formulas which expand Hecke-type double-sums in terms of Appell functions \cite{HM,Mo24AA,MZ}, which are the building blocks of Ramanujan's mock theta functions.

Mock theta functions and false theta functions were famously observed by Ramanujan in the beginning of 20th century, in both his last letter to G. H. Hardy and the Lost Notebook.  Recently, mock theta functions and false theta functions have appeared in various contexts in representation theory and mathematical physics, for example in the relations to the characters of vertex operator algebras or $\mathcal{W}$-algebras \cite{BM15, BM17, BKM19, BKMN21, BKMN23}, supercharacters \cite{KW01, KW14, KW16adv, KW16izv}, Mathieu group $M_{24}$ and Umbral Moonshine \cite{CDH, EOT}, supercoset theories \cite{ES, ES14}, non-compact elliptic genus \cite{Troo}, Liouville theory \cite{ES16}, holomorphic anomaly in gauge theories \cite{DPW, KMMN, Man}, D3-instantons \cite{ABMP17,ABMP18}, and black holes \cite{AP20, DMZ}.   In \cite{DMZ}, polar-finite decompositions were originally developed. 

Our approach here will be the Zagier--Zwegers polar-finite decomposition.  Dabholkar, Murthy, and Zagier \cite{DMZ}  extended results of Zwegers \cite{Zw2} and introduced a canonical decomposition of a meromorphic Jacobi form into a ``finite'' part and a ``polar'' part, where the former is a finite linear combination of theta functions with mock modular forms as coefficients, and the latter is completely determined by the poles of the meromorphic Jacobi form. In this paper we present a new approach to Zagier--Zwegers polar-finite decomposition by introducing quasi-periodic relations, which allow us to extend Zagier--Zwegers' analysis to the case of admissible characters, which are vector-valued meromorphic Jacobi forms \cite{KW88advmath}. Using the polar-finite decomposition of admissible characters we then find mock theta conjecture-like identities for string functions of certain positive fractional-levels, thus extending our previous results \cite{BoMo2024}. 

Our plan for the rest of the introduction is as follows. In Section \ref{subsection:charStrings}, we define the admissible highest weight representations of the affine Kac--Moody algebra $A_1^{(1)}$ and introduce the characters and string functions of these representations \cite{KW88pnas}. In Section \ref{subsection:parafchar} we discuss how admissible string functions and parafermionic characters are related. In Section \ref{subsection:compintstrfunc} we recall some of Kac and Peterson's expressions for integrable string functions in terms of theta functions.

\subsection{Kac--Wakimoto admissible characters and string functions}\label{subsection:charStrings}
For $p \geq 1$, $p^{\prime} \geq 2$ coprime integers, we define the admissible-level as
\begin{equation} \label{equation:admlevel}
N:=\frac{p^{\prime}}{p}-2.
\end{equation}
We then denote by $L(\lambda)$ an admissible $A_{1}^{(1)}$ highest weight representation of highest weight 
\begin{equation}\label{equation:highestweight}
\lambda =  \lambda^{I} - (N+2)\lambda^{F},
\end{equation}
where $\lambda^{I}$ and $\lambda^{F}$ are two integrable weights of levels $p'-2$ and $p-1$ respectively, that is, for $0 \leq \ell \leq p'-2$ and $0 \leq k \leq p-1$ we have
\begin{align*}
\lambda^{I} &= (p'-\ell-2) \Lambda_0 + \ell \Lambda_1,\\
\lambda^{F} &= (p-k-1) \Lambda_0 + k \Lambda_1,
\end{align*}
where $\Lambda_0$ and $\Lambda_1$ are the fundamental weights of $A_{1}^{(1)}$. In this sequel paper we will continue to consider only the case of $k=0$, so that the spin, the coefficient of $\Lambda_1$ in \eqref{equation:highestweight}, is equal to $\ell$ and hence is a positive integer. Again we remind the reader that when $p= 1$, admissible representations reduce to integrable ones \cite{KP84}, that is, the fractional part vanishes $\lambda^{F} = 0$.

Let $q := e^{2\pi i \tau}$ with $\im(\tau) >0$ and $z \in \mathbb{C}\backslash \{0\}$. The character for irreducible highest weight representation of admissible highest weight \eqref{equation:highestweight} is
\begin{equation} \label{eq:chdef}
\chi_{\ell}^N(z;q):=\textup{Tr}_{L(\lambda)} \left(q^{s_{\lambda}-d}z^{-\frac{1}{2}J^{0} }\right),
\end{equation}
where $d$ is a derivation, $J^{0}$ is the generator of the Cartan subalgebra of $A_1$ in the spin basis and 
\begin{equation*}
s_{\lambda}:=-\frac{1}{8}+\frac{(\ell+1)^2}{4(N+2)}.
\end{equation*}
By the Weyl--Kac formula it is possible to express the character as
\begin{equation}\label{equation:WK-formula}
\chi_{\ell}^N(z;q)=\frac{\sum_{\sigma=\pm 1}\sigma \Theta_{\sigma (\ell+1),p^{\prime}}(z;q^{p})}
{\sum_{\sigma=\pm 1}\sigma\Theta_{\sigma,2}(z;q)},
\end{equation}
where we denote the theta function as
\begin{equation}\label{equation:SW-thetaDef}
\Theta_{n,m}(z;q):=\sum_{j\in\mathbb{Z}+n/2m}q^{mj^2}z^{-mj}.
\end{equation}
Using \eqref{equation:WK-formula} Kac and Wakimoto showed that the characters form a vector-valued Jacobi form.

Let us denote the energy eigenspaces with respect to $-d$ as
\begin{equation*}
L(\lambda)_{(n)} := \{v \in L(\lambda) \ | \ (-d) v = n v \  \}
\end{equation*}
 for $n \in \Z_{\geq 0}$ and the weight space associated to the weight 
 \begin{equation} \label{eq:arbweight}
 \mu = (N-m)\Lambda_0+m\Lambda_1 
 \end{equation}
 as
\begin{equation*}
L(\lambda)_{[m]} := \{v \in L(\lambda) \ | \ J^{0} v = m v \  \}.
\end{equation*}

For admissible highest weight \eqref{equation:highestweight} and arbitrary weight \eqref{eq:arbweight} we define the string function as
\begin{equation*}
c^{\lambda}_{\mu} = c^{N-\ell,\ell}_{N-m,m} = C^{N}_{m,\ell}(q) := q^{s_{\lambda,\mu}} \sum_{n\geq 0} \dim (L(\lambda)_{[m]} \cap L(\lambda)_{(n)}) q^n,
\end{equation*}
where
\begin{equation*}
s_{\lambda,\mu} := s_{\lambda} - \frac{m^2}{4N}.   
\end{equation*}
We also define the additional notation
\begin{equation}
\mathcal{C}_{m,\ell}^{N}(q) := q^{-s_{\lambda,\mu}}C_{m,\ell}^{N}(q) \in \Z[[q]].\label{equation:mathCalCtoStringC}
\end{equation} 
From the definition we have the Fourier expansion
\begin{equation} \label{equation:fourcoefexp}
\chi_{\ell}^N (z,q)=\sum_{m\in 2\mathbb{Z}+\ell}
C_{m,\ell}^{N}(q) q^{\frac{m^2}{4N}}z^{-\frac{1}{2}m}.
\end{equation} 
We have the following symmetry of string functions \cite[(3.4), (3.5)]{SW}, \cite[(2.40)]{ACT}
\begin{align*}
C_{m,\ell}^{N}(q) = C_{-m,\ell}^{N}(q),
\ \ C_{m,\ell}^{N}(q) = C_{N-m,N-\ell}^{N}(q).
\end{align*}
For the integral level $N$ we have the periodicity property \cite[(3.5)]{SW}
\begin{equation} \label{eq:inglevelperiod}
C_{m,\ell}^{N}(q) = C_{m+2N,\ell}^{N}(q),
\end{equation}
and hence from \eqref{equation:fourcoefexp} the theta-expansion
\begin{equation}\label{eq:intlevelthetadecomp}
\chi_{\ell}^N(z,q)=\sum_{\substack{0\le m <2N\\m \in 2\Z + \ell}}C_{m,l}^{N}(q)\Theta_{m,N}(z,q).
\end{equation}

\subsection{Characters of the generalized Fateev--Zamolodchikov parafermionic theories}\label{subsection:parafchar}
We introduce the generalized FZ parafermionic theories as the GKO coset theory constructed from $\operatorname{SL}(2)_N$ Wess--Zumino--Witten theories 
\begin{equation*}
    \operatorname{Z}_N = \frac{\SL(2)_N}{\U(1)},
\end{equation*}
where $N$ is the admissible-level \eqref{equation:admlevel}. This theory has central charge $c = 2(N-1)/(N+2)$. For integrable-level $N$ this is the GKO coset realization of classical unitary FZ parafermionic theories \cite{FZ85}; for $N = 2, 3$ it describes respectively the Ising model and the three-state Potts model. 

The $\operatorname{SL}(2)_N$ Hilbert spaces $\mathcal{H}_{N,\ell}$ of spin $\ell$ can be decomposed into Hilbert spaces of states of fixed $J^{0}$ quantum number $m$ as
\begin{equation} \label{equation:SLintoVir}
\mathcal{H}_{N,\ell} = \bigoplus_{m \in 2\Z + \ell} \mathcal{H}_{N,\ell,m}.
\end{equation}
Ahn, Chung, and Tye \cite{ACT} factored $\mathcal{H}_{N,\ell,m}$ into the FZ parafermionic Hilbert space and that of the boson,
\begin{equation} \label{equation:VirIntoPFandB}
\mathcal{H}_{N,\ell,m} = \mathcal{H}_{N,\ell,m}^{\operatorname{PF}} \otimes  \mathcal{H}_{N,m}^{\operatorname{b}}.
\end{equation}
Let us denote the parafermionic character by $e_{m,\ell}^{N}(q)$,  from \eqref{equation:SLintoVir} and \eqref{equation:VirIntoPFandB} we get
\begin{equation} \label{equation:string-def}
\chi_{\ell}^N (z,q)=\sum_{m\in 2\mathbb{Z}+\ell}
e_{m,\ell}^{N}(q) \cdot \frac{q^{\frac{m^2}{4N}}z^{-\frac{1}{2}m}}{\eta(q)},
\end{equation} 
where we denote the Dedekind eta-function as
\begin{equation*} 
\eta(q) :=q^{1/24}\prod_{n\ge 1}(1-q^{n}).
\end{equation*}
From \eqref{equation:fourcoefexp} we see that
\begin{equation*}
e_{m,\ell}^{N}(q) = \eta(q) C_{m,\ell}^{N}(q).   
\end{equation*}

\subsection{Computation of the integrable string functions}\label{subsection:compintstrfunc}

An important problem in the representation theory of the affine Kac--Moody algebras is the explicit calculation of the string functions. Kac and Peterson \cite{KP84} gave several examples of evaluations of string functions of integrable highest weight representations of $A_1^{(1)}$ in terms of theta functions. In order to state their results, let us introduce our notation for theta functions. We recall the $q$-Pochhammer notation
\begin{equation*}
(x)_n=(x;q)_n:=\prod_{i=0}^{n-1}(1-q^ix), \ \ (x)_{\infty}=(x;q)_{\infty}:=\prod_{i\ge 0}(1-q^ix),
\end{equation*}
and the theta function
\begin{equation}\label{equation:JTPid}
j(x;q):=(x)_{\infty}(q/x)_{\infty}(q)_{\infty}=\sum_{n=-\infty}^{\infty}(-1)^nq^{\binom{n}{2}}x^n,
\end{equation}
where the last equality is the Jacobi triple product identity.  We will frequently use the notation,
\begin{equation} \label{equation:shortnottheta}
J_{a,b}:=j(q^a;q^b), \ \mathcal{J}_{a,b}:=q^{\frac{(b-2a)^2}{8b}}J_{a,b},
 \ \overline{J}_{a,b}:=j(-q^a;q^b), \ {\text{and }}J_a:=J_{a,3a}=\prod_{i\ge 1}(1-q^{ai}),
\end{equation}
where $a,b$ are positive integers.  One notes that the two definitions for theta functions are equivalent via the identity
\begin{equation}
\Theta_{n,m}(z;q)=z^{-\frac{n}{2}}q^{\frac{n^2}{4m}}j\left ( -q^{n+m}z^{-m};q^{2m}\right). \label{equation:Theta-to-j}
\end{equation}
Among more general results, Kac and Peterson \cite{KP84} showed
{\allowdisplaybreaks \begin{gather*}
c^{01}_{01} = \eta(q)^{-1},\\
c^{11}_{11} = \eta(q)^{-2}\eta(q^2),\\
c^{21}_{21} = \eta(q)^{-2} q^{3/40} J_{6,15},\\
c^{40}_{22} = \eta(q)^{-2} \eta(q^6) \eta(q^{12})^{2},\\
c^{40}_{40} - c^{40}_{04} = \eta(q^2)^{-2}.
\end{gather*}}%
Kac and Peterson appealed to modularity to prove such string function identities \cite[p. 220]{KP84}.  In \cite[Example 1.3]{HM}, we find a closed expression in terms of theta functions for the general integral-level string function.  Additional calculations on integrable string functions can be found in \cite{Mo24B, MPS}.

\section{New results}

In this section we present our new results on admissible characters and admissible string functions. In Section \ref{subsection:quasiPeriods}, we derive quasi-periodicity properties for admissible fractional-level string functions.  Quasi-periodicity for fractional-level string functions will contrast with the periodicity properties for integral-level string functions.  In Section \ref{subsection:polarFinite}, the defining Fourier expansion for string functions coupled with our notion of quasi-periodicity gives us a polar-finite decomposition of characters after Zagier et al. \cite{DMZ}  and Zwegers \cite{Zw2}.  The first application of which is to give a new proof of the mock theta conjecture-like identities of our paper \cite{BoMo2024}.  In Section  \ref{subsection:mockThetaIdentities}, we use our new polar-finite decomposition to express select characters in terms of mock theta conjecture-like identities for Ramanujan's third-order and tenth-order mock theta functions.  In Section \ref{subsection:genExpansion}, we give a heuristic argument to demonstrate that in general, that the positive admissible-level  string functions are similar in form to the polar-finite decompositions for characters.  We also draw comparisons with our formula for general negative-level admissible string functions in terms of false theta functions from our paper \cite{BoMo2024}.

\subsection{Quasi-periodic relations and cross-spin identities} \label{subsection:quasiPeriods}

Among the first of our results is a quasi-periodic relation for even-spin.  For positive admissible-level string functions we have the following quasi-periodicity property in contrast to the periodicity property \eqref{eq:inglevelperiod} for integral-level string functions.

\begin{theorem}\label{theorem:generalQuasiPeriodicity} For $(p,p^{\prime})=(p,2p+j)$, we have the quasi-periodic relation for even-spin 
\begin{align*}
& (q)_{\infty}^{3}C_{2jt+2s,2r}^{(p,2p+j)}(q)
 -(q)_{\infty}^{3}C_{2s,2r}^{(p,2p+j)}(q)\\
& \  = (-1)^{p}q^{-\frac{1}{8}+\frac{p(2r+1)^2}{4(2p+j)}}q^{\binom{p}{2}-p(r-s)-\frac{p}{j}s^2}\sum_{i=1}^{t}q^{-2pj\binom{i}{2}-2psi} \\
&\quad  \times 
\sum_{m=1}^{p-1}(-1)^{m}q^{\binom{m+1}{2}+m(r-p)}
 \left ( q^{m(ji+s-j)}-q^{-m(ji+s)}\right )\\
&\quad   \times 
\Big (  j(-q^{m(2p+j)+p(2r+1)};q^{2p(2p+j)} )
  -  q^{m(2p+j)-m(2r+1)}j(-q^{-m(2p+j)+p(2r+1)};q^{2p(2p+j)})\Big ).
\end{align*}
\end{theorem}

In some situations, we have cross-spin identities in that we can relate odd-spin to even-spin.
\begin{theorem}\label{theorem:crossSpin-j-Odd}  For $(p,p^{\prime})=(p,2p+j)$, $j$ odd, we have the cross-spin identity for odd spin
\begin{align*}
(q)_{\infty}^3&\mathcal{C}_{2i-1,2r-1}^{(p,2p+j)}(q)\\
&=(-1)^{p+1}q^{p(i-r)+\binom{p}{2}}(q)_{\infty}^{3}\mathcal{C}_{2i-1-j,2p-2r+j-1}^{(p,2p+j)}(q)\\
&\qquad +(-1)^{p}q^{\binom{p}{2}+p(i+r)}
\sum_{m=1}^{p-1}(-1)^{m}q^{\binom{m+1}{2}-m(i+p+r)}\\
&\qquad \qquad \times
\left ( j(-q^{m(2p+j)-2pr};q^{2p(2p+j)})
-q^{2r(m-p)}j(-q^{m(2p+j)+2pr};q^{2p(2p+j)})\right ) .
\end{align*}
\end{theorem}

\subsection{A polar-finite decomposition and our first application}\label{subsection:polarFinite}

The Fourier expansion (\ref{equation:fourcoefexp}) coupled with our notion of quasi-periodicity in Theorem \ref{theorem:generalQuasiPeriodicity} then gives us a polar-finite decomposition of characters after Zagier et al. \cite{DMZ} and Zwegers \cite{Zw2}.  We point out that Zagier et al. carried out their polar-finite decomposition analysis for non-vector-valued meromorphic Jacobi forms.   In our work, we consider admissible characters, which are vector-valued Jacobi forms. Whereas Zagier et al.~ used Cauchy's theorem and considered residues,  we carry out our polar-finite decomposition analysis in a different way: we employ quasi-periodicity properties.   The polar-finite decomposition is written in terms of Appell-functions.

Following Hickerson and Mortenson \cite{HM} we define the Appell function as
\begin{equation}
m(x,z;q):=\frac{1}{j(z;q)}\sum_{r\in\Z}\frac{(-1)^rq^{\binom{r}{2}}z^r}{1-q^{r-1}xz}.\label{equation:m-def}
\end{equation}

\noindent Ramanujan's classical mock theta functions can all be expressed in terms of Appell functions \cite[Section 5]{HM}.  As an example, Ramanujan's classical second-order mock theta functions read
\begin{gather}
A_2(q)
:=\sum_{n\ge 0}\frac{q^{n+1}(-q^2;q^2)_n}{(q;q^2)_{n+1}}
=\sum_{n\ge 0}\frac{q^{(n+1)^2}(-q;q^2)_n}{(q;q^2)_{n+1}^2}
=-m(q,q^2;q^4),\label{equation:2nd-A(q)}\\
\mu_2(q)
:=\sum_{n\ge 0}\frac{(-1)^nq^{n^2}(q;q^2)_n}{(-q^2;q^2)_{n}^2}
=2m(-q,-1;q^4)+2m(-q,q;q^4),
\label{equation:2nd-mu(q)}
\end{gather}
both of which appearing in the ``Lost Notebook''  \cite[p. 8]{RLN}.

We will use our quasi-periodicity relation Theorem \ref{theorem:generalQuasiPeriodicity} to prove a polar-finite decomposition for characters of even-spin.

\begin{theorem}\label{theorem:generalPolarFinite} For $(p,p^{\prime})=(p,2p+j)$, we have the polar-finite decomposition for even-spin 
\begin{align*}
&\chi_{2r}^{(p,2p+j)} (z;q)\\
&\ \ =\sum_{s=0}^{j-1}z^{-s}q^{\frac{p}{j}s^2}C_{2s,2r}^{(p,2p+j)}(q)j(-z^{j}q^{p(j-2s)};q^{2pj})\\
&\ \ \ \  +\frac{1}{(q)_{\infty}^3}\sum_{s=0}^{j-1}
(-1)^{p}q^{-\frac{1}{8}+\frac{p(2r+1)^2}{4(2p+j)}}q^{\binom{p}{2}-p(r-s)}z^{-s}
j(-q^{p(j-2s)}z^{j};q^{2jp})
 \sum_{m=1}^{p-1}(-1)^{m}q^{\binom{m+1}{2}+m(r-p)} \\
&\qquad      \times\Big (  j(-q^{m(2p+j)+p(2r+1)};q^{2p(2p+j)} )
 -  q^{m(2p+j)-m(2r+1)}j(-q^{-m(2p+j)+p(2r+1)};q^{2p(2p+j)})\Big )\\
&\qquad      \times
\Big (q^{ms-2ps}m(-q^{jm-2ps},-q^{p(j+2s)}z^{-j};q^{2jp}) 
 + q^{-ms}m(-q^{jm+2ps},-q^{p(j-2s)}z^{j};q^{2jp})\Big ).
\end{align*}
\end{theorem}
\begin{remark}  In a sense, Theorems \ref{theorem:generalQuasiPeriodicity} and \ref{theorem:generalPolarFinite} are already in their most general forms, for we can easily replace $j$ with $(n-2)p+j$.

\end{remark}
For $j=1$ we have the following immediate corollary.
\begin{corollary}\label{corollary:polarFinite1p} For $(p,p^{\prime})=(p,2p+1)$, $(m,\ell)=(2k,2r)$, we have
\begin{align*}
\chi_{2r}^{(p,2p+1)}&(z;q)\\
&=C_{0,2r}^{(p,2p+1)}(q)
j(-q^{p}z;q^{2p})
 \\
&\qquad +(-1)^{p}q^{-\frac{1}{8}+\frac{p(2r+1)^2}{4(2p+1)}+\binom{p}{2}-rp}\frac{j(-q^{p}z;q^{2p})}{(q)_{\infty}^3}
\sum_{m=1}^{p-1} 
(-1)^{m}q^{\binom{m+1}{2}+m(r-p)}
 \\
&\qquad \qquad \times \left ( j(-q^{m(2p+1)+p(2r+1)};q^{2p(2p+1)}) 
-q^{2m(p-r)}j(-q^{-m(2p+1)+p(2r+1)};q^{2p(2p+1)})\right )
 \\
&\qquad \qquad \qquad \times 
 \left ( m(-q^{m},-q^{p}z;q^{2p})
+m(-q^{m},-q^{p}z^{-1};q^{2p})\right ).
\end{align*}
\end{corollary}

Our first application of Theorem \ref{theorem:generalPolarFinite} will be to give new proofs of results found in \cite[Corollaries 2.5, 2.6]{BoMo2024}.  We recall the definition for Hecke-type double-sums.  Let $x,y\in\mathbb{C} \backslash \{0\}$. Then
\begin{equation} \label{equation:fabc-def2}
f_{a,b,c}(x,y;q):=\left( \sum_{r,s\ge 0 }-\sum_{r,s<0}\right)(-1)^{r+s}x^ry^sq^{a\binom{r}{2}+brs+c\binom{s}{2}},
\end{equation}
where we define the discriminant to be
\begin{equation*}
D:=b^2-ac.
\end{equation*}
Fractional-level string functions can be expressed in terms of Hecke-type double-sums.  Let $p'\geq 2$, $p\geq 1$ be coprime integers, $0\leq \ell \leq p'-2$ and $m\in 2\Z +\ell$, then \cite[Proposition 1.2]{BoMo2024}, \cite[(3.8)]{SW} gives
\begin{equation}
(q)_{\infty}^3\mathcal{C}_{m,\ell}^{N}(q)
 = f_{1,p^{\prime},2pp^{\prime}}(q^{1+\frac{m+\ell}{2}},-q^{p(p^{\prime}+\ell+1)};q)
 -f_{1,p^{\prime},2pp^{\prime}}(q^{\frac{m-\ell}{2}},-q^{p(p^{\prime}-(\ell+1))};q).
 \label{equation:modStringFnHeckeForm}
\end{equation}      

In \cite{HM,MZ}, Hecke-type double-sums \eqref{equation:fabc-def2} with positive discriminant $D$ are extensively studied.  Expansions are obtained that express the double-sums in terms of theta and Appell functions, \cite[Theorems 1.3, 1.4]{HM}, \cite[Corollary 4.2]{MZ}.  In our work \cite{BoMo2024}, we used our newly-found modular transformation properties for $1/2$-level string functions to relate the asymmetric Hecke-type double-sums in (\ref{equation:modStringFnHeckeForm}) for $1/2$-level string functions to the more symmetric double-sums found in \cite[Theorem 1.4]{HM}, which one can easily evaluate in terms of mock theta functions.

For the $1/2$-level string functions with quantum number $m=0$ and even-spin we derived the mixed mock modular transformation properties on the whole modular group \cite[Theorem 2.3]{BoMo2024}: 
\begin{equation*}
\begin{pmatrix}
C^{1/2}_{0,0}\\
C^{1/2}_{0,2}
\end{pmatrix}(\tau+1) = \begin{pmatrix}
\zeta_{40}^{-1} & 0\\
0 & \zeta_{40}^{-9} \\
\end{pmatrix}
\begin{pmatrix}
C^{1/2}_{0,0}\\
C^{1/2}_{0,2}
\end{pmatrix}(\tau)    
\end{equation*}
and
\begin{multline*}
\begin{pmatrix}
C^{1/2}_{0,0}\\
C^{1/2}_{0,2}
\end{pmatrix}(\tau) = \sqrt{-i\tau} \cdot \frac{2}{\sqrt{5}}  \begin{pmatrix}
\sin\left(\frac{2\pi}{5}\right) & -\sin\left(\frac{\pi}{5}\right)  \\
-\sin\left(\frac{\pi}{5}\right) & -\sin\left(\frac{2\pi}{5}\right)
\end{pmatrix} \begin{pmatrix}
  C^{1/2}_{0,0}\\
  C^{1/2}_{0,2}
\end{pmatrix}\left(-\frac{1}{\tau}\right) \\
- \frac{i}{2} \cdot \frac{1}{\eta(\tau)^3}\cdot \begin{pmatrix}
\mathcal{J}_{1,5}\\
\mathcal{J}_{2,5}
\end{pmatrix}(\tau) \cdot \int_{0}^{i\infty} \frac{\eta(z)^3}{\sqrt{-i(z+\tau)}}dz.
\end{multline*}
Using the newly-found modular transformation properties we related the asymmetric Hecke-type double-sums of (\ref{equation:modStringFnHeckeForm}) for $1/2$-level string functions to the more symmetric double-sums found in \cite[Theorem 1.4]{HM}, which one can easily evaluate in terms of mock theta functions.  We found  \cite[Corollaries 2.5, 2.6]{BoMo2024}:  For $(p,p^{\prime})=(2,5)$, $(m,\ell)=(0,2r)$, $r\in\{0,1\}$, we have
    \begin{gather}
(q)_{\infty}^3\mathcal{C}_{0,2r}^{1/2}(q)
=(-1)^{r}
\frac{J_{1}^4J_{4}}{J_{2}^4}j(q^{4r+12};q^{20})
 -2q^{-r} j(q^{1+2r};q^{5})  A(-q),
 \label{equation:mockThetaConj2502r-2ndA}\\
(q)_{\infty}^3\mathcal{C}_{0,2r}^{1/2}(q)
=(-q)^{-r}\frac{1}{2}
\frac{J_{1}^3}{J_{2}J_{4}}j(-q^{2r+1};-q^{5})
+ q^{-r} \frac{1}{2}  j(q^{1+2r};q^{5}) \mu(q).
 \label{equation:mockThetaConj2502r-2ndmu}
\end{gather}

Not only does our new polar-finite decomposition Theorem \ref{theorem:generalPolarFinite} give new proofs of identities  (\ref{equation:mockThetaConj2502r-2ndA}) and  (\ref{equation:mockThetaConj2502r-2ndmu}), but we can also obtain new string function identities in terms of other mock theta functions.


\subsection{Mock theta functions and new mock theta conjecture-like identities}
\label{subsection:mockThetaIdentities}
In Ramanujan's last letter to Hardy, he gave a list of seventeen so-called mock theta functions.   Each function was defined by Ramanujan as a $q$-series convergent for $|q|<1$.  Although the $q$-series are not theta functions, they do have certain asymptotic properties similar to those of ordinary theta functions.  In the letter, one finds four `3rd' order mock theta functions, ten `5th' order functions, and three `7th' order functions, as well as several identities relating the mock theta functions to each other. 

Decades later, additional mock theta functions and mock theta function identities were found in the lost notebook \cite{RLN}.   In particular, we find the mock theta conjectures, which are ten identities where each identity expresses a different 5th order mock theta function in terms of a building block and a single quotient of theta functions.  This particular building block is the so-called universal mock theta function $g(x;q)$, which is defined
\begin{equation*}
g(x;q):=x^{-1}\Big ( -1 +\sum_{n=0}^{\infty}\frac{q^{n^2}}{(x)_{n+1}(q/x)_{n}} \Big ).
\end{equation*}
Using our notation, two of the ten mock theta conjectures read \cite{AG, H1}
\begin{align*}
f_0(q)&:=\sum_{n= 0}^{\infty}\frac{q^{n^2}}{(-q)_n}=\frac{J_{5,10}J_{2,5}}{J_1}-2q^2g(q^2;q^{10}),\\
f_1(q)&:=\sum_{n= 0}^{\infty}\frac{q^{n(n+1)}}{(-q)_n}=\frac{J_{5,10}J_{1,5}}{J_1}-2q^3g(q^4;q^{10}).
\end{align*}

The universal mock theta function $g(x;q)$ is only related to the odd-ordered mock theta functions.  Appell functions are the building blocks for both even and odd-ordered mock theta functions \cite[Section 5]{HM}, \cite{Zw2}.  In particular, the universal mock theta function can be expressed in terms of Appell functions \cite[Proposition $4.2$]{HM}, \cite[Theorem $2.2$]{H1}:
\begin{equation*}
g(x;q)=-x^{-1}m(q^2x^{-3},x^2;q^3)-x^{-2}m(qx^{-3},x^2;q^3).
\end{equation*}


We can use our polar-finite decomposition Theorem \ref{theorem:generalPolarFinite} to relate positive admissible-level string functions to other mock theta functions.  To state another new result, we recall two of the third-order mock theta functions found in Ramanujan's last letter to Hardy:
\begin{gather*}
f_3(q):=\sum_{n\ge 0}\frac{q^{n^2}}{(-q)_n^2}, \ \ 
\omega_3(q)
:=\sum_{n\ge 0}\frac{q^{2n(n+1)}}{(q;q^2)_{n+1}^2}.
\end{gather*}
For asymptotic analysis of $f_{3}(q)$ see \cite{And66, BrOno06}, and for the modularity of $f_{3}(q)$ and $\omega_{3}(q)$ see \cite{Zw1}.

For the $1/3$-level string functions we have the result:
\begin{theorem}\label{theorem:newMockThetaIdentitiespP37m0ell2r}
For $(p,p^{\prime})=(3,7)$, $(m,\ell)=(0,2r)$, $r\in \{0,1,2\}$ we have
\begin{align*}
(q)_{\infty}^3&\mathcal{C}_{0,2r}^{1/3}(q)
=(-q)^{-r}\frac{(q)_{\infty}^3}{J_{2}}
\frac{j(-q^{1+2r};q^{14})j(q^{16+4r};q^{28})}
{j(-1;q)J_{28}} \\
&\qquad   -q^{2-2r}\frac{j(q^{6-2r};q^{14})j(q^{26-4r};q^{28})}{J_{28}} \omega_3(-q)
+  \frac{q^{-r}}{2} \frac{j(q^{1+2r};q^{14})j(q^{16+4r};q^{28})}{J_{28}} 
f_{3}(q^2).
\end{align*}
\end{theorem}

To relate odd-spin $1/3$-level string functions to even-spin $1/3$-level string functions, one could use the following specialization of Theorem \ref{theorem:crossSpin-j-Odd}
\begin{align}
(q)_{\infty}^3\mathcal{C}_{2i-1,2r-1}^{(3,7)}(q)
&=q^{3(1+i-r)}(q)_{\infty}^{3}\mathcal{C}_{2i-2,6-2r}^{(3,7)}(q)
\label{equation:crossSpin37}\\
&\qquad  -q^{1+2(i-r)}
\frac{j(q^{7+2r};q^{14})j(q^{4r};q^{28})}{J_{28}}
+q^{i-r}
\frac{j(q^{2r};q^{14})j(q^{14+4r};q^{28})}{J_{28}}.
\notag
\end{align}

For the $2/3$-level string functions we have two results.  We recall the floor function $\lfloor \cdot \rfloor$.  We have
\begin{theorem}\label{theorem:newMockThetaIdentitiespP38m0ell2r} For $(p,p^{\prime})=(3,8)$, $(m,\ell)=(0,2r)$, $r\in\{0,1,2,3\}$, we have
\begin{align*}
(q)_{\infty}^3&\mathcal{C}_{0,2r}^{2/3}(q) 
=(-1)^{\lfloor (r+1)/2\rfloor}\cdot
\frac{q^{-r}}{2}\frac{J_{1}^2J_{2}}{J_{4}^2J_{8}}
j(-q^{7-2r};q^{16})j(q^{1+2r};q^{8})\\
&\qquad 
- q^{3-2r}    
\frac{j(q^{7-2r};q^{16})j(q^{30-4r};q^{32})}{J_{32}}
\omega_{3}(-q^{2}) 
 + q^{-r} 
\frac{j(q^{1+2r};q^{16})j(q^{18+4r};q^{32})}{J_{32}} 
 \frac{1}{2}f_{3}(q^4).
 \end{align*}
\end{theorem}

\begin{theorem}\label{theorem:newMockThetaIdentitiespP38m2ell2r} For $(p,p^{\prime})=(3,8)$, $(m,\ell)=(2,2r)$, $r\in\{0,1,2,3\}$, we have
\begin{align*}
(q)_{\infty}^3\mathcal{C}_{2,2r}^{2/3}(q) 
&=(-1)^{\lfloor (r+1)/2\rfloor}\cdot 
\frac{q^{3-2r}}{2}\frac{J_{1}^2J_{2}}{J_{4}^2J_{32}}j(q^{2+4r};q^{32})j(q^{7-2r};q^{16})\\
&\qquad 
- q^{3-2r}   
\frac{j(q^{7-2r};q^{16})j(q^{30-4r};q^{32})}{J_{32}}
\left (1- \frac{1}{2}f_{3}(q^4) \right )  \\
&\qquad + q^{1-r} 
\frac{j(q^{1+2r};q^{16})j(q^{18+4r};q^{32})}{J_{32}} 
 \left (1-q^2\omega_{3}(-q^{2}) \right ).
\end{align*}
\end{theorem}

Of course, if one is familiar with the list of Appell function forms for the classical mock theta functions as found in \cite[Section 5]{HM}, and one has a feel for the polar-finite decomposition Theorem \ref{theorem:generalPolarFinite}, one can make good guesses as to where which mock theta functions appear and for which string functions.  However, can we continue the theme of a single-quotient of theta functions?  This leads us to the four tenth-order mock theta functions.

Also in the lost notebook, we find the four tenth-order mock theta functions and their six identities \cite{RLN}.  The six identities  were first proved by Choi \cite{C1, C2, C3} using methods similar to those of Hickerson \cite{H1, H2}.  We recall the four tenth-order mock theta functions \cite{C1, C2, C3, RLN}
\begin{align*}
{\phi}_{10}(q)&:=\sum_{n\ge 0}\frac{q^{\binom{n+1}{2}}}{(q;q^2)_{n+1}}, \ \ {\psi}_{10}(q):=\sum_{n\ge 0}\frac{q^{\binom{n+2}{2}}}{(q;q^2)_{n+1}}, \\
& \ \ \ \ \ {X}_{10}(q):=\sum_{n\ge 0}\frac{(-1)^nq^{n^2}}{(-q;q)_{2n}}, \ \  {\chi}_{10}(q):=\sum_{n\ge 0}\frac{(-1)^nq^{(n+1)^2}}{(-q;q)_{2n+1}}.
\end{align*}
The four functions satisfy six identities, each of which involves a single quotient of theta functions.  Letting $\omega$ be a primitive third-root of unity.  Two of the six identities read \cite{C1}
\begin{gather*}
q^{2}\phi_{10}(q^9)-\frac{\psi_{10}(\omega q)-\psi_{10}(\omega^2 q)}{\omega - \omega^2}
=-q\frac{J_{1,2}}{J_{3,6}}\cdot \frac{J_{3,15}J_{6}}{J_{3}},\\
q^{-2}\psi_{10}(q^9)+\frac{\omega \phi_{10}(\omega q)-\omega^2\phi_{10}(\omega^2 q)}{\omega - \omega^2}
=\frac{J_{1,2}}{J_{3,6}}\cdot \frac{J_{6,15}J_{6}}{J_{3}},
\end{gather*}
with the other four being similar \cite{C2, C3}.  Zwegers \cite{Zw3} has since given short proof of four of the identities.  Mortenson \cite{Mo2018} later gave short proofs of all six identities.  For more on single-quotient identities involving mock theta functions and Appell functions see \cite{Mo24C, MU24}.

For the $1/5$-level string functions, we have the result:
\begin{theorem} \label{theorem:newMockThetaIdentitiespP511m0ell2r}  For $(p,p^{\prime})=(5,11)$, $(m,\ell)=(0,2r)$, $r\in\{0,1,2,3,4\}$, we have
\begin{align*}
(q)_{\infty}^3\mathcal{C}_{0,2r}^{1/5}(q)
&=-q^{r^2-3r+1}J_{1,2}j(q^{4+8r};q^{22})
 \\
&\qquad -
q^{6-4r} \times \left ( j(-q^{16+10r};q^{110}) 
-q^{4+8r}j(-q^{6-10r};q^{110})\right )
 \times q\phi_{10}(-q) \\
& \qquad +q^{3-3r}\times \left ( j(-q^{27+10r};q^{110}) 
-q^{3+6r}j(-q^{17-10r};q^{110})\right )
\times \chi_{10}(q^2) \\
&\qquad 
-q^{1-2r}  \times \left ( j(-q^{38+10r};q^{110}) 
-q^{2+4r}j(-q^{28-10r};q^{110})\right )
\times \left ( -\psi_{10}(-q)\right ) \\
&\qquad  
+q^{-r} \times \left ( j(-q^{49+10r};q^{110}) 
-q^{1+2r}j(-q^{39-10r};q^{110})\right )\times X_{10}(q^2).
\end{align*}
\end{theorem}

\subsection{Heuristic calculations on the general positive-level admissible string functions}\label{subsection:genExpansion}
Earlier in the discussion we mentioned the results of Hickerson and Mortenson \cite{HM} as well as Mortenson and Zwegers \cite[Corollary 4.2]{MZ}, where Hecke-type double-sums are expanded in terms of Appell functions and theta functions.  In our work \cite{BoMo2024}, we used the newly-found modular transformation properties for $1/2$-level string functions to relate the asymmetric Hecke-type double-sums in (\ref{equation:modStringFnHeckeForm}) for $1/2$-level string functions to the more symmetric double-sums found in \cite[Theorem 1.4]{HM}, which one can easily evaluate in terms of mock theta functions.  Our new results such as Theorem \ref{theorem:generalPolarFinite} direct us towards a starting point for the general case.  This brings us to the most general double-sum expansion which is found in \cite[Corollary 4.2]{MZ}.

We first present a compact notation for a special family of Appell functions:  Let $a,b,$ and $c$ be positive integers with $D:=b^2-ac>0$.  Then
\begin{align} 
m_{a,b,c}(x,y,z_1,z_0;q)
&:=\sum_{t=0}^{a-1}(-y)^tq^{c\binom{t}{2}}j(q^{bt}x;q^a)m\Big (-q^{a\binom{b+1}{2}-c\binom{a+1}{2}-tD}\frac{(-y)^a}{(-x)^b},z_0;q^{aD}\Big ) \label{equation:mabc-def}\\
&\ \ \ \ \ +\sum_{t=0}^{c-1}(-x)^tq^{a\binom{t}{2}}j(q^{bt}y;q^c)m\Big ({-}q^{c\binom{b+1}{2}-a\binom{c+1}{2}-tD}\frac{(-x)^c}{(-y)^b},z_1;q^{cD}\Big ).\notag
\end{align}

We have the following decomposition for the general form (\ref{equation:fabc-def2}) with positive discriminant $D$.  

\begin{theorem}\cite[Corollary 4.2]{MZ}\label{theorem:posDisc} Let $a,b,$ and $c$ be positive integers with $D:=b^2-ac>0$. For generic $x$ and $y$, we have
\begin{align*}
& f_{a,b,c}(x,y;q)=m_{a,b,c}(x,y,-1,-1;q)+\frac{1}{j(-1;q^{aD})j(-1;q^{cD})}\cdot \vartheta_{a,b,c}(x,y;q),
\end{align*}
where
\begin{align*}
&\vartheta_{a,b,c}(x,y;q):=
\sum_{d^*=0}^{b-1}\sum_{e^*=0}^{b-1}q^{a\binom{d-c/2}{2}+b( d-c/2 ) (e+a/2  )+c\binom{e+a/2}{2}}(-x)^{d-c/2}(-y)^{e+a/2}\\
&\cdot\sum_{f=0}^{b-1}q^{ab^2\binom{f}{2}+\big (a(bd+b^2+ce)-ac(b+1)/2 \big )f} (-y)^{af}
\cdot j(-q^{c\big ( ad+be+a(b-1)/2+abf \big )}(-x)^{c};q^{cb^2})\\
&\cdot j(-q^{a\big ( (d+b(b+1)/2+bf)(b^2-ac) +c(a-b)/2\big )}(-x)^{-ac}(-y)^{ab};q^{ab^2D})\\
&\cdot \frac{(q^{bD};q^{bD})_{\infty}^3j(q^{ D(d+e)+ac-b(a+c)/2}(-x)^{b-c}(-y)^{b-a};q^{bD})}
{j(q^{De+a(c-b)/2}(-x)^b(-y)^{-a};q^{bD})j(q^{Dd+c(a-b)/2}(-y)^b(-x)^{-c};q^{bD})}.
\end{align*}
Here $d:=d^*+\{c/2 \}$ and $e:=e^*+\{ a/2\}$, with  $0\le \{\alpha \}<1$ denoting fractional part of $\alpha$.
\end{theorem}

The obvious question is, can Theorem \ref{theorem:posDisc} be used to explain our polar-finite decomposition Theorem \ref{theorem:generalPolarFinite}?  In a way, yes.  Using Appell function properties found in \cite[Section 3]{HM}, one can obtain the right-hand side of Theorem \ref{theorem:generalPolarFinite} up to a theta function, where the theta function can be writen down explicitly in terms of (\ref{equation:JTPid}).  However, the theta expression is not practical for computing examples.

For our final result, we will give a heuristic argument for the following: For $(p,p^{\prime})=(p,2p+j)$ the general even-spin string function can be expressed modulo a theta function:
\begin{align}
(q)_{\infty}^{3}&\mathcal{C}_{2k,2r}^{(p,2p+j)}(q)
\label{equation:generalStringHeuristic}\\
&  = -(-1)^{p}q^{\binom{p}{2}-p(r+k)}
\sum_{m=1}^{p-1}(-1)^{m}q^{\binom{m+1}{2}+m(r+k-p)}
\notag\\
&  \qquad \times\left (  j(-q^{(2p+j)m+p(2r+1)};q^{2p(2p+j)}) -
q^{m(2p+j-(2r+1))} j(-q^{-m(2p+j)+p(2r+1)};q^{2p(2p+j)})\right )
\notag \\
&  \qquad \qquad \times 
\left ( m(-q^{-2pk+jm},*;q^{2pj}) 
+ q^{2k(p-m)}m(-q^{2pk+jm},*;q^{2pj})\right ) + theta, 
\notag
\end{align}
where by ``$*$'' we mean any generic value.  The heuristic expansion (\ref{equation:generalStringHeuristic}) gives a more refined look at the polar-finite decomposition of Theorem \ref{theorem:generalPolarFinite}.  

Finally, and to tie in with the title, the interested reader will immediately note the similarity with the false theta function expansion found in \cite{BoMo2024}.  Let us recall that
\begin{equation*}
    sg(r):=\begin{cases}
1 & \textup{if } r\ge 0,\\
-1 & \textup{if } r<0.
    \end{cases}
\end{equation*}
Theorem $2.7$ of \cite{BoMo2024} reads:
Let $N := p'/p-2$ be a negative admissible-level, that is, $p\geq 1$, $p'\geq 2$ are coprime integers and $p'<2p$. For $0\le \ell \le p^{\prime}-2$, and $m\in2\mathbb{Z}+\ell$ we have
\begin{align*}
(q)_{\infty}^{3}&\mathcal{C}_{m,\ell}^{N}(q)\\
&=- \frac{1}{2}\sum_{k=1}^{p-1} 
(-1)^{k}q^{(\frac{m-\ell}{2})k}q^{\binom{k}{2}}
\left ( j(-q^{p^{\prime}k+p(p^{\prime} -(\ell+1))};q^{2pp^{\prime}})
-q^{(1+\ell)k}
j(-q^{p^{\prime}k+p(p^{\prime}+\ell+1)};q^{2pp^{\prime}})\right) \\ 
&\qquad \times
\left ( \sum_{R\in\mathbb{Z}}\sg(R)
q^{-p^2NR^2+pR(m-kN)}
-q^{(k-p)(pN-m)}
\sum_{R\in\mathbb{Z}}\sg(R)
q^{-p^2NR^2+pR(m-(2p-k)N)}
\right ).
\end{align*}

\section{Overview of the paper}

We begin with some preliminaries.  In Section \ref{section:technicalPrelim}, we will review basic facts about theta functions, Appell functions, and Hecke-type double-sums.  We will also give a more convenient form of the Weyl--Kac character formula (\ref{equation:WK-formula}).  In Section \ref{section:alternateAppellForms}, we give alternate Appell function forms for Ramanujan's third-order and tenth-order mock theta functions.
In Section \ref{section:fryeGarvan}, we prove the three families of theta function identities necessary for the proofs of mock theta conjecture-like identities found in Theorems \ref{theorem:newMockThetaIdentitiespP38m0ell2r}, \ref{theorem:newMockThetaIdentitiespP38m2ell2r}, and \ref{theorem:newMockThetaIdentitiespP511m0ell2r}. 

In Section \ref{section:quasiPeriodicEvenSpin}, we prove the even-spin quasi-periodic relation found in Theorem \ref{theorem:generalQuasiPeriodicity}.   In Section \ref{section:crossSpin-j-Odd}, we give a proof of the $j$-odd cross-spin identity of Theorem \ref{theorem:crossSpin-j-Odd}.  Using Theorem \ref{theorem:crossSpin-j-Odd}, we also give a new proof of the old $1/2$-level cross-spin identity \cite[Corollary $5.1$]{BoMo2024}, as well as a proof of a new $1/3$-level cross-spin identity (\ref{equation:crossSpin37})

In Section \ref{section:polarFinite}, we use our quasi-periodic relation of Theorem \ref{theorem:generalQuasiPeriodicity} to prove our even-spin polar-finite decomposition of Theorem \ref{theorem:generalPolarFinite}.  In Sections \ref{section:mockTheta12-level}--\ref{section:mockTheta15-level}, we use our polar-finite decomposition Theorem \ref{theorem:generalPolarFinite} to give proofs of the old mock theta conjecture-like identities (\ref{equation:2nd-A(q)}) and (\ref{equation:2nd-mu(q)}), as well as the new mock theta conjecture-like identities of Theorems \ref{theorem:newMockThetaIdentitiespP37m0ell2r}, \ref{theorem:newMockThetaIdentitiespP38m0ell2r}, \ref{theorem:newMockThetaIdentitiespP38m2ell2r}, and \ref{theorem:newMockThetaIdentitiespP511m0ell2r}.  In Section \ref{section:mockTheta12-level}, we prove (\ref{equation:2nd-A(q)}) and (\ref{equation:2nd-mu(q)}).  In Section \ref{section:mockTheta13-level} we prove Theorem \ref{theorem:newMockThetaIdentitiespP37m0ell2r}.  In Section \ref{section:mockTheta23-level} we prove Theorems \ref{theorem:newMockThetaIdentitiespP38m0ell2r} and \ref{theorem:newMockThetaIdentitiespP38m2ell2r}.  In Section \ref{section:mockTheta15-level} we prove Theorem \ref{theorem:newMockThetaIdentitiespP511m0ell2r}.

In Section \ref{section:generalStringHeuristic}, we give a heuristic argument to support our mod theta expansion for the positive admissible-level even-spin string function (\ref{equation:generalStringHeuristic}).


\section{Technical preliminaries}\label{section:technicalPrelim}

We will frequently use the following product rearrangements:
\begin{subequations}
\begin{gather}
\overline{J}_{0,1}=2\overline{J}_{1,4}=\frac{2J_2^2}{J_1},  \overline{J}_{1,2}=\frac{J_2^5}{J_1^2J_4^2},   J_{1,2}=\frac{J_1^2}{J_2},   \overline{J}_{1,3}=\frac{J_2J_3^2}{J_1J_6}, \notag\\
J_{1,4}=\frac{J_1J_4}{J_2},   J_{1,6}=\frac{J_1J_6^2}{J_2J_3},   \overline{J}_{1,6}=\frac{J_2^2J_3J_{12}}{J_1J_4J_6}.\notag
\end{gather}
\end{subequations}

\noindent Following from the definitions are the general identities:
\begin{subequations}
{\allowdisplaybreaks \begin{gather}
j(q^n x;q)=(-1)^nq^{-\binom{n}{2}}x^{-n}j(x;q), \ \ n\in\mathbb{Z},\label{equation:j-elliptic}\\
j(x;q)=j(q/x;q)=-xj(x^{-1};q)\label{equation:j-flip},\\
j(x;q)={J_1}j(x,qx,\dots,q^{n-1}x;q^n)/{J_n^n} \ \ {\text{if $n\ge 1$,}}\label{equation:1.10}\\
j(x;-q)={j(x;q^2)j(-qx;q^2)}/{J_{1,4}},\label{equation:1.11}\\
j(x^n;q^n)={J_n}j(x,\zeta_nx,\dots,\zeta_n^{n-1}x;q^n)/{J_1^n} \ \ {\text{if $n\ge 1$.}}\label{equation:1.12}\\
j(z;q)=\sum_{k=0}^{m-1}(-1)^k q^{\binom{k}{2}}z^k
j\big ((-1)^{m+1}q^{\binom{m}{2}+mk}z^m;q^{m^2}\big ),\label{equation:jsplit}\\
j(qx^3;q^3)+xj(q^2x^3;q^3)=j(-x;q)j(qx^2;q^2)/J_2={J_1j(x^2;q)}/{j(x;q)},\label{equation:quintuple}
\end{gather}}%
\end{subequations}
\noindent  where identity (\ref{equation:quintuple}) is the quintuple product identity.   For later use, we state the $m=2$ specialization of (\ref{equation:jsplit})
\begin{equation}
j(z;q)=j(-qz^2;q^4)-zj(-q^3z^2;q^4).\label{equation:jsplit-m2}
\end{equation}
We recall the classical partial fraction  expansion for the reciprocal of Jacobi's theta product.  Here $z$ is not an integral power of $q$:
\begin{equation}
\sum_{n\in\mathbb{Z}}\frac{(-1)^nq^{\binom{n+1}{2}}}{1-q^nz}=\frac{J_1^3}{j(z;q)}.\label{equation:jacobiThetaReciprocal}
\end{equation}

We then recall \cite[Theorem 1.1]{H1}
\begin{equation}
j(x;q)j(y;q)
=j(-xy;q^{2})j(-qx^{-1}y;q^{2})-xj(-qxy;q^{2})j(-x^{-1}y;q^{2}),
\label{equation:H1Thm1.1}
\end{equation}
and \cite[Theorem 1.2]{H1}
\begin{gather}
j(-x;q)j(y;q)-j(x;q)j(-y;q)
=2xj(x^{-1}y;q^2)j(qxy;q^2),\label{equation:H1Thm1.2A}\\
j(-x;q)j(y;q)+j(x;q)j(-y;q)
=2j(xy;q^2)j(qx^{-1}y;q^2).\label{equation:H1Thm1.2B}
\end{gather}

We will need to use a few elementary properties of Appell functions \cite[Section 3]{HM}, \cite{Zw2}.
\begin{proposition}  For generic $x,z\in \mathbb{C}^*$
{\allowdisplaybreaks \begin{subequations}
\begin{gather}
m(x,z;q)=m(x,qz;q),\label{equation:mxqz-fnq-z}\\
m(x,z;q)=x^{-1}m(x^{-1},z^{-1};q),\label{equation:mxqz-flip}\\
m(qx,z;q)=1-xm(x,z;q).\label{equation:mxqz-fnq-x}
\end{gather}
\end{subequations}}
\end{proposition}
We have the changing-$z$ property:
\begin{theorem} \label{theorem:changing-z-theorem}  For generic $x,z_0,z_1\in \mathbb{C}^*$
\begin{equation}
m(x,z_1;q)-m(x,z_0;q)=\Psi(x,z_1,z_0;q),
\end{equation}
where
\begin{equation}
\ \Psi(x,z_1,z_0;q):=\frac{z_0J_1^3j(z_1/z_0;q)j(xz_0z_1;q)}{j(z_0;q)j(z_1;q)j(xz_0;q)j(xz_1;q)}.\label{equation:PsiDef}
\end{equation}
\end{theorem}
The changing-$z$ theorem has the following immediate corollary
\begin{corollary} \label{corollary:mxqz-flip-xz} We have
\begin{equation}
m(x,q,z)=m(x,q,x^{-1}z^{-1}).
\end{equation}
\end{corollary}

We also have a general form of the changing-$z$ theorem:
\begin{theorem} \label{theorem:msplit-general-n} \cite[Theorem 3.5]{HM} For generic $x,z,z'\in \mathbb{C}\backslash\{0\}$ 
{\allowdisplaybreaks \begin{align*}
m(&x,q,z) = \sum_{r=0}^{n-1} q^{{-\binom{r+1}{2}}} (-x)^r 
m\big({-}q^{{\binom{n}{2}-nr}} (-x)^n, q^{n^2}, z' \big)\notag\\
& + \frac{z' J_n^3}{j(xz;q) j(z';q^{n^2})}  \sum_{r=0}^{n-1}
\frac{q^{{\binom{r}{2}}} (-xz)^r
j\big({-}q^{{\binom{n}{2}+r}} (-x)^n z z';q^n\big)
j(q^{nr} z^n/z';q^{n^2})}
{ j\big({-}q^{{\binom{n}{2}}} (-x)^n z';q^n\big ) 
j\big (  q^r z;q^n\big )}.
\end{align*}}
\end{theorem}

We have an important functional equation for double-sums:
\begin{proposition}\cite[Proposition 6.3]{HM}  \label{proposition:f-functionaleqn} For $x,y\in\mathbb{C}^*$ and $\ell, k \in \mathbb{Z}$
\begin{align}
f_{a,b,c}(x,y;q)&=(-x)^{\ell}(-y)^kq^{a\binom{\ell}{2}+b\ell k+c\binom{k}{2}}f_{a,b,c}(q^{a\ell+bk}x,q^{b\ell+ck}y;q) \notag\\
&\ \ \ \ +\sum_{m=0}^{\ell-1}(-x)^mq^{a\binom{m}{2}}j(q^{mb}y;q^c)+\sum_{m=0}^{k-1}(-y)^mq^{c\binom{m}{2}}j(q^{mb}x;q^a),\label{equation:Gen1}
\end{align}
where when $b<a$, we follow the usual summation convention 
\begin{equation}
\sum_{r=a}^{b}c_r:=-\sum_{r=b+1}^{a-1}c_r, \ \textup{e.g.} \ 
\sum_{r=0}^{-1}c_r=-\sum_{r=0}^{-1}c_r=0. \label{equation:sumconvention}
\end{equation}
\end{proposition}

We also have a more convenient form of the Weyl--Kac theorem:
\begin{proposition} \label{proposition:WeylKac} We have
\begin{align*}
\chi_{\ell}^{(p,p^{\prime})}(z;q)
&=z^{-\frac{\ell+1}{2}}q^{p\frac{(\ell+1)^2}{4p^{\prime}}}
\frac{j(-q^{p(\ell+1)+pp^{\prime}}z^{-p^{\prime}};q^{2pp^{\prime}})
-z^{\ell+1}j(-q^{-p(\ell+1)+pp^{\prime}}z^{-p^{\prime}};q^{2pp^{\prime}}) }
{z^{-\frac{1}{2}}q^{\frac{1}{8}}j(z;q)}.
\end{align*}
\end{proposition}

\begin{proof}[Proof of Proposition \ref{proposition:WeylKac}]
We recall (\ref{equation:WK-formula}):
\begin{equation*}
\chi_{\ell}^N(z;q)=\frac{\sum_{\sigma=\pm 1}\sigma \Theta_{\sigma (\ell+1),p^{\prime}}(z;q^{p})}
{\sum_{\sigma=\pm 1}\sigma\Theta_{\sigma,2}(z;q)}.
\end{equation*}

We consider general $(p,p^{\prime})$.  We first change the notation of the theta functions using (\ref{equation:Theta-to-j}).  This gives
\begin{align*}
\chi_{\ell}^{(p,p^{\prime})}(z;q)
&=\frac{\Theta_{ (\ell+1),p^{\prime}}(z;q^{p})-\Theta_{- (\ell+1),p^{\prime}}(z;q^{p})}
{\Theta_{1,2}(z;q)-\Theta_{-1,2}(z;q)}\\
&=\frac{z^{-\frac{\ell+1}{2}}q^{p\frac{(\ell+1)^2}{4p^{\prime}}}j(-q^{p(\ell+1)+pp^{\prime}}z^{-p^{\prime}};q^{2pp^{\prime}})
-z^{\frac{\ell+1}{2}}q^{p\frac{(\ell+1)^2}{4p^{\prime}}}j(-q^{-p(\ell+1)+pp^{\prime}}z^{-p^{\prime}};q^{2pp^{\prime}}) }
{z^{-\frac{1}{2}}q^{\frac{1}{8}}j(-q^{3}z^{-2};q^4)-z^{\frac{1}{2}}q^{\frac{1}{8}}j(-qz^{-2};q^4)}.
\end{align*}
We then adjust the denominator using (\ref{equation:j-flip}) and (\ref{equation:jsplit-m2}) to obtain
\begin{align*}
\chi_{\ell}^{(p,p^{\prime})}(z;q)
&=\frac{z^{-\frac{\ell+1}{2}}q^{p\frac{(\ell+1)^2}{4p^{\prime}}}j(-q^{p(\ell+1)+pp^{\prime}}z^{-p^{\prime}};q^{2pp^{\prime}})
-z^{\frac{\ell+1}{2}}q^{p\frac{(\ell+1)^2}{4p^{\prime}}}j(-q^{-p(\ell+1)+pp^{\prime}}z^{-p^{\prime}};q^{2pp^{\prime}}) }
{z^{-\frac{1}{2}}q^{\frac{1}{8}}j(-qz^{2};q^4)-z^{\frac{1}{2}}q^{\frac{1}{8}}j(-q^3z^{2};q^4)}\\
&=\frac{z^{-\frac{\ell+1}{2}}q^{p\frac{(\ell+1)^2}{4p^{\prime}}}j(-q^{p(\ell+1)+pp^{\prime}}z^{-p^{\prime}};q^{2pp^{\prime}})
-z^{\frac{\ell+1}{2}}q^{p\frac{(\ell+1)^2}{4p^{\prime}}}j(-q^{-p(\ell+1)+pp^{\prime}}z^{-p^{\prime}};q^{2pp^{\prime}}) }
{z^{-\frac{1}{2}}q^{\frac{1}{8}}j(z;q)}.
\end{align*}
Factoring out common terms in the numerator gives the result.
\end{proof}


\section{Ramanujan's classical mock theta functions}
\label{section:alternateAppellForms}
From \cite[Section 5]{HM}, we have the following classical mock theta functions in Appell function form.

\noindent {\bf `3rd order' functions}
{\allowdisplaybreaks \begin{align}
f_3(q)&=\sum_{n\ge 0}\frac{q^{n^2}}{(-q)_n^2}
=2m(-q,q;q^3)+2m(-q,q^2;q^3)\label{equation:3rd-f(q)}\\
\omega_3(q)
&=\sum_{n\ge 0}\frac{q^{2n(n+1)}}{(q;q^2)_{n+1}^2}
=-q^{-1}m(q,q^2;q^6)-q^{-1}m(q,q^4;q^6)
\label{equation:3rd-omega(q)}
\end{align}}%

\noindent {\bf `10th order' functions}
{\allowdisplaybreaks \begin{align}
{\phi_{10}}(q)&=\sum_{n\ge 0}\frac{q^{\binom{n+1}{2}}}{(q;q^2)_{n+1}}
=-q^{-1}m(q,q;q^{10})-q^{-1}m(q,q^2;q^{10})
\label{equation:10th-phi(q)}\\
{\psi_{10}}(q)
&=\sum_{n\ge 0}\frac{q^{\binom{n+2}{2}}}{(q;q^2)_{n+1}}
=-m(q^3,q;q^{10})-m(q^3,q^{3};q^{10})
\label{equation:10th-psi(q)}\\
{X_{10}}(q)
&=\sum_{n\ge 0}\frac{(-1)^nq^{n^2}}{(-q;q)_{2n}}
=m(-q^2,q;q^{5})+m(-q^2,q^{4};q^{5})
\label{equation:10th-BigX(q)}\\
{\chi_{10}}(q)
&=\sum_{n\ge 0}\frac{(-1)^nq^{(n+1)^2}}{(-q;q)_{2n+1}}
=m(-q,q^{2};q^{5})+m(-q,q^{3};q^{5})
\label{equation:10th-chi(q)}
\end{align}}%

Using the known Appell function forms of the mock theta functions that we will be using, we find more convenient Appell function forms.
\begin{proposition} \label{proposition:alternat3rdAppellForms2} We have
\begin{gather}
2m(-q^4,q^6;q^{12})
=\frac{1}{2}f_{3}(q^4)
-\frac{1}{2}\frac{J_{2}^4J_{12}^6}{J_{4}^3J_{6}^4J_{24}^2},
\label{equation:altAppellForm3rd-f}\\
2m(-q^2,q^6;q^{12})
=q^2\omega_{3}(-q^2)
-q^2\frac{J_{24}^2J_{2}^4}{J_{4}^3J_{6}^2}.
\label{equation:altAppellForm3rd-omega}
\end{gather}
\end{proposition}

\begin{proposition}\label{proposition:alternat10thAppellForms} We have
{\allowdisplaybreaks \begin{gather}
2m(-q,q^5;q^{10})
=q\phi_{10}(-q)
-q\frac{J_{10}^2J_{3,10}}
{\overline{J}_{1,5}J_{2,10}}
\cdot \frac{J_{1}}{\overline{J}_{3,10}},
\label{equation:newAppell10thPhi}\\
2m(-q^2,q^5;q^{10})
=\chi_{10}(q^2) 
-q^{2}\frac{J_{10}^2J_{1,10}}{\overline{J}_{2,5}J_{4,10}}
\cdot \frac{J_{1}}
{\overline{J}_{4,10}},
\label{equation:newAppell10thChi}\\
2m(-q^3,q^5;q^{10})
=-\psi(-q)
 -q\frac{J_{10}^2 J_{1,10}}
{\overline{J}_{2,5}J_{4,10}}
\cdot \frac{J_{1}}
{\overline{J}_{1,10}},
\label{equation:newAppell10thPsi}\\
2m(-q^4,q^5;q^{10})
=X_{10}(q^2)
-\frac{J_{10}^2J_{3,10}}
{\overline{J}_{1,5}J_{2,10}}
\cdot \frac{J_{1}}
{\overline{J}_{2,10}}.
\label{equation:newAppell10thX}
\end{gather}}%
\end{proposition}

\begin{proof}[Proof of Proposition \ref{proposition:alternat3rdAppellForms2}]
We prove (\ref{equation:altAppellForm3rd-f}).  Theorem \ref{theorem:changing-z-theorem} allows us to write
\begin{align*}
2m(-q^4,q^6;q^{12})
&=m(-q^4,-1;q^{12})
-\frac{J_{12}^3j(-q^{6};q^{12})J_{10,12}}{J_{6,12}\overline{J}_{2,12}\overline{J}_{0,12}J_{4,12}}\\
&\qquad + m(-q^4,q^4;q^{12})
+q^{4}\frac{J_{12}^3J_{2,12}j(-q^{14};q^{12})}{J_{6,12}\overline{J}_{2,12}J_{4,12}\overline{J}_{8,12}}.
\end{align*}
Using (\ref{equation:j-elliptic}) and the Appell function form of $f_{3}(q)$ (\ref{equation:3rd-f(q)}) gives us 
\begin{align*}
2m(-q^4,q^6;q^{12})
&=\frac{1}{2}f_{3}(q^4)
-\frac{J_{12}^3J_{10,12}}{J_{6,12}\overline{J}_{2,12}J_{4,12}}
\left (\frac{\overline{J}_{6,12}\overline{J}_{8,12}-q^2\overline{J}_{0,12}\overline{J}_{2,12}}
{\overline{J}_{0,12}\overline{J}_{8,12}} \right ).
\end{align*}
Using a classic theta function identity (\ref{equation:H1Thm1.1}) gives
\begin{align*}
2m(-q^4,q^6;q^{12})
&=\frac{1}{2}f_{3}(q^4)
-\frac{J_{12}^3J_{10,12}}{J_{6,12}\overline{J}_{2,12}J_{4,12}}
\left (\frac{j(q^2;q^{6})j(q^4;q^{6})}
{\overline{J}_{0,12}\overline{J}_{8,12}} \right ).
\end{align*}
The result follows from elementary product rearrangements.

We prove (\ref{equation:altAppellForm3rd-omega}).  Using Theorem \ref{theorem:changing-z-theorem} we have
\begin{align*}
2m(-q^2,q^6;q^{12})
&=m(-q^2,q^4;q^{12})
+q^4\frac{J_{12}^3J_{2,12}\overline{J}_{0,12}}{J_{6,12}\overline{J}_{4,12}J_{4,12}\overline{J}_{6,12}}\\
&\qquad + m(-q^2,q^8;q^{12})
+q^8\frac{J_{12}^3j(q^{-2};q^{12})j(-q^{16};q^{12})}{J_{6,12}\overline{J}_{4,12}J_{4,12}\overline{J}_{2,12}}.
\end{align*}
We adjust the theta functions with (\ref{equation:j-elliptic}) and then use the Appell function form of $\omega_{3}(q)$ (\ref{equation:3rd-omega(q)}).  This gives us 
\begin{align*}
2m(-q^2,q^6;q^{12})
&=q^2\omega_{3}(-q^2)
-q^2\frac{J_{12}^3J_{2,12}}{J_{6,12}\overline{J}_{4,12}J_{4,12}}
\left (\frac{\overline{J}_{4,12}\overline{J}_{6,12} -q^2\overline{J}_{0,12}\overline{J}_{2,12}}{\overline{J}_{2,12}\overline{J}_{6,12}}\right ) .
\end{align*}
Using the identity (\ref{equation:H1Thm1.1}) yields
\begin{align*}
2m(-q^2,q^6;q^{12})
&=q^2\omega_{3}(-q^2)
-q^2\frac{J_{12}^3J_{2,12}}{J_{6,12}\overline{J}_{4,12}J_{4,12}}
\left (\frac{j(q^2;q^6)j(q^4;q^6)}{\overline{J}_{2,12}\overline{J}_{6,12}}\right ) .
\end{align*}
Elementary product rearrangements give the result.
\end{proof}

\begin{proof}[Proof of Proposition \ref{proposition:alternat10thAppellForms}]
We prove (\ref{equation:newAppell10thPhi}).
Using Theorem \ref{theorem:changing-z-theorem}, gives
\begin{align*}
2m(-q,q^5;q^{10})
&=m(-q,-q;q^{10})
+\frac{J_{10}^3j(-q^4;q^{10})j(q^7;q^{10})}
{J_{5,10}\overline{J}_{6,10}\overline{J}_{1,10}J_{2,10}}\\
&\qquad + m(-q,q^2;q^{10})
+\frac{J_{10}^3j(q^3;q^{10})j(-q^8;q^{10})}
{J_{5,10}\overline{J}_{6,10}\overline{J}_{3,10}J_{2,10}}.
\end{align*}
Using (\ref{equation:10th-phi(q)}), we then get
\begin{align*}
2m(-q,q^5;q^{10})
&=q\phi_{10}(-q)
-q\frac{J_{10}^3j(-q^4;q^{10})j(q^7;q^{10})}
{J_{5,10}\overline{J}_{6,10}\overline{J}_{1,10}J_{2,10}}\\
&\qquad +q^2\frac{J_{10}^3j(q^3;q^{10})j(-q^8;q^{10})}
{J_{5,10}\overline{J}_{6,10}\overline{J}_{3,10}J_{2,10}}\\
&=q\phi_{10}(-q)
-q\frac{J_{10}^3J_{3,10}}
{J_{5,10}\overline{J}_{6,10}J_{2,10}}
\left (\frac{\overline{J}_{3,10}\overline{J}_{6,10}
-q\overline{J}_{8,10}\overline{J}_{1,10}}{\overline{J}_{1,10}\overline{J}_{3,10}}
\right).
\end{align*}
Using (\ref{equation:H1Thm1.1}) gives
\begin{align*}
2m(-q,q^5;q^{10})
&=q\phi_{10}(-q)
-q\frac{J_{10}^3J_{3,10}}
{J_{5,10}\overline{J}_{6,10}J_{2,10}}
\left (\frac{J_{1,5}J_{2,5}}{\overline{J}_{1,10}\overline{J}_{3,10}}
\right)\\
&=q\phi_{10}(-q)
-q\frac{J_{10}^2J_{3,10}}
{J_{2,10}}
\cdot \frac{J_{1}}{\overline{J}_{1,5}\overline{J}_{3,10}}.
\end{align*}


We prove (\ref{equation:newAppell10thChi}). Using Theorem \ref{theorem:changing-z-theorem}, gives
\begin{align*}
2m(-q^2,q^5;q^{10})
&=m(-q^2,q^4;q^{10}) 
+ q^4\frac{J_{10}J_{1,10}j(-q^{11};q^{10})}
{J_{5,10}\overline{J}_{3,10}J_{4,10}\overline{J}_{4,10}}\\
&\qquad +m(-q^2,q^6;q^{10})
+q^{6}\frac{J_{10}^3j(q^{-1};q^{10})j(-q^{13};q^{10})}
{J_{5,10}\overline{J}_{3,10}J_{6,10}\overline{J}_{2,10}}.
\end{align*}
Using (\ref{equation:10th-chi(q)}), we then get
\begin{align*}
2m(-q^2,q^5;q^{10})
&=\chi_{10}(q^2) 
-q^{2}\frac{J_{10}^3J_{1,10}}{J_{5,10}\overline{J}_{3,10}J_{4,10}}
\left ( \frac{\overline{J}_{3,10}\overline{J}_{6,10}
-q\overline{J}_{8,10}\overline{J}_{1,10}}{\overline{J}_{2,10}\overline{J}_{6,10}}\right ). 
\end{align*}
Using (\ref{equation:H1Thm1.1}) gives
\begin{align*}
2m(-q^2,q^5;q^{10})
&=\chi_{10}(q^2) 
-q^{2}\frac{J_{10}^3J_{1,10}}{J_{5,10}\overline{J}_{3,10}J_{4,10}}
\left ( \frac{J_{1,5}J_{2,5}}{\overline{J}_{2,10}\overline{J}_{6,10}}\right )\\ 
&=\chi_{10}(q^2) 
-q^{2}\frac{J_{10}^2J_{1,10}}{J_{4,10}}
\cdot \frac{J_{1}}
{\overline{J}_{2,5}\overline{J}_{6,10}}.
\end{align*}

We prove (\ref{equation:newAppell10thPsi}). Using Theorem \ref{theorem:changing-z-theorem}, gives
\begin{align*}
2m(-q^3,q^5;q^{10})
&=m(-q^3,-q,q^{10})
-q\frac{J_{10}^3\overline{J}_{4,10}J_{1,10}}
{J_{5,10}\overline{J}_{2,10}\overline{J}_{1,10}J_{4,10}}\\
&\qquad +m(-q^3,-q^3;q^{10})
-q^3\frac{J_{10}^3\overline{J}_{2,10}J_{11,10}}
{J_{5,10}\overline{J}_{2,10}\overline{J}_{3,10}J_{4,10}}.
\end{align*}
Using (\ref{equation:H1Thm1.2A}) and then (\ref{equation:H1Thm1.2B}) gives
{\allowdisplaybreaks \begin{align*}
2m(-q^3,q^5;q^{10})
&=-\psi(-q)
 -q\frac{J_{10}^3J_{1,10}}
{J_{5,10}\overline{J}_{2,10}J_{4,10}}
\left ( \frac{\overline{J}_{3,10}\overline{J}_{6,10}
-q\overline{J}_{8,10}\overline{J}_{1,10}}
{\overline{J}_{1,10}\overline{J}_{3,10}}\right )\\
&=-\psi(-q)
 -q\frac{J_{10}^3 J_{1,10}}
{J_{5,10}\overline{J}_{2,5}J_{4,10}}
\left ( \frac{J_{1,5}J_{2,5}}
{\overline{J}_{1,10}\overline{J}_{3,10}}\right )\\
&=-\psi(-q)
 -q\frac{J_{10}^2 J_{1,10}}
{\overline{J}_{2,5}J_{4,10}}
\cdot \frac{J_{1}}
{\overline{J}_{1,10}}.
\end{align*}}%

We prove (\ref{equation:newAppell10thX}). Using Theorem \ref{theorem:changing-z-theorem}, gives
 \begin{align*}
2m(-q^4,q^5;q^{10})
&=m(-q^4,q^2;q^{10})
+q^2\frac{J_{10}^3J_{3,10}j(-q^{11};q^{10})}
{J_{5,10}\overline{J}_{1,10}J_{2,10}\overline{J}_{4,10}}\\
&\qquad + m(-q^4,q^8;q^{10})
+q^8\frac{J_{10}^3j(q^{-3};q^{10})j(-q^{17};q^{10})}
{J_{5,10}\overline{J}_{1,10}J_{8,10}j(-q^{12};q^{10})}.
 \end{align*}
Using (\ref{equation:10th-BigX(q)}) and (\ref{equation:j-elliptic}) yields
 \begin{align*}
2m(-q^4,q^5;q^{10})
&=X_{10}(q^2)
+q\frac{J_{10}^3J_{3,10}\overline{J}_{1,10}}
{J_{5,10}\overline{J}_{1,10}J_{2,10}\overline{J}_{4,10}}\\
&\qquad 
-\frac{J_{10}^3J_{3,10}\overline{J}_{3,10}}
{J_{5,10}\overline{J}_{1,10}J_{8,10}\overline{J}_{2,10}}\\
&=X_{10}(q^2)
-\frac{J_{10}^3J_{3,10}}
{J_{5,10}\overline{J}_{1,10}J_{2,10}}
\left ( \frac{\overline{J}_{3,10}\overline{J}_{6,10}
-q\overline{J}_{8,10}\overline{J}_{1,10}}
{\overline{J}_{2,10}\overline{J}_{4,10}}\right )\\
&=X_{10}(q^2)
-\frac{J_{10}^3J_{3,10}}
{J_{5,10}\overline{J}_{1,10}J_{2,10}}
\left ( \frac{J_{1,5}J_{2,5}}
{\overline{J}_{2,10}\overline{J}_{4,10}}\right )\\
&=X_{10}(q^2)
-\frac{J_{10}^2J_{3,10}}
{\overline{J}_{1,5}J_{2,10}}
\cdot \frac{J_{1}}
{\overline{J}_{2,10}}.\qedhere
\end{align*}
\end{proof}


\section{Families of theta function identities through Frye and Garvan}
\label{section:fryeGarvan}

In this section we prove the three families of theta function identities that are necessary to the proofs of the $2/3$-level mock theta conjecture-like identities found in Theorems \ref{theorem:newMockThetaIdentitiespP38m0ell2r} and \ref{theorem:newMockThetaIdentitiespP38m2ell2r}, and the $1/5$-level mock theta conjecture-like identities found in Theorem \ref{theorem:newMockThetaIdentitiespP511m0ell2r}.

For the quantum-number zero, even-spin, $2/3$-level mock theta conjecture-like identities found in Theorem \ref{theorem:newMockThetaIdentitiespP38m0ell2r}, we need to prove
\begin{proposition}
\label{proposition:masterThetaIdentitypP38m0ell2r} 
For $r\in\{0,1,2,3 \}$, we have
{\allowdisplaybreaks \begin{align*}
(-1)^{\kappa(r)}&\frac{J_{1}^3}{J_{6,12}}\frac{J_2}{J_{1}J_{4}}j(-q^{27+6r};q^{48})\\
&\qquad +q^{3-2r} 
\frac{j(q^{7-2r};q^{16})j(q^{30-4r};q^{32})}{J_{32}}
\frac{J_{24}^2J_{2}^4}{J_{4}^3J_{6}^2} \\
&\qquad    - q^{-r} 
\frac{j(q^{1+2r};q^{16})j(q^{18+4r};q^{32})}{J_{32}}
\frac{1}{2}\frac{J_{2}^4J_{12}^6}{J_{4}^3J_{6}^4J_{24}^2}\\
&=(-1)^{\delta(r)}
\frac{q^{-r}}{2}\frac{J_{1}^2J_{2}}{J_{4}^2J_{8}}
j(-q^{7-2r};q^{16})j(q^{1+2r};q^{8}),
\end{align*}}%
where $\kappa(r):=\begin{cases} 0 & \textup{if} \ r=0,1,\\
1 & \textup{if} \ r=2,3, \end{cases}$ and $\delta(r):=\begin{cases} 0 & \textup{if} \ r=0,3,\\
1 & \textup{if} \ r=1,2. \end{cases}$
\end{proposition}

For the quantum-number two, even-spin,  $2/3$-level mock theta conjecture-like identities found in Theorem \ref{theorem:newMockThetaIdentitiespP38m2ell2r}, we need to prove
\begin{proposition}
\label{proposition:masterThetaIdentitypP38m2ell2r}
For $r\in\{0,1,2,3 \}$, we have
{\allowdisplaybreaks \begin{align*}
-(-1)^{\kappa(r)}&q^{6-3r}
\frac{J_{1}^3}{J_{6,12}}\frac{J_{2}}{J_{1}J_{4}}j(-q^{3+6r};q^{48})\\
 &\qquad -    q^{3-2r} 
\frac{j(q^{7-2r};q^{16})j(q^{30-4r};q^{32})}{J_{32}} 
\frac{1}{2}\frac{J_{2}^4J_{12}^6}{J_{4}^3J_{6}^4J_{24}^2} \\
& \qquad   + q^{1-r} 
\frac{j(q^{1+2r};q^{16})j(q^{18+4r};q^{32})}{J_{32}}
q^2\frac{J_{24}^2J_{2}^4}{J_{4}^3J_{6}^2}\\
&=(-1)^{\delta(r)}
\frac{q^{3-2r}}{2}\frac{J_{1}^2J_{2}}{J_{4}^2J_{8}}
\frac{J_{8}}{J_{32}}j(q^{2+4r};q^{32})j(q^{7-2r};q^{16}),
\end{align*}}%
where $\kappa(r):=\begin{cases} 0 & \textup{if} \ r=0,1,\\
1 & \textup{if} \ r=2,3, \end{cases}$ and $\delta(r):=\begin{cases} 0 & \textup{if} \ r=0,3,\\
1 & \textup{if} \ r=1,2. \end{cases}$
\end{proposition}

And lastly for the quantum-number zero, even-spin,  $1/5$-level mock theta conjecture-like identities found in Theorem \ref{theorem:newMockThetaIdentitiespP511m0ell2r}, we need to prove
\begin{proposition}\label{proposition:masterThetaIdentitypP511m0ell2r}
For $r\in\{0,1,2,3,4 \}$, we have
{\allowdisplaybreaks \begin{align*}
-q^{r^2-3r+1}&J_{1,2}j(q^{4+8r};q^{22})\\
&=2(-1)^{r}\frac{(q)_{\infty}^3}{J_{5,10}}
\frac{j(q^{50-10r};q^{110})}{j(-1;q)}\\
&\qquad 
+q^{6-4r} \times \left ( j(-q^{16+10r};q^{110}) 
-q^{4+8r}j(-q^{6-10r};q^{110})\right )
 \times q\frac{J_{10}^2J_{3,10}}
{\overline{J}_{1,5}J_{2,10}}
\cdot \frac{J_{1}}{\overline{J}_{3,10}}   \\
& \qquad -q^{3-3r}\times \left ( j(-q^{27+10r};q^{110}) 
-q^{3+6r}j(-q^{17-10r};q^{110})\right )
\times q^{2}\frac{J_{10}^2J_{1,10}}{\overline{J}_{2,5}J_{4,10}}
\cdot \frac{J_{1}}
{\overline{J}_{4,10}}  \\
&\qquad 
+q^{1-2r}  \times \left ( j(-q^{38+10r};q^{110}) 
-q^{2+4r}j(-q^{28-10r};q^{110})\right )
\times q\frac{J_{10}^2 J_{1,10}}
{\overline{J}_{2,5}J_{4,10}}
\cdot \frac{J_{1}}
{\overline{J}_{1,10}} \\
&\qquad  
-q^{-r} \times \left ( j(-q^{49+10r};q^{110}) 
-q^{1+2r}j(-q^{39-10r};q^{110})\right )
\times \frac{J_{10}^2J_{3,10}}
{\overline{J}_{1,5}J_{2,10}}
\cdot \frac{J_{1}}
{\overline{J}_{2,10}}.
\end{align*}}%
\end{proposition}

\begin{proof}
[Proofs of Proposition \ref{proposition:masterThetaIdentitypP38m0ell2r}, \ref{proposition:masterThetaIdentitypP38m2ell2r}, and \ref{proposition:masterThetaIdentitypP511m0ell2r}]

We use Frye and Garvan's Maple packages {\em qseries} and  {\em thetaids} \cite{FG} to prove all three families of theta function identities.  As a running example, we use their method to prove the family of identities in Proposition \ref{proposition:masterThetaIdentitypP38m0ell2r}.

We normalize the family of identities to obtain an equivalent family, which is suitable for Frye and Garvan \cite{FG}.  Here we rewrite two theta functions according to the simple product rearrangement 
\begin{equation*}
j(-q^a;q^m)=\frac{j(q^{2a};q^{2m})}{j(q^a;q^m)}\frac{J_{m}^2}{J_{2m}}.
\end{equation*}
This gives the equivalent
\begin{equation}
g_{r}(\tau):=f_{1,r}(\tau)+f_{2,r}(\tau)-f_{3,r}(\tau)-1=0,
\label{equation:finalTheta-family1-normal}
\end{equation}
where
\begin{gather*}
f_{1,r}(\tau):=\frac{1}{\Psi_{r}(\tau)}
(-1)^{\kappa(r)}\frac{J_{1}^3}{J_{6,12}}\frac{J_2}{J_{1}J_{4}}
\frac{j(q^{54+12r};q^{96})}{j(q^{27+6r};q^{48})}\frac{J_{48}^2}{J_{96}},\\
f_{2,r}(\tau):=\frac{1}{\Psi_{r}(\tau)}
q^{3-2r} \frac{j(q^{7-2r};q^{16})j(q^{30-4r};q^{32})}{J_{32}}
\frac{J_{24}^2J_{2}^4}{J_{4}^3J_{6}^2},\\
f_{3,r}(\tau):=\frac{1}{\Psi_{r}(\tau)}
q^{-r} \frac{j(q^{1+2r};q^{16})j(q^{18+4r};q^{32})}{J_{32}}
\frac{1}{2}\frac{J_{2}^4J_{12}^6}{J_{4}^3J_{6}^4J_{24}^2},
\end{gather*}
and
\begin{equation*}
\Psi_{r}(\tau):=(-1)^{\delta(r)}
\frac{q^{-r}}{2}\frac{J_{1}^2J_{2}}{J_{4}^2J_{8}}
\frac{j(q^{14-4r};q^{32})}{j(q^{7-2r};q^{16})}\frac{J_{16}^2}{J_{32}}j(q^{1+2r};q^{8}).
\end{equation*}

For the first step, we use  \cite[Theorem 18]{Rob} to verify that each $f_{i,r}(\tau)$ is a modular function on $\Gamma_{1}(96)$ for $1\le i\le 3$ and $0\le r \le 3$.  We note that the second family of identities also deals with modular functions on $\Gamma_{1}(96)$; whereas the third family deals with modular functions on $\Gamma_{1}(220)$.   

For the second step, we use \cite[Corollary 4]{CKP} to find a complete set $\mathcal{S}_{96}$ of inequivalent cusps for $\Gamma_{1}(96)$.  For $\Gamma_{1}(96)$ there are $128$ inequivalent cusps; whereas for $\Gamma_{1}(220)$, there are $400$ inequivalent cusps.  We determine the fan width of each cusp $\zeta$, call it $\kappa(\zeta,\Gamma_{1}(96))$.  

For the third step, we use \cite[Lemma 3.2]{Biag} to determine the invariant order of each modular function $f_{i,r}$ at each cusp $\zeta \pmod{\Gamma_{1}(96)}$, call it $\textup{ord}(f_{i,r},\zeta)$.

For the fourth step, we use the valence formula \cite[p. 98]{Rank} to calculate to which order $\mathcal{O}(q^{n})$ we need to confirm identity (\ref{equation:finalTheta-family1-normal}).  Let $f$ be a nonzero modular form of weigh $k$ with respect to a subgroup $\Gamma$ of $\Gamma(1):=SL_{2}(\mathbb{Z})$.  We define
\begin{equation*}
\textup{ORD}(f,\Gamma):=\sum_{\zeta \in R^{\star}}\textup{ORD}(f,\zeta,\Gamma), 
\ \textup{where} \ 
\textup{ORD}(f,\zeta,\Gamma):=\kappa(\zeta,\Gamma)\textup{ord}(f,\zeta),
\end{equation*}
where $R^{\star}$ is a fundamental domain region for $\Gamma$, and $\zeta$ is a cusp $\pmod{\Gamma}$.  If $\mu$ is defined to be the index of $\widehat{\Gamma}$ in $\widehat{\Gamma(1)}$ the valence formula then reads
\begin{equation*}
\textup{ORD}(f,\Gamma)=\frac{1}{12}\mu k.
\end{equation*}
However, we have $k=0$, so in our case
\begin{equation*}
\textup{ORD}(f,\Gamma)=0.
\end{equation*}
Hence if we define 
\begin{equation*}
B_{r}:=\sum_{\substack{s\in\mathcal{S}_{96}\\ s\ne i\infty}}\textup{min}(\{ \textup{ORD}(f_{j,r},s,\Gamma_{1}(96)):1\le j\le n\}\cup\{0\}),
\end{equation*}
then we know that (\ref{equation:finalTheta-family1-normal}) is true if and only if 
\begin{equation*}
\textup{ORD}(g_{r}(\tau), i\infty,\Gamma_{1}(96))>-B_{r}.
\end{equation*}
More details for the process can be found in \cite[p. 91]{Rank}.  For the first and second family of identities, we find that $B_{r}=-200$.  For the third family of identities we have $B_{r}=-1415$ for $r=0,2$ but $B_{r}=-1416$ for $r=1,3,4$.

For the fifth and final step, we verify identity (\ref{equation:finalTheta-family1-normal}) out through $\mathcal{O}(q^{202}).$ \qedhere
\end{proof}


\section{The quasi-periodic relations for general positive admissible-level string functions}
\label{section:quasiPeriodicEvenSpin}

We prove Theorem \ref{theorem:generalQuasiPeriodicity}, but we will do so in a series of steps.  We first recall the Hecke-type double-sum form for our string function.  Let $p'\geq 2$, $p\geq 1$ be co-prime integers, $0\leq \ell \leq p'-2$ and $m\in 2\Z +\ell$. We specialize  (\ref{equation:modStringFnHeckeForm}) with $(m,\ell)=(2k,2r)$, $(p,p^{\prime})=(p,2p+j)$.  This gives
\begin{align}
(q)_{\infty}^{3}\mathcal{C}_{2k,2r}^{(p,2p+j)}(q)
 &= f_{1,2p+j,2p(2p+j)}(q^{1+k+r},-q^{p(2p+j+2r+1)};q)
 \label{equation:stringQuasiHecke}\\
&\qquad  -f_{1,2p+j,2p(2p+j)}(q^{k-r},-q^{p(2p+j-2r-1)};q).
\notag
\end{align}  

Using the functional equation Proposition \ref{proposition:f-functionaleqn} for Hecke-type double-sums, we are able to establish the following relation:

\begin{proposition} \label{proposition:quasiPeriodicity-step1} We have the following expression
\begin{align}
 (q)_{\infty}^{3}&\mathcal{C}_{2k,2r}^{(p,2p+j)}(q)
 -(q)_{\infty}^{3}q^{-p(2k+j)} \mathcal{C}_{2k+2j,2r}^{(p,2p+j)}(q)
 \label{equation:quasiPeriodString1}\\
&\qquad   = - \sum_{m=1}^{2p}(-1)^{m}q^{-m(1+k+r)}q^{\binom{m+1}{2}}
j(-q^{(p-m)(2p+j)+p(2r+1)};q^{2p(2p+j)})
\notag \\
&\qquad  \qquad  + \sum_{m=1}^{2p}(-1)^{m}q^{-m(k-r)}q^{\binom{m+1}{2}}j(-q^{(p+m)(2p+j)+p(2r+1)};q^{2p(2p+j)}).
\notag
\end{align}
\end{proposition}

Taking advantage of some cancellation as well as some symmetries brings us to a more compact expression:

\begin{proposition}\label{proposition:quasiPeriodicity-step2} We have
{\allowdisplaybreaks \begin{align*}
 (q)_{\infty}^{3}&\mathcal{C}_{2k,2r}^{(p,2p+j)}(q)
 -(q)_{\infty}^{3}q^{-p(2k+j)} \mathcal{C}_{2k+2j,2r}^{(p,2p+j)}(q) \\
&  = -(-1)^{p}q^{\binom{p}{2}-p(k+r)}\sum_{m=1}^{p-1}(-1)^{m}q^{\binom{m+1}{2}+m(r-p)}(q^{mk}-q^{-m(k+j)})\\
&\qquad \qquad \times \Big (  j(-q^{m(2p+j)+p(2r+1)};q^{2p(2p+j)} )\\
&\qquad \qquad \qquad  \qquad  -  q^{m(2p+j)-m(2r+1)}j(-q^{-m(2p+j)+p(2r+1)};q^{2p(2p+j)})\Big ).
\end{align*}}%
\end{proposition}

Now we can start the proof of the Theorem \ref{theorem:generalQuasiPeriodicity}.  The proofs of the two propositions follow.

\begin{proof}[Proof of Theorem \ref{theorem:generalQuasiPeriodicity}]
Changing the string function notation with (\ref{equation:mathCalCtoStringC}), we have
\begin{equation*}
    q^{-\frac{1}{8}+\frac{p(2r+1)^2}{4(2p+j)}-\frac{p}{j}k^2}
    \mathcal{C}_{2k,2r}^{(p,2p+j)}(q)
    =C_{2k,2r}^{(p,2p+j)}(q).
\end{equation*}
Hence we have
{\allowdisplaybreaks \begin{align*}
 (q)_{\infty}^{3}&C_{2k,2r}^{(p,2p+j)}(q)
 -(q)_{\infty}^{3}C_{2k+2j,2r}^{(p,2p+j)}(q) \\
& = -(-1)^{p}q^{-\frac{1}{8}+\frac{p(2r+1)^2}{4(2p+j)}-\frac{p}{j}k^2}q^{\binom{p}{2}-p(k+r)}\sum_{m=1}^{p-1}(-1)^{m}q^{\binom{m+1}{2}+m(r-p)}(q^{mk}-q^{-m(k+j)})\\
&\qquad \qquad \times \Big (  j(-q^{m(2p+j)+p(2r+1)};q^{2p(2p+j)} )\\
&\qquad \qquad \qquad  \qquad  -  q^{m(2p+j)-m(2r+1)}j(-q^{-m(2p+j)+p(2r+1)};q^{2p(2p+j)})\Big ).
\end{align*}}%
Rearranging terms and replacing $k\to k-j$ gives
\begin{align*}
 (q)_{\infty}^{3}&C_{2k,2r}^{(p,2p+j)}(q)
 -(q)_{\infty}^{3}C_{2k-2j,2r}^{(p,2p+j)}(q) \\
& = (-1)^{p}q^{-\frac{1}{8}+\frac{p(2r+1)^2}{4(2p+j)}-\frac{p}{j}(k-j)^2}q^{\binom{p}{2}-p(k-j+r)}\sum_{m=1}^{p-1}(-1)^{m}q^{\binom{m+1}{2}+m(r-p)}(q^{m(k-j)}-q^{-mk})\\
&\qquad \qquad \times \Big (  j(-q^{m(2p+j)+p(2r+1)};q^{2p(2p+j)} )\\
&\qquad \qquad \qquad  \qquad  -  q^{m(2p+j)-m(2r+1)}j(-q^{-m(2p+j)+p(2r+1)};q^{2p(2p+j)})\Big ).
\end{align*}
Simplifying the exponents yields
{\allowdisplaybreaks\begin{align*}
 (q)_{\infty}^{3}&C_{2k,2r}^{(p,2p+j)}(q)
 -(q)_{\infty}^{3}C_{2k-2j,2r}^{(p,2p+j)}(q) \\
&= (-1)^{p}q^{-\frac{1}{8}+\frac{p(2r+1)^2}{4(2p+j)}-\frac{p}{j}k^2}q^{\binom{p}{2}+p(k-r)}\sum_{m=1}^{p-1}(-1)^{m}q^{\binom{m+1}{2}+m(r-p)}(q^{m(k-j)}-q^{-mk})\\
&\qquad \qquad \times \Big (  j(-q^{m(2p+j)+p(2r+1)};q^{2p(2p+j)} )\\
&\qquad \qquad \qquad  \qquad  -  q^{m(2p+j)-m(2r+1)}j(-q^{-m(2p+j)+p(2r+1)};q^{2p(2p+j)})\Big ).
\end{align*}}%

Now, we consider $m=2k$, $k=jt+s$, $0\le s\le j-1$.  This allows us to write
{\allowdisplaybreaks \begin{align*}
 (q)_{\infty}^{3}&C_{2jt+2s,2r}^{(p,2p+j)}(q)
 -(q)_{\infty}^{3}C_{2j(t-1)+2s,2r}^{(p,2p+j)}(q) \\
&= (-1)^{p}q^{-\frac{1}{8}+\frac{p(2r+1)^2}{4(2p+j)}-\frac{p}{j}(jt+s)^2}q^{\binom{p}{2}+p(jt+s-r)}
\sum_{m=1}^{p-1}(-1)^{m}q^{\binom{m+1}{2}+m(r-p)}\\
&\qquad \qquad \times (q^{m(jt+s-j)}-q^{-m(jt+s)})\Big (  j(-q^{m(2p+j)+p(2r+1)};q^{2p(2p+j)} )\\
&\qquad \qquad \qquad  \qquad  -  q^{m(2p+j)-m(2r+1)}j(-q^{-m(2p+j)+p(2r+1)};q^{2p(2p+j)})\Big )\\
 &= (-1)^{p}q^{-\frac{1}{8}+\frac{p(2r+1)^2}{4(2p+j)}}q^{\binom{p}{2}-2pj\binom{t}{2}-p(r-s+2st)-\frac{p}{j}s^2}
\sum_{m=1}^{p-1}(-1)^{m}q^{\binom{m+1}{2}+m(r-p)}\\
&\qquad \qquad \times (q^{m(jt+s-j)}-q^{-m(jt+s)})\Big (  j(-q^{m(2p+j)+p(2r+1)};q^{2p(2p+j)} )\\
&\qquad \qquad \qquad  \qquad  -  q^{m(2p+j)-m(2r+1)}j(-q^{-m(2p+j)+p(2r+1)};q^{2p(2p+j)})\Big ).
\end{align*}}%
Iterating gives us
{\allowdisplaybreaks \begin{align*}
 (q)_{\infty}^{3}&C_{2jt+2s,2r}^{(p,2p+j)}(q)\\
 &=(q)_{\infty}^{3}C_{2s,2r}^{(p,2p+j)}(q) \\
& \qquad + (-1)^{p}q^{-\frac{1}{8}+\frac{p(2r+1)^2}{4(2p+j)}}\sum_{i=1}^{t}q^{\binom{p}{2}-2pj\binom{i}{2}-p(r-s+2si)-\frac{p}{j}s^2}
\sum_{m=1}^{p-1}(-1)^{m}q^{\binom{m+1}{2}+m(r-p)}\\
&\qquad \qquad \times (q^{m(ji+s-j)}-q^{-m(ji+s)})\Big (  j(-q^{m(2p+j)+p(2r+1)};q^{2p(2p+j)} )\\
&\qquad \qquad \qquad  \qquad  -  q^{m(2p+j)-m(2r+1)}j(-q^{-m(2p+j)+p(2r+1)};q^{2p(2p+j)})\Big ).
\end{align*}}%
Rewriting the expression gives the result.
\end{proof}


\begin{proof}[Proof of Proposition \ref{proposition:quasiPeriodicity-step1}] In Proposition \ref{proposition:f-functionaleqn}, we set $(\ell,k)=(-2p,1)$.  For the first double-sum in (\ref{equation:stringQuasiHecke}), this gives
{\allowdisplaybreaks \begin{align*}
 &f_{1,2p+j,2p(2p+j)}(q^{1+k+r},-q^{p(2p+j+2r+1)};q)\\
 &\qquad =(-q^{1+k+r})^{-2p}(q^{p(2p+j+2r+1)})q^{\binom{-2p}{2}-2p(2p+j)}\\
&\qquad \qquad \times   f_{1,2p+j,2p(2p+j)}(q^{1+k+r+j},-q^{p(2p+j+2r+1)};q)\\
  &\qquad \qquad \qquad + \sum_{m=0}^{-2p-1}(-q^{1+k+r})^{m}q^{\binom{m}{2}}j(-q^{m(2p+j)}q^{p(2p+j+2r+1)};q^{2p(2p+j)})
  +j(q^{1+k+r};q)\\
&\qquad =q^{-p(2k+j)}   f_{1,2p+j,2p(2p+j)}(q^{1+k+r+j},-q^{p(2p+j+2r+1)};q)\\
  &\qquad \qquad \qquad - \sum_{m=-2p}^{-1}(-q^{1+k+r})^{m}q^{\binom{m}{2}}j(-q^{m(2p+j)}q^{p(2p+j+2r+1)};q^{2p(2p+j)}), 
\end{align*}}%
where in the last equality we have used the summation convention (\ref{equation:sumconvention}) and the fact that $j(q^{n};q)=~0$ for $n\in\mathbb{Z}$.  Simplifying some more yields
\begin{align}
 &f_{1,2p+j,2p(2p+j)}(q^{1+k+r},-q^{p(2p+j+2r+1)};q)
 \label{equation:quasi-preSum1}\\
 &\qquad =q^{-p(2k+j)}   f_{1,2p+j,2p(2p+j)}(q^{1+k+r+j},-q^{p(2p+j+2r+1)};q)\notag\\
  &\qquad \qquad \qquad - \sum_{m=1}^{2p}(-1)^{m}q^{-m(1+k+r)}q^{\binom{m+1}{2}}j(-q^{(p-m)(2p+j)+p(2r+1)};q^{2p(2p+j)}).
  \notag
\end{align}

\noindent For the second double-sum in (\ref{equation:stringQuasiHecke}), the same approach yields
{\allowdisplaybreaks \begin{align*}
&f_{1,2p+j,2p(2p+j)}(q^{k-r},-q^{p(2p+j-(2r+1))};q)\\
&\qquad = (-q^{k-r})^{-2p}(q^{p(2p+j-2r-1)})q^{\binom{-2p}{2}-2p(2p+j)}\\
&\qquad \qquad \times f_{1,2p+j,2p(2p+j)}(q^{k-r+j},-q^{p(2p+j-(2r+1))};q)\\
&\qquad \qquad \qquad + \sum_{m=0}^{-2p-1}(-q^{k-r})^{m}q^{\binom{m}{2}}j(-q^{m(2p+j)}q^{p(2p+j-(2r+1))};q^{2p(2p+j)})
  +j(q^{k-r};q)\\ 
&\qquad = q^{-p(2k+j)} f_{1,2p+j,2p(2p+j)}(q^{k-r+j},-q^{p(2p+j-(2r+1))};q)\\
&\qquad \qquad \qquad - \sum_{m=-2p}^{-1}(-q^{k-r})^{m}q^{\binom{m}{2}}j(-q^{m(2p+j)}q^{p(2p+j-(2r+1))};q^{2p(2p+j)})\\ 
&\qquad = q^{-p(2k+j)} f_{1,2p+j,2p(2p+j)}(q^{k-r+j},-q^{p(2p+j-(2r+1)};q)\\
&\qquad \qquad \qquad - \sum_{m=1}^{2p}(-1)^{m}q^{-m(k-r)}q^{\binom{m+1}{2}}j(-q^{(p-m)(2p+j)-p(2r+1)};q^{2p(2p+j)}).
\end{align*}}%

\noindent Applying (\ref{equation:j-flip}) to the theta function yields
{\allowdisplaybreaks \begin{align}
&f_{1,2p+j,2p(2p+j)}(q^{k-r},-q^{p(2p+j-(2r+1))};q)
\label{equation:quasi-preSum2}\\
&\qquad  = q^{-p(2k+j)} f_{1,2p+j,2p(2p+j)}(q^{k-r+j},-q^{p(2p+j-(2r+1)};q)\notag\\
&\qquad \qquad \qquad - \sum_{m=1}^{2p}(-1)^{m}q^{-m(k-r)}q^{\binom{m+1}{2}}j(-q^{(p+m)(2p+j)+p(2r+1)};q^{2p(2p+j)}).
\notag
\end{align}}%

Subtracting (\ref{equation:quasi-preSum2}) from (\ref{equation:quasi-preSum1}) gives
{\allowdisplaybreaks \begin{align*}
 &f_{1,2p+j,2p(2p+j)}(q^{1+k+r},-q^{p(2p+j+2r+1)};q)
 -f_{1,2p+j,2p(2p+j)}(q^{k-r},-q^{p(2p+j-(2r+1))};q)\\
 &\qquad =q^{-p(2k+j)} \Big  (   f_{1,2p+j,2p(2p+j)}(q^{1+k+r+j},-q^{p(2p+j+2r+1)};q) \\
&\qquad  \qquad \qquad \qquad - f_{1,2p+j,2p(2p+j)}(q^{k-r+j},-q^{p(2p+j-(2r+1))};q) \Big ) \\
&\qquad  \qquad \qquad \qquad - \sum_{m=1}^{2p}(-1)^{m}q^{-m(1+k+r)}q^{\binom{m+1}{2}}
j(-q^{(p-m)(2p+j)+p(2r+1)};q^{2p(2p+j)})\\
&\qquad  \qquad \qquad \qquad + \sum_{m=1}^{2p}(-1)^{m}q^{-m(k-r)}q^{\binom{m+1}{2}}j(-q^{(p+m)(2p+j)+p(2r+1)};q^{2p(2p+j)}).
\end{align*}}%
Using the string function notation (\ref{equation:stringQuasiHecke}), we obtain the stated result.
\end{proof}

\begin{proof}[Proof of Proposition \ref{proposition:quasiPeriodicity-step2}]  We first note that the terms for $m=p$ in the first and second sums in (\ref{equation:quasiPeriodString1}) cancel.  This follows from (\ref{equation:j-elliptic}) and simplifying:
\begin{align*}
- (-1)^{p}&q^{-p(1+k+r)}q^{\binom{p+1}{2}}
j(-q^{p(2r+1)};q^{2p(2p+j)})\\
&\qquad +(-1)^{p}q^{-p(k-r)}q^{\binom{p+1}{2}}j(-q^{2p(2p+j)+p(2r+1)};q^{2p(2p+j)})\\
&=- (-1)^{p}q^{-p(1+k+r)}q^{\binom{p+1}{2}}
j(-q^{p(2r+1)};q^{2p(2p+j)})\\
&\qquad \qquad +(-1)^{p}q^{-p(k-r)}q^{\binom{p+1}{2}}q^{-p(2r+1)}j(-q^{p(2r+1)};q^{2p(2p+j)})
=0.
\end{align*}
So we can omit the two terms for $m=p$ and break up the two summations in (\ref{equation:quasiPeriodString1}) to write
{\allowdisplaybreaks\begin{align}
 (q)_{\infty}^{3}&\mathcal{C}_{2k,2r}^{(p,2p+j)}(q)
 -(q)_{\infty}^{3}q^{-p(2k+j)} \mathcal{C}_{2k+2j,2r}^{(p,2p+j)}(q)
 \label{equation:quasiPeriodicInitBreak}\\
&\qquad   = - \sum_{m=1}^{p-1}(-1)^{m}q^{-m(1+k+r)}q^{\binom{m+1}{2}}
j(-q^{(p-m)(2p+j)+p(2r+1)};q^{2p(2p+j)})
\notag\\
& \qquad \qquad - \sum_{m=p+1}^{2p}(-1)^{m}q^{-m(1+k+r)}q^{\binom{m+1}{2}}
j(-q^{(p-m)(2p+j)+p(2r+1)};q^{2p(2p+j)})
\notag\\
&\qquad  \qquad  + \sum_{m=1}^{p-1}(-1)^{m}q^{-m(k-r)}q^{\binom{m+1}{2}}j(-q^{(p+m)(2p+j)+p(2r+1)};q^{2p(2p+j)})
\notag\\
&\qquad  \qquad  + \sum_{m=p+1}^{2p}(-1)^{m}q^{-m(k-r)}q^{\binom{m+1}{2}}j(-q^{(p+m)(2p+j)+p(2r+1)};q^{2p(2p+j)}).
\notag
\end{align}}%
In the first and third summation symbols of (\ref{equation:quasiPeriodicInitBreak}), we make the substitution $m\to p-m$: 
{\allowdisplaybreaks \begin{align}
 (q)_{\infty}^{3}&\mathcal{C}_{2k,2r}^{(p,2p+j)}(q)
 -(q)_{\infty}^{3}q^{-p(2k+j)} \mathcal{C}_{2k+2j,2r}^{(p,2p+j)}(q)
 \label{equation:quasiPeriodicInitBreak2}\\
&   = - \sum_{m=1}^{p-1}(-1)^{p-m}q^{(m-p)(1+k+r)}q^{\binom{p-m+1}{2}}
j(-q^{m(2p+j)+p(2r+1)};q^{2p(2p+j)})
\notag\\
&  \qquad - \sum_{m=p+1}^{2p}(-1)^{m}q^{-m(1+k+r)}q^{\binom{m+1}{2}}
j(-q^{(p-m)(2p+j)+p(2r+1)};q^{2p(2p+j)})
\notag\\
&  \qquad  + \sum_{m=1}^{p-1}(-1)^{p-m}q^{(m-p)(k-r)}q^{\binom{p-m+1}{2}}j(-q^{(2p-m)(2p+j)+p(2r+1)};q^{2p(2p+j)})
\notag\\
&  \qquad  + \sum_{m=p+1}^{2p}(-1)^{m}q^{-m(k-r)}q^{\binom{m+1}{2}}j(-q^{(p+m)(2p+j)+p(2r+1)};q^{2p(2p+j)}).
\notag
\end{align}}%
In the third summation of (\ref{equation:quasiPeriodicInitBreak2}), we use (\ref{equation:j-elliptic}) and simplify.  This gives
{\allowdisplaybreaks \begin{align}
 (q)_{\infty}^{3}&\mathcal{C}_{2k,2r}^{(p,2p+j)}(q)
 -(q)_{\infty}^{3}q^{-p(2k+j)} \mathcal{C}_{2k+2j,2r}^{(p,2p+j)}(q)
 \label{equation:quasiPeriodicInitBreak3}\\
&   = - (-1)^{p}q^{\binom{p}{2}-p(k+r)}\sum_{m=1}^{p-1}(-1)^{m}q^{\binom{m+1}{2}+m(k+r)-mp}
j(-q^{m(2p+j)+p(2r+1)};q^{2p(2p+j)})\notag\\
&  \qquad - \sum_{m=p+1}^{2p}(-1)^{m}q^{-m(1+k+r)}q^{\binom{m+1}{2}}
j(-q^{(p-m)(2p+j)+p(2r+1)};q^{2p(2p+j)})\notag\\
&  \qquad  + (-1)^{p}q^{\binom{p}{2}-p(k+r)}\sum_{m=1}^{p-1}(-1)^{m}q^{\binom{m+1}{2}+m(k-r+p+j-1)}j(-q^{-m(2p+j)+p(2r+1)};q^{2p(2p+j)})\notag \\
&  \qquad  + \sum_{m=p+1}^{2p}(-1)^{m}q^{-m(k-r)}q^{\binom{m+1}{2}}j(-q^{(p+m)(2p+j)+p(2r+1)};q^{2p(2p+j)}).
\notag
\end{align}}%
In the second and fourth summations of (\ref{equation:quasiPeriodicInitBreak3}) we make the substitution $m\to p+m$:
{\allowdisplaybreaks \begin{align}
 (q)_{\infty}^{3}&\mathcal{C}_{2k,2r}^{(p,2p+j)}(q)
 -(q)_{\infty}^{3}q^{-p(2k+j)} \mathcal{C}_{2k+2j,2r}^{(p,2p+j)}(q) 
 \label{equation:quasiPeriodicInitBreak4}\\
& = - (-1)^{p}q^{\binom{p}{2}-p(k+r)}\sum_{m=1}^{p-1}(-1)^{m}q^{\binom{m+1}{2}+m(k+r)-mp}
j(-q^{m(2p+j)+p(2r+1)};q^{2p(2p+j)})\notag\\
&  \qquad - \sum_{m=1}^{p}(-1)^{p+m}q^{-(m+p)(1+k+r)}q^{\binom{p+m+1}{2}}
j(-q^{-m(2p+j)+p(2r+1)};q^{2p(2p+j)})\notag\\
&  \qquad  + (-1)^{p}q^{\binom{p}{2}-p(k+r)}\sum_{m=1}^{p-1}(-1)^{m}q^{\binom{m+1}{2}+m(k-r+p+j-1)}
j(-q^{-m(2p+j)+p(2r+1)};q^{2p(2p+j)})\notag\\
&  \qquad  + \sum_{m=1}^{p}(-1)^{p+m}q^{-(m+p)(k-r)}q^{\binom{p+m+1}{2}}j(-q^{(2p+m)(2p+j)+p(2r+1)};q^{2p(2p+j)}).
\notag
\end{align}}%
In the fourth summation of (\ref{equation:quasiPeriodicInitBreak4}) we use (\ref{equation:j-elliptic}).  This brings us to
{\allowdisplaybreaks \begin{align}
 (q)_{\infty}^{3}&\mathcal{C}_{2k,2r}^{(p,2p+j)}(q)
 -(q)_{\infty}^{3}q^{-p(2k+j)} \mathcal{C}_{2k+2j,2r}^{(p,2p+j)}(q) 
 \label{equation:quasiPeriodicInitBreak5}\\
& = - (-1)^{p}q^{\binom{p}{2}-p(k+r)}\sum_{m=1}^{p-1}(-1)^{m}q^{\binom{m+1}{2}+m(k+r-p)}
j(-q^{m(2p+j)+p(2r+1)};q^{2p(2p+j)})\notag\\
&  \qquad - (-1)^{p}q^{\binom{p}{2}-p(k+r)}\sum_{m=1}^{p}(-1)^{m}q^{\binom{m+1}{2}-m(k+r-p+1)}
j(-q^{-m(2p+j)+p(2r+1)};q^{2p(2p+j)})\notag\\
&  \qquad  + (-1)^{p}q^{\binom{p}{2}-p(k+r)}\sum_{m=1}^{p-1}(-1)^{m}q^{\binom{m+1}{2}- m(r-k-p-j+1)}
j(-q^{-m(2p+j)+p(2r+1)};q^{2p(2p+j)})\notag\\
&  \qquad  + (-1)^{p}q^{\binom{p}{2}-p(k+r)}\sum_{m=1}^{p}(-1)^{m}q^{\binom{m+1}{2}+m(r-k-p-j)}
j(-q^{m(2p+j)+p(2r+1)};q^{2p(2p+j)}).\notag
\end{align}}%

Now, in the second and fourth summations in (\ref{equation:quasiPeriodicInitBreak5}), we see that the terms for $m=p$ cancel:
\begin{align*}
- &q^{\binom{p}{2}-p(k+r)}q^{\binom{p+1}{2}-p(k+r-p+1)}
j(-q^{-p(2p+j)+p(2r+1)};q^{2p(2p+j)})\\
&\qquad + q^{\binom{p}{2}-p(k+r)}q^{\binom{p+1}{2}+p(r-k-p-j)}j(-q^{p(2p+j)+p(2r+1)};q^{2p(2p+j)})\\
&=- q^{\binom{p}{2}-p(k+r)}q^{\binom{p+1}{2}-p(k+r-p+1)}
j(-q^{-p(2p+j)+p(2r+1)};q^{2p(2p+j)})\\
&\qquad + q^{\binom{p}{2}-p(k+r)}q^{\binom{p+1}{2}+p(r-k-p-j)}q^{p(2p+j)-p(2r+1)}j(-q^{-p(2p+j)+p(2r+1)};q^{2p(2p+j)})
=0,
\end{align*}
where we used (\ref{equation:j-elliptic}) and simplified.  In (\ref{equation:quasiPeriodicInitBreak5}), we group together the first and fourth summations as well as the second and third summations to obtain
{\allowdisplaybreaks \begin{align*}
 (q)_{\infty}^{3}&\mathcal{C}_{2k,2r}^{(p,2p+j)}(q)
 -(q)_{\infty}^{3}q^{-p(2k+j)} \mathcal{C}_{2k+2j,2r}^{(p,2p+j)}(q) \\
& = (-1)^{p}q^{\binom{p}{2}-p(k+r)}\sum_{m=1}^{p-1}(-1)^{m}q^{\binom{m+1}{2}+m(r-k-p-j)}(1-q^{m(2k+j)})\\
&\qquad \qquad \times  j(-q^{m(2p+j)+p(2r+1)};q^{2p(2p+j)} )\\
&  \qquad  - (-1)^{p}q^{\binom{p}{2}-p(k+r)}\sum_{m=1}^{p-1}(-1)^{m}q^{\binom{m+1}{2}- m(r+k-p+1)}(1-q^{m(2k+j)})\\
&\qquad \qquad \times j(-q^{-m(2p+j)+p(2r+1)};q^{2p(2p+j)}).
\end{align*}}%

\noindent Combining the remaining two sums yields
\begin{align*}
 (q)_{\infty}^{3}&\mathcal{C}_{2k,2r}^{(p,2p+j)}(q)
 -(q)_{\infty}^{3}q^{-p(2k+j)} \mathcal{C}_{2k+2j,2r}^{(p,2p+j)}(q) \\
& = (-1)^{p}q^{\binom{p}{2}-p(k+r)}\sum_{m=1}^{p-1}(-1)^{m}q^{\binom{m+1}{2}+m(r-k-p-j)}(1-q^{m(2k+j)})\\
&\qquad \times \Big (  j(-q^{m(2p+j)+p(2r+1)};q^{2p(2p+j)} )  -  q^{m(2p+j)-m(2r+1)}
j(-q^{-m(2p+j)+p(2r+1)};q^{2p(2p+j)})\Big ).
\end{align*}
Rewriting the exponents gives the result.
\end{proof}

\section{A cross-spin identity for $j/p$--level string functions with $j$ odd}
\label{section:crossSpin-j-Odd}

In this section we prove the cross-spin identity in Theorem \ref{theorem:crossSpin-j-Odd}.  At the end of the section we will give a new proof of the old $1/2$-level cross-spin identities \cite[Corollary $5.1$]{BoMo2024}:
\begin{align*}
(q)_{\infty}^3\mathcal{C}_{1,2r-1}^{(2,5)}(q)
&=-q^{3-2r}(q)_{\infty}^{3}\mathcal{C}_{0,4-2r}^{(2,5)}(q)
 +q^{1-r}j(q^{2r};q^{5}),
\end{align*}
and we will give new $1/3$-level cross-spin identities:
\begin{align*}
(q)_{\infty}^3\mathcal{C}_{2i-1,2r-1}^{(3,7)}(q)
&=q^{3(1+i-r)}(q)_{\infty}^{3}\mathcal{C}_{2i-2,6-2r}^{(3,7)}(q)\\
&\qquad  -q^{1+2(i-r)}
\frac{j(q^{7+2r};q^{14})j(q^{4r};q^{28})}{J_{28}}
+q^{i-r}
\frac{j(q^{2r};q^{14})j(q^{14+4r};q^{28})}{J_{28}}.
\end{align*}

\begin{proof}[Proof of Theorem \ref{theorem:crossSpin-j-Odd}]  Using the Hecke-type form for string functions (\ref{equation:modStringFnHeckeForm}), we write
\begin{align*}
(q)_{\infty}^3\mathcal{C}_{2i-1,2r-1}^{(p,2p+j)}(q)
&=f_{1,2p+j,2p(2p+j)}(q^{i+r},-q^{p(2p+j+2r)};q)\\
&\qquad -f_{1,2p+j,2p(2p+j)}(q^{i-r},-q^{p(2p+j-2r)};q).
\end{align*}
In the functional equation for Hecke-type double-sums (\ref{equation:Gen1}), we set $(\ell,k)=(p,-1)$.  This gives
\begin{align*}
&f_{1,2p+j,2p(2p+j)}(q^{i+r},-q^{p(2p+j+2r)};q)\\
&\qquad =(-1)^{p}q^{p(i+r)-p(2p+j+2r)+\binom{p}{2}-p(2p+j)+2p(2p+j)}
 f_{1,2p+j,2p(2p+j)}(q^{-p-j+i+r},-q^{p(2r)};q)\\
&\qquad \qquad + \sum_{m=0}^{p-1}(-1)^{m}q^{m(i+r)}q^{\binom{m}{2}}j(-q^{m(2p+j)}q^{p(2p+j+2r)};q^{2p(2p+j)})\\
&\qquad \qquad + \sum_{m=0}^{-2}q^{mp(2p+j+2r)}q^{2p(2p+j)\binom{m}{2}}j(q^{m(2p+j)}q^{i+r};q).
\end{align*}
We then simplify and use the fact that $j(q^n;q)=0$ for $n\in\mathbb{Z}$.  This yields
\begin{align*}
&f_{1,2p+j,2p(2p+j)}(q^{i+r},-q^{p(2p+j+2r)};q)\\
&\qquad =(-1)^{p}q^{p(i-r)+\binom{p}{2}}
 f_{1,2p+j,2p(2p+j)}(q^{-p-j+i+r},-q^{2pr};q)\\
&\qquad \qquad + \sum_{m=0}^{p-1}(-1)^{m}q^{m(i+r)}q^{\binom{m}{2}}
j(-q^{m(2p+j)+p(2p+j+2r)};q^{2p(2p+j)}).
\end{align*}

\noindent Now in the functional equation (\ref{equation:Gen1}), we specialize $(\ell,k)=(p,0)$ to get
\begin{align*}
f_{1,2p+j,2p(2p+j)}&(q^{i-r},-q^{p(2p+j-2r)};q)\\
&=(-1)^{p}q^{p(i-r)+\binom{p}{2}}f_{1,2p+j,2p(2p+j)}(q^{p+i-r},-q^{p(4p+2j-2r)};q)\\
&\qquad + \sum_{m=0}^{p-1}(-1)^{m}q^{m(i-r)}q^{\binom{m}{2}}
j(-q^{m(2p+j)+p(2p+j-2r)};q^{2p(2p+j)}).
\end{align*}

If we consider the system of equations
\begin{equation*}
1+\frac{m+\ell}{2}=p+i-r, \ \ \frac{m-\ell}{2}=-p-j+i+r,
\end{equation*}
then
\begin{gather*}
m=2i-j-1, \ \ \ell=2p-2r+j-1.
\end{gather*}
Using (\ref{equation:modStringFnHeckeForm}), we then find that
\begin{align*}
(q)_{\infty}^{3}\mathcal{C}_{2i-1-j,2p-2r+j-1}^{(p,2p+j)}(q)
&=f_{1,2p+j,2p(2p+j)}(q^{p-r+i},-q^{p(2p+j+2p-2r+j)};q)\\
&\qquad -f_{1,2p+j,2p(2p+j)}(q^{-j-p+i+r},-q^{p(2p+j-2p+2r-j)};q).
\end{align*}
Hence we can write
{\allowdisplaybreaks \begin{align*}
(q)_{\infty}^3\mathcal{C}_{2i-1,2r-1}^{(p,2p+j)}(q)
&=(-1)^{p+1}q^{p(i-r)+\binom{p}{2}}(q)_{\infty}^{3}\mathcal{C}_{2i-1-j,2p-2r+j-1}^{(p,2p+j)}(q)\\
&\qquad +\sum_{m=0}^{p-1}(-1)^{m}q^{m(i+r)}q^{\binom{m}{2}}
j(-q^{m(2p+j)+p(2p+j+2r)};q^{2p(2p+j)})\\
&\qquad -\sum_{m=0}^{p-1}(-1)^{m}q^{m(i-r)}q^{\binom{m}{2}}
j(-q^{m(2p+j)+p(2p+j-2r)};q^{2p(2p+j)}).
\end{align*}}%

Let us rework the above expression.  We first note that the $m=0$ terms cancel.  Indeed, the functional equation (\ref{equation:j-flip}) easily gives
\begin{align*}
&j(-q^{p(2p+j+2r)};q^{2p(2p+j)})
-j(-q^{p(2p+j-2r)};q^{2p(2p+j)})\\
&\qquad =j(-q^{p(2p+j+2r)};q^{2p(2p+j)})
-j(-q^{p(2p+j+2r)};q^{2p(2p+j)})=0.
\end{align*}
Using (\ref{equation:j-flip}) again gives us
\begin{align*}
(q)_{\infty}^3\mathcal{C}_{2i-1,2r-1}^{(p,2p+j)}(q)
&=(-1)^{p+1}q^{p(i-r)+\binom{p}{2}}(q)_{\infty}^{3}\mathcal{C}_{2i-1-j,2p-2r+j-1}^{(p,2p+j)}(q)\\
&\qquad +\sum_{m=1}^{p-1}(-1)^{m}q^{m(i+r)}q^{\binom{m}{2}}
j(-q^{(p-m)(2p+j)-2pr};q^{2p(2p+j)})\\
&\qquad -\sum_{m=1}^{p-1}(-1)^{m}q^{m(i-r)}q^{\binom{m}{2}}
j(-q^{(p-m)(2p+j)+2pr};q^{2p(2p+j)}).
\end{align*}
Reversing summation by replacing $m$ with $p-m$ and simplifying yields
\begin{align*}
(q)_{\infty}^3\mathcal{C}_{2i-1,2r-1}^{(p,2p+j)}(q)
&=(-1)^{p+1}q^{p(i-r)+\binom{p}{2}}(q)_{\infty}^{3}\mathcal{C}_{2i-1-j,2p-2r+j-1}^{(p,2p+j)}(q)\\
&\qquad +(-1)^{p}q^{\binom{p}{2}+p(i+r)}
\sum_{m=1}^{p-1}(-1)^{m}q^{\binom{m+1}{2}-m(i+p+r)}
j(-q^{m(2p+j)-2pr};q^{2p(2p+j)})\\
&\qquad -(-1)^{p}q^{\binom{p}{2}+p(i-r)}
\sum_{m=1}^{p-1}(-1)^{m}q^{\binom{m+1}{2}-m(i+p- r)}
j(-q^{m(2p+j)+2pr};q^{2p(2p+j)}).
\end{align*}
Combining terms yields our final form:
\begin{align*}
(q)_{\infty}^3&\mathcal{C}_{2i-1,2r-1}^{(p,2p+j)}(q)\\
&=(-1)^{p+1}q^{p(i-r)+\binom{p}{2}}(q)_{\infty}^{3}\mathcal{C}_{2i-1-j,2p-2r+j-1}^{(p,2p+j)}(q)\\
&\qquad +(-1)^{p}q^{\binom{p}{2}+p(i+r)}
\sum_{m=1}^{p-1}(-1)^{m}q^{\binom{m+1}{2}-m(i+p+r)}\\
&\qquad \qquad \times
\left ( j(-q^{m(2p+j)-2pr};q^{2p(2p+j)})
-q^{2r(m-p)}j(-q^{m(2p+j)+2pr};q^{2p(2p+j)})\right ). \qedhere
\end{align*}
\end{proof}


\subsection{Cross-spin identities for $1/2$-level string functions}
Let us check this against known examples.  We try $p=2,j=1,i=1$, $r\in\{ 1,2\}$.  This gives
\begin{align*}
(q)_{\infty}^3\mathcal{C}_{1,2r-1}^{(2,5)}(q)
=-q^{3-2r}(q)_{\infty}^{3}\mathcal{C}_{0,4-2r}^{(2,5)}(q)
 -q^{1+r}
\left ( j(-q^{5-4r};q^{20})-q^{-2r}j(-q^{5+4r};q^{20})\right ) .
\end{align*}
Using (\ref{equation:j-flip}) and rearranging terms gives
\begin{align*}
(q)_{\infty}^3\mathcal{C}_{1,2r-1}^{(2,5)}(q)
&=-q^{3-2r}(q)_{\infty}^{3}\mathcal{C}_{0,4-2r}^{(2,5)}(q)
 +q^{1-r}
\left (j(-q^{5+4r};q^{20})- q^{2r}j(-q^{15+4r};q^{20})\right ).
\end{align*}
A final use of (\ref{equation:jsplit-m2}) yields the desired form
\begin{align*}
(q)_{\infty}^3\mathcal{C}_{1,2r-1}^{(2,5)}(q)
&=-q^{3-2r}(q)_{\infty}^{3}\mathcal{C}_{0,4-2r}^{(2,5)}(q)
 +q^{1-r}j(q^{2r};q^{5}).
\end{align*}


\subsection{New cross-spin identities for $1/3$-level string functions}

Here we specialize to 
{\allowdisplaybreaks \begin{align*}
(q)_{\infty}^3\mathcal{C}_{2i-1,2r-1}^{(3,7)}(q)
&=q^{3(i-r)+3}(q)_{\infty}^{3}\mathcal{C}_{2i-2,6-2r}^{(3,7)}(q)\\
&\qquad -q^{3+3(i+r)}
\sum_{m=1}^{2}(-1)^{m}q^{\binom{m+1}{2}-m(3+i+r)}\\
&\qquad \qquad \times
\left ( j(-q^{7m-6r};q^{42})
-q^{2r(m-3)}j(-q^{7m+6r};q^{42})\right ) .
\end{align*}}%
Expanding the sum gives
\begin{align*}
(q)_{\infty}^3\mathcal{C}_{2i-1,2r-1}^{(3,7)}(q)
&=q^{3(i-r)+3}(q)_{\infty}^{3}\mathcal{C}_{2i-2,6-2r}^{(3,7)}(q)\\
&\qquad -q^{3+3(i+r)}
\Big ( -q^{-2-i-r}
\left ( j(-q^{7-6r};q^{42})
-q^{-4r}j(-q^{7+6r};q^{42})\right )\\
&\qquad +q^{-3-2i-2r}
\left ( j(-q^{14-6r};q^{42})
-q^{-2r}j(-q^{14+6r};q^{42})\right )
\Big ),
\end{align*}
which simplifies to
{\allowdisplaybreaks \begin{align*}
(q)_{\infty}^3\mathcal{C}_{2i-1,2r-1}^{(3,7)}(q)
&=q^{3(i-r)+3}(q)_{\infty}^{3}\mathcal{C}_{2i-2,6-2r}^{(3,7)}(q)\\
&\qquad  +q^{1+2i+2r}
\left ( j(-q^{7-6r};q^{42})
-q^{-4r}j(-q^{7+6r};q^{42})\right )\\
&\qquad -q^{i+r}
\left ( j(-q^{14-6r};q^{42})
-q^{-2r}j(-q^{14+6r};q^{42})\right ).
\end{align*}}%

We rewrite the theta coefficients for the quintuple product identity.  
Using (\ref{equation:j-flip}), (\ref{equation:j-elliptic}), and then (\ref{equation:quintuple}) yields
\begin{align*}
j(-q^{7-6r};q^{42})-q^{-4r}j(-q^{7+6r};q^{42})
&=j(-q^{35+6r};q^{42})-q^{7+2r}j(-q^{49+6r};q^{42})\\
&=\frac{j(q^{7+2r};q^{14})j(q^{28+4r};q^{28})}{J_{28}}.
\end{align*}
We use (\ref{equation:j-flip}) and then (\ref{equation:quintuple}) to get
\begin{align*}
j(-q^{14-6r};q^{42})
-q^{-2r}j(-q^{14+6r};q^{42})
&=j(-q^{14-6r};q^{42})
-q^{-2r}j(-q^{28-6r};q^{42})\\
&=\frac{j(q^{-2r};q^{14})j(q^{14-4r};q^{28})}{J_{28}}.
\end{align*}

Combining expressions gives
\begin{align*}
(q)_{\infty}^3\mathcal{C}_{2i-1,2r-1}^{(3,7)}(q)
&=q^{3(i-r)+3}(q)_{\infty}^{3}\mathcal{C}_{2i-2,6-2r}^{(3,7)}(q)\\
&\qquad  +q^{1+2i+2r}
\frac{j(q^{7+2r};q^{14})j(q^{28+4r};q^{28})}{J_{28}}
-q^{i+r}
\frac{j(q^{-2r};q^{14})j(q^{14-4r};q^{28})}{J_{28}}.
\end{align*}
Rewriting the theta functions using (\ref{equation:j-flip}) and (\ref{equation:j-elliptic}) yields
\begin{align*}
(q)_{\infty}^3\mathcal{C}_{2i-1,2r-1}^{(3,7)}(q)
&=q^{3(1+i-r)}(q)_{\infty}^{3}\mathcal{C}_{2i-2,6-2r}^{(3,7)}(q)\\
&\qquad  -q^{1+2(i-r)}
\frac{j(q^{7+2r};q^{14})j(q^{4r};q^{28})}{J_{28}}
+q^{i-r}
\frac{j(q^{2r};q^{14})j(q^{14+4r};q^{28})}{J_{28}}.
\end{align*}


\section{The polar-finite decompositions for general positive admissible-level characters}
\label{section:polarFinite}
We prove Theorem \ref{theorem:generalPolarFinite}.    We begin with  propositions whose proofs we delay until the end of the section.
\begin{proposition} \label{proposition:polarFinitePreAppell}
We have
{\allowdisplaybreaks \begin{align}
(q)_{\infty}^3&\chi_{2r}^{(p,2p+j)}(z;q)
\label{equation:polarFinitePreFinal}\\
&=   (q)_{\infty}^3 \sum_{s=0}^{j-1}z^{-s}q^{\frac{p}{j}s^2}C_{2s,2r}^{(p,2p+j)}(q)
    j(-z^{-j}q^{p(2s+j)};q^{2pj})\notag\\
&\qquad +      (-1)^{p}q^{-\frac{1}{8}+\frac{p(2r+1)^2}{4(2p+j)}}
\sum_{s=0}^{j-1}q^{\binom{p}{2}-p(r-s)}z^{-s}
\times \sum_{m=1}^{p-1}(-1)^{m}q^{\binom{m+1}{2}+m(r-p)}
\notag\\
&\qquad \qquad  \qquad  \times \Big (  j(-q^{m(2p+j)+p(2r+1)};q^{2p(2p+j)} )
\notag\\
&\qquad \qquad  \qquad \qquad \qquad  -  q^{m(2p+j)-m(2r+1)}j(-q^{-m(2p+j)+p(2r+1)};q^{2p(2p+j)})\Big )
\notag\\
&\qquad \qquad  \qquad  \times 
\sum_{t\in\mathbb{Z}}q^{pjt^2+2pst}z^{-jt}
\sum_{i=1}^{t}q^{-pji(i-1)-2psi}  \times (q^{m(ji+s-j)}-q^{-m(ji+s)}).
\notag
\end{align} }%
\end{proposition}

What we will do is to write (\ref{equation:polarFinitePreFinal}) in terms of Appell functions; however, we only need to focus on the last line:
\begin{equation}
\sum_{t\in\mathbb{Z}}q^{pjt^2+2pst}z^{-jt}
\sum_{i=1}^{t}q^{-pji(i-1)-2psi}  \times (q^{m(ji+s-j)}-q^{-m(ji+s)}).
\label{equation:initSumOver_i_t}
\end{equation}
This leads us to another proposition.

\begin{proposition}\label{proposition:initSumOver_i_tPreAppellFinal} We have
{\allowdisplaybreaks \begin{align}
\sum_{t\in\mathbb{Z}}&q^{pjt^2+2pst}z^{-jt}
\sum_{i=1}^{t}q^{-pji(i-1)-2psi}  \times (q^{m(ji+s-j)}-q^{-m(ji+s)})
\label{equation:initSumOver_i_tPreAppellFinal}\\
&=-q^{-m(j-s)}j(-q^{p(j-2s)}z^{j};q^{2jp})m(-q^{-jm+2ps},-q^{p(j-2s)}z^{j};q^{2jp})\notag \\
&\qquad + q^{-ms} j(-q^{p(j-2s)}z^{j};q^{2jp})m(-q^{jm+2ps},-q^{p(j-2s)}z^{j};q^{2jp}).\notag
\end{align}}%
\end{proposition}

\noindent Now we have the pieces to prove our general polar-finite decomposition.
\begin{proof}[Proof of Theorem \ref{theorem:generalPolarFinite}]
Inserting (\ref{equation:initSumOver_i_tPreAppellFinal}) into (\ref{equation:polarFinitePreFinal}) yields
{\allowdisplaybreaks \begin{align*}
&\chi_{2r}^{(p,2p+j)} (z;q)\\
&\qquad =\sum_{s=0}^{j-1}z^{-s}q^{\frac{p}{j}s^2}C_{2s,2r}^{(p,2p+j)}(q)j(-z^{-j}q^{p(2s+j)};q^{2pj})\\
&\qquad \qquad +\frac{1}{(q)_{\infty}^3}\sum_{s=0}^{j-1}
(-1)^{p}q^{-\frac{1}{8}+\frac{p(2r+1)^2}{4(2p+j)}}q^{\binom{p}{2}-p(r-s)}z^{-s}
 \sum_{m=1}^{p-1}(-1)^{m}q^{\binom{m+1}{2}+m(r-p)}\\
&\qquad   \qquad  \qquad \times \Big (  j(-q^{m(2p+j)+p(2r+1)};q^{2p(2p+j)} )\\
&\qquad  \qquad \qquad  \qquad -  q^{m(2p+j)-m(2r+1)}j(-q^{-m(2p+j)+p(2r+1)};q^{2p(2p+j)})\Big )\\
&\qquad  \qquad  \qquad \times
\Big (-q^{-m(j-s)}j(-q^{p(j-2s)}z^{j};q^{2jp})m(-q^{-jm+2ps},-q^{p(j-2s)}z^{j};q^{2jp})\\
&\qquad  \qquad \qquad \qquad + q^{-ms} j(-q^{p(j-2s)}z^{j};q^{2jp})m(-q^{jm+2ps},-q^{p(j-2s)}z^{j};q^{2jp})\Big ).
\end{align*}}%
Pulling out a common theta function and using (\ref{equation:mxqz-flip}), we get
{\allowdisplaybreaks \begin{align*}
&\chi_{2r}^{(p,2p+j)} (z;q)\\
&=\sum_{s=0}^{j-1}z^{-s}q^{\frac{p}{j}s^2}C_{2s,2r}^{(p,2p+j)}(q)j(-q^{p(j-2s)}z^{j};q^{2jp})\\
&\qquad +\frac{1}{(q)_{\infty}^3}\sum_{s=0}^{j-1}
(-1)^{p}q^{-\frac{1}{8}+\frac{p(2r+1)^2}{4(2p+j)}}q^{\binom{p}{2}-p(r-s)}z^{-s}j(-q^{p(j-2s)}z^{j};q^{2jp})
\times  \sum_{m=1}^{p-1}(-1)^{m}q^{\binom{m+1}{2}+m(r-p)} \\
&\qquad \qquad \times \Big (  j(-q^{m(2p+j)+p(2r+1)};q^{2p(2p+j)} )
-  q^{m(2p+j)-m(2r+1)}j(-q^{-m(2p+j)+p(2r+1)};q^{2p(2p+j)})\Big )\\
&\qquad \qquad   \times 
\Big (q^{ms-2ps}m(-q^{jm-2ps},-q^{-p(j-2s)}z^{-j};q^{2jp})
+ q^{-ms} m(-q^{jm+2ps},-q^{p(j-2s)}z^{j};q^{2jp})\Big ).
\end{align*}}%
Applying (\ref{equation:mxqz-fnq-z}) to the first Appell function gives the result.
\end{proof}

\begin{proof}[Proof of Proposition \ref{proposition:polarFinitePreAppell}]
We begin by recalling the defining equation for string functions (\ref{equation:fourcoefexp}) with the specialization $(p,p^{\prime})=(p,2p+j)$, $1\le j \le p-1$, $\ell=2r$:
\begin{equation*}
\chi_{2r}^{(p,2p+j)} (z;q)=\sum_{k\in\mathbb{Z}}
C_{2k,2r}^{(p,2p+j)}(q) q^{\frac{p}{j}k^2}z^{-k}.
\end{equation*} 
Because our quasi-periodicity formula of Theorem \ref{theorem:generalQuasiPeriodicity} has quasi-period $2j$, we need to consider $k \pmod j$.  We write
\begin{equation*}
    \chi_{2r}^{(p,2p+j)}(z;q)
    =\sum_{t\in\mathbb{Z}}
    \sum_{s=0}^{j-1}C_{2jt+2s,2r}^{(p,2p+j)}(q)q^{\frac{p}{j}(jt+s)^2}z^{-jt-s}.
\end{equation*}
Multiplying by $(q)_{\infty}^3$ and using Theorem \ref{theorem:generalQuasiPeriodicity}, this reads
{\allowdisplaybreaks \begin{align}
(q)_{\infty}^3&\chi_{2r}^{(p,2p+j)}(z;q)
\label{equation:prePolarFiniteV1}\\
& =\sum_{t\in\mathbb{Z}}
    \sum_{s=0}^{j-1}(q)_{\infty}^{3}C_{2s,2r}^{(p,2p+j)}(q)
    q^{\frac{p}{j}(jt+s)^2}z^{-jt-s}
    \notag\\
&\qquad   +\sum_{t\in\mathbb{Z}}
    \sum_{s=0}^{j-1} (-1)^{p}q^{-\frac{1}{8}+\frac{p(2r+1)^2}{4(2p+j)}}q^{\binom{p}{2}-p(r-s)-\frac{p}{j}s^2}q^{\frac{p}{j}(jt+s)^2}z^{-jt-s}\sum_{i=1}^{t}q^{-2pj\binom{i}{2}-2psi}\notag\\
&\qquad \qquad \times \sum_{m=1}^{p-1}(-1)^{m}q^{\binom{m+1}{2}+m(r-p)} 
\notag\\
&\qquad \qquad \qquad \times (q^{m(ji+s-j)}-q^{-m(ji+s)})\Big (  j(-q^{m(2p+j)+p(2r+1)};q^{2p(2p+j)} )
\notag\\
&\qquad \qquad \qquad \qquad   -  q^{m(2p+j)-m(2r+1)}j(-q^{-m(2p+j)+p(2r+1)};q^{2p(2p+j)})\Big ).\notag
\end{align}}%

For the first double-sum over $t$ and $s$ in (\ref{equation:prePolarFiniteV1}), we use the Jacobi triple product identity (\ref{equation:JTPid}) to obtain
{\allowdisplaybreaks \begin{align*}
   (q)_{\infty}^3\sum_{s=0}^{j-1}\sum_{t\in\mathbb{Z}}
    C_{2s,2r}^{(p,2p+j)}(q)q^{\frac{p}{j}(jt+s)^2}z^{-jt-s}
    &=(q)_{\infty}^3\sum_{s=0}^{j-1}z^{-s}q^{\frac{p}{j}s^2}C_{2s,2r}^{(p,2p+j)}(q)
    \sum_{t\in\mathbb{Z}}
    q^{2pj\binom{t}{2}+pjt+2pst}z^{-jt}\\
&=(q)_{\infty}^3\sum_{s=0}^{j-1}z^{-s}q^{\frac{p}{j}s^2}C_{2s,2r}^{(p,2p+j)}(q)
    j(-z^{-j}q^{p(2s+j)};q^{2pj}).
\end{align*}}%
For the second double-sum over $t$ and $s$ in (\ref{equation:prePolarFiniteV1}), rewriting the exponents yields
{\allowdisplaybreaks \begin{align*}
(-1)^{p}&q^{-\frac{1}{8}+\frac{p(2r+1)^2}{4(2p+j)}}q^{\binom{p}{2}}
\sum_{s=0}^{j-1}q^{-p(r-s)}z^{-s}
\sum_{t\in\mathbb{Z}}q^{pjt^2+2pst}z^{-jt}
\sum_{i=1}^{t}q^{-pji(i-1)-2psi}\\
& \times \sum_{m=1}^{p-1}(-1)^{m}q^{\binom{m+1}{2}+m(r-p)}
 \times (q^{m(ji+s-j)}-q^{-m(ji+s)})\\
& \times \Big (  j(-q^{m(2p+j)+p(2r+1)};q^{2p(2p+j)} )
-  q^{m(2p+j)-m(2r+1)}j(-q^{-m(2p+j)+p(2r+1)};q^{2p(2p+j)})\Big ).
\end{align*}}%
Interchanging the sums in order to isolate the sums over $i$ and $t$ brings us to 
{\allowdisplaybreaks \begin{align*} 
(-1)^{p}&q^{-\frac{1}{8}+\frac{p(2r+1)^2}{4(2p+j)}}\sum_{s=0}^{j-1}q^{\binom{p}{2}-p(r-s)}z^{-s}
\times \sum_{m=1}^{p-1}(-1)^{m}q^{\binom{m+1}{2}+m(r-p)}\\
&  \times \Big (  j(-q^{m(2p+j)+p(2r+1)};q^{2p(2p+j)} )
-  q^{m(2p+j)-m(2r+1)}j(-q^{-m(2p+j)+p(2r+1)};q^{2p(2p+j)})\Big )\\
&  \times 
\sum_{t\in\mathbb{Z}}q^{pjt^2+2pst}z^{-jt}
\sum_{i=1}^{t}q^{-pji(i-1)-2psi}  \times (q^{m(ji+s-j)}-q^{-m(ji+s)}).
\end{align*}}%
The result follows.
\end{proof}


\begin{proof}[Proof of Proposition \ref{proposition:initSumOver_i_tPreAppellFinal}]
We will focus on the last line.  For $t=0$, the inner sum vanishes by the summation convention (\ref{equation:sumconvention}), so we can write
{\allowdisplaybreaks \begin{align*}
\sum_{t\in\mathbb{Z}}&q^{pjt^2+2pst}z^{-jt}
\sum_{i=1}^{t}q^{-pji(i-1)-2psi}  \times (q^{m(ji+s-j)}-q^{-m(ji+s)})\\
&=\sum_{t\ge 1}\sum_{i=1}^{t}q^{pjt^2+2pst}z^{-jt}
q^{-pji(i-1)-2psi}   (q^{m(ji+s-j)}-q^{-m(ji+s)})\\
&\quad \qquad + \sum_{t\le -1}\sum_{i=1}^{t}q^{pjt^2+2pst}z^{-jt}
q^{-pji(i-1)-2psi}   (q^{m(ji+s-j)}-q^{-m(ji+s)}).
\end{align*}}%
Again using the summation convention (\ref{equation:sumconvention}), we rewrite the second double-sum
{\allowdisplaybreaks \begin{align*}
\sum_{t\in\mathbb{Z}}&q^{pjt^2+2pst}z^{-jt}
\sum_{i=1}^{t}q^{-pji(i-1)-2psi}  \times (q^{m(ji+s-j)}-q^{-m(ji+s)})\\
&=\sum_{t\ge 1}\sum_{i=1}^{t}q^{pjt^2+2pst}z^{-jt}
q^{-pji(i-1)-2psi}  (q^{m(ji+s-j)}-q^{-m(ji+s)})\\
&\quad \qquad - \sum_{t\le -1}\sum_{i=t+1}^{0}q^{pjt^2+2pst}z^{-jt}
q^{-pji(i-1)-2psi}  (q^{m(ji+s-j)}-q^{-m(ji+s)}).
\end{align*}}%
In the last line we have make the substitutions $t\to -t$ and $i\to -i+1$ to get
{\allowdisplaybreaks \begin{align*}
\sum_{t\in\mathbb{Z}}&q^{pjt^2+2pst}z^{-jt}
\sum_{i=1}^{t}q^{-pji(i-1)-2psi}  \times (q^{m(ji+s-j)}-q^{-m(ji+s)})\\
&=\sum_{t\ge 1}\sum_{i=1}^{t}q^{pjt^2+2pst}z^{-jt}
q^{-pji(i-1)-2psi}  (q^{m(ji+s-j)}-q^{-m(ji+s)})\\
&\quad \qquad + \sum_{t\ge 1}\sum_{i=1}^{t}q^{pjt^2-2pst}z^{jt}
q^{-pji(i-1)+2psi-2ps}  (q^{m(ji- s-j)}-q^{-m(ji-s)}).
\end{align*}}%
Replacing $i$ with $i+1$ and $t$ with $t+1$ brings us to
{\allowdisplaybreaks \begin{align*}
\sum_{t\in\mathbb{Z}}&q^{pjt^2+2pst}z^{-jt}
\sum_{i=1}^{t}q^{-pji(i-1)-2psi}  \times (q^{m(ji+s-j)}-q^{-m(ji+s)})\\
&=\sum_{t\ge 1}\sum_{i=0}^{t-1}q^{pjt^2+2pst}z^{-jt}
q^{-pj(i+1)i-2ps(i+1)}  (q^{m(j(i+1)+s-j)}-q^{-m(j(i+1)+s)})\\
&\qquad  + \sum_{t\ge 1}\sum_{i=0}^{t-1}q^{pjt^2-2pst}z^{jt}
q^{-pj(i+1)i+2ps(i+1)-2ps}  (q^{m(j(i+1)- s-j)}-q^{-m(j(i+1)-s)})\\
&=\sum_{t\ge 0}\sum_{i=0}^{t}q^{pj(t+1)^2+2ps(t+1)}z^{-jt-j}
q^{-pj(i+1)i-2ps(i+1)}  (q^{m(j(i+1)+s-j)}-q^{-m(j(i+1)+s)})\\
&\qquad  + \sum_{t\ge 0}\sum_{i=0}^{t}q^{pj(t+1)^2-2ps(t+1)}z^{jt+j}
q^{-pj(i+1)i+2ps(i+1)-2ps}  (q^{m(j(i+1)- s-j)}-q^{-m(j(i+1)-s)}).
\end{align*}}%
We interchange summation symbols and then simplify to obtain
{\allowdisplaybreaks \begin{align*}
\sum_{t\in\mathbb{Z}}&q^{pjt^2+2pst}z^{-jt}
\sum_{i=1}^{t}q^{-pji(i-1)-2psi}  \times (q^{m(ji+s-j)}-q^{-m(ji+s)})\\
&=\sum_{i= 0}^{\infty}\sum_{t=i}^{\infty}q^{pj(t+1)^2+2ps(t+1)}z^{-jt-j}
q^{-pj(i+1)i-2ps(i+1)}  (q^{m(j(i+1)+s-j)}-q^{-m(j(i+1)+s)})\\
&\qquad + \sum_{i= 0}^{\infty}\sum_{t=i}^{\infty}q^{pj(t+1)^2-2ps(t+1)}z^{jt+j}
q^{-pj(i+1)i+2ps(i+1)-2ps}  (q^{m(j(i+1)- s-j)}-q^{-m(j(i+1)-s)}).
\end{align*}}%
We then replace $t$ with $t+i$ to get
{\allowdisplaybreaks \begin{align*}
\sum_{t\in\mathbb{Z}}&q^{pjt^2+2pst}z^{-jt}
\sum_{i=1}^{t}q^{-pji(i-1)-2psi}  \times (q^{m(ji+s-j)}-q^{-m(ji+s)})\\
&=\sum_{i= 0}^{\infty}\sum_{t=0}^{\infty}q^{pj(t+i+1)^2+2ps(t+i+1)}z^{-jt-ji-j}
q^{-pj(i+1)i-2ps(i+1)}  (q^{m(j(i+1)+s-j)}-q^{-m(j(i+1)+s)})\\
& \qquad + \sum_{i= 0}^{\infty}\sum_{t=0}^{\infty}q^{pj(t+i+1)^2-2ps(t+i+1)}z^{jt+ji+j}
q^{-pj(i+1)i+2ps(i+1)-2ps} \\
&\qquad \qquad \times  (q^{m(j(i+1)- s-j)}-q^{-m(j(i+1)-s)}).
\end{align*}}%
Simplifying the exponents and distributing the sums brings us to
{\allowdisplaybreaks \begin{align*}
\sum_{t\in\mathbb{Z}}&q^{pjt^2+2pst}z^{-jt}
\sum_{i=1}^{t}q^{-pji(i-1)-2psi}  \times (q^{m(ji+s-j)}-q^{-m(ji+s)})\\
 &=  q^{-2ps+ms}\sum_{i=0}^{\infty} \sum_{t=0}^{\infty}q^{2jp\binom{t+1}{2}+p(j+2s)(t+1)+i(2jpt+jp+jm)}z^{-j(t+1)-ji}
\\
&\qquad -q^{-2ps-m(j+s)}\sum_{i=0}^{\infty} \sum_{t=0}^{\infty}q^{2jp\binom{t+1}{2}+p(j+2s)(t+1)+i(2jpt+jp-jm)}z^{-j(t+1)-ji}\\
&\qquad + q^{-ms}\sum_{i=0}^{\infty} \sum_{t=0}^{\infty}q^{2jp\binom{t+1}{2}+p(j-2s)(t+1)+i(2jpt+jp+jm)}z^{j(t+1)+ji}\\
&\qquad - q^{-m(j-s)}\sum_{i=0}^{\infty} \sum_{t=0}^{\infty}q^{2jp\binom{t+1}{2}+p(j-2s)(t+1)+i(2jpt+jp-jm)}z^{j(t+1)+ji}.
\end{align*}}%
Using the geometric series yields
{\allowdisplaybreaks \begin{align}
\sum_{t\in\mathbb{Z}}&q^{pjt^2+2pst}z^{-jt}
\sum_{i=1}^{t}q^{-pji(i-1)-2psi}  \times (q^{m(ji+s-j)}-q^{-m(ji+s)})
\label{equation:initSumOver_i_tPreAppell}\\
&=  q^{-2ps+ms} \sum_{t=0}^{\infty}\frac{q^{2jp\binom{t+1}{2}+p(j+2s)(t+1)}z^{-j(t+1)}}{1-q^{2jpt+jp+jm}z^{-j}}
\notag \\
&\qquad -q^{-2ps-m(j+s)} \sum_{t=0}^{\infty}\frac{q^{2jp\binom{t+1}{2}+p(j+2s)(t+1)}z^{-j(t+1)}}{1-q^{2jpt+jp-jm}z^{-j}}
\notag \\
&\qquad + q^{-ms}\sum_{t=0}^{\infty}\frac{q^{2jp\binom{t+1}{2}+p(j-2s)(t+1)}z^{j(t+1)}}{1-q^{2jpt+jp+jm}z^{j}}
\notag \\
&\qquad - q^{-m(j-s)} \sum_{t=0}^{\infty}\frac{q^{2jp\binom{t+1}{2}+p(j-2s)(t+1)}z^{j(t+1)}}{1-q^{2jpt+jp-jm}z^{j}}.\notag
\end{align}}%

We rework the right-hand side of (\ref{equation:initSumOver_i_tPreAppell}). In the first and second lines, we replace $t$ with $-t$, and in the third and fourth lines, we replace $t$ with $t-1$.  This gives
{\allowdisplaybreaks \begin{align*}
\sum_{t\in\mathbb{Z}}&q^{pjt^2+2pst}z^{-jt}
\sum_{i=1}^{t}q^{-pji(i-1)-2psi}  \times (q^{m(ji+s-j)}-q^{-m(ji+s)})\\
&=  q^{-2ps+ms} \sum_{t=-\infty}^{0}\frac{q^{2jp\binom{-t+1}{2}+p(j+2s)(-t+1)}z^{-j(-t+1)}}{1-q^{-2jpt+jp+jm}z^{-j}}
\frac{q^{2jpt-jp-jm}z^{j}}{q^{2jpt-jp-jm}z^{j}}\\
&\qquad -q^{-2ps-m(j+s)} \sum_{t=-\infty}^{0}\frac{q^{2jp\binom{-t+1}{2}+p(j+2s)(-t+1)}z^{-j(-t+1)}}{1-q^{-2jpt+jp-jm}z^{-j}}
\frac{q^{2jpt-jp+jm}z^{j}}{q^{2jpt-jp+jm}z^{j}}\\
&\qquad + q^{-ms}\sum_{t=1}^{\infty}\frac{q^{2jp\binom{t}{2}+p(j-2s)t}z^{jt}}{1-q^{2jp(t-1)+jp+jm}z^{j}}
 - q^{-m(j-s)} \sum_{t=1}^{\infty}\frac{q^{2jp\binom{t}{2}+p(j-2s)t}z^{jt}}{1-q^{2jp(t-1)+jp-jm}z^{j}}.
 \end{align*}}%
Simplifying brings us to
 {\allowdisplaybreaks \begin{align}
\sum_{t\in\mathbb{Z}}&q^{pjt^2+2pst}z^{-jt}
\sum_{i=1}^{t}q^{-pji(i-1)-2psi}  \times (q^{m(ji+s-j)}-q^{-m(ji+s)})
\label{equation:initSumOver_i_tPreAppellV2}\\
&=  -q^{-m(j-s)} \sum_{t=-\infty}^{0}\frac{q^{2jp\binom{t}{2}+p(j-2s)t}z^{jt}}{1-q^{2jpt-jp-jm}z^{j}}
  +q^{-ms} \sum_{t=-\infty}^{0}\frac{q^{2jp\binom{t}{2}+p(j- 2s)t}z^{jt}}{1-q^{2jpt-jp+jm}z^{j}}\notag \\
&\qquad + q^{-ms}\sum_{t=1}^{\infty}\frac{q^{2jp\binom{t}{2}+p(j-2s)t}z^{jt}}{1-q^{2jp(t-1)+jp+jm}z^{j}}
 - q^{-m(j-s)} \sum_{t=1}^{\infty}\frac{q^{2jp\binom{t}{2}+p(j-2s)t}z^{jt}}{1-q^{2jp(t-1)+jp-jm}z^{j}}\notag\\
 &=  -q^{-m(j-s)} \sum_{t=-\infty}^{\infty}\frac{q^{2jp\binom{t}{2}+p(j-2s)t}z^{jt}}{1-q^{2jp(t-1)+jp-jm}z^{j}}
  +q^{-ms} \sum_{t=-\infty}^{\infty}\frac{q^{2jp\binom{t}{2}+p(j- 2s)t}z^{jt}}{1-q^{2jp(t-1)+jp+jm}z^{j}}.\notag
\end{align}}%
Rewriting our find (\ref{equation:initSumOver_i_tPreAppellV2}) in terms of Appell functions gives the result.
\end{proof}


\section{New proof of mock theta conjecture-like identities for $1/2$-level string functions}
\label{section:mockTheta12-level}
We give a new proof of identities (\ref{equation:mockThetaConj2502r-2ndA}) and (\ref{equation:mockThetaConj2502r-2ndmu}).  First we set up the machinery, and then we prove identities (\ref{equation:mockThetaConj2502r-2ndA}) and (\ref{equation:mockThetaConj2502r-2ndmu}).  We specialize Theorem \ref{theorem:generalPolarFinite} to $(p,j)=(2,1)$.  This gives
\begin{align*}
\chi_{2r}^{(2,5)} (z;q)
&=C_{0,2r}^{(2,5)}(q)j(-zq^{2};q^{4})\\
&\qquad -\frac{1}{(q)_{\infty}^3}
q^{-\frac{1}{8}+\frac{(2r+1)^2}{10}}q^{-r}
j(-q^{2}z;q^{4})
\times 
\Big (  j(-q^{7+4r};q^{20} )
-  q^{4-2r}j(-q^{-3+4r};q^{20})\Big )\\
&\qquad \qquad \qquad    \times
\Big (m(-q,-q^{2}z^{-1};q^{4}) + m(-q,-q^{2}z;q^{4})\Big ).
\end{align*}
Focusing on the two theta functions within the parentheses, we us (\ref{equation:j-elliptic}) to set up for, and then apply (\ref{equation:jsplit-m2}).  This gives us
{\allowdisplaybreaks \begin{align*}
\chi_{2r}^{(2,5)} (z;q)
&=C_{0,2r}^{(2,5)}(q)j(-zq^{2};q^{4})\\
&\qquad -\frac{1}{(q)_{\infty}^3}
q^{-\frac{1}{8}+\frac{(2r+1)^2}{10}}q^{-r}
j(-q^{2}z;q^{4})
   \times j(q^{1+2r};q^{5})\\
&\qquad \qquad \qquad    \times
\Big (m(-q,-q^{2}z^{-1};q^{4}) + m(-q,-q^{2}z;q^{4})\Big ).
\end{align*}}%
Rewriting the string function using (\ref{equation:mathCalCtoStringC}) brings us to
{\allowdisplaybreaks \begin{align}
\chi_{2r}^{(2,5)} (z;q)
&=q^{-\frac{1}{8}+\frac{(2r+1)^2}{10}}
\mathcal{C}_{0,2r}^{(2,5)}(q)j(-zq^{2};q^{4})
\label{equation:newFourier25}\\
&\qquad -\frac{1}{(q)_{\infty}^3}
q^{-\frac{1}{8}+\frac{(2r+1)^2}{10}}q^{-r}
j(-q^{2}z;q^{4})
   \times j(q^{1+2r};q^{5})
\notag\\
&\qquad \qquad \qquad    \times
\Big (m(-q,-q^{2}z^{-1};q^{4}) + m(-q,-q^{2}z;q^{4})\Big ).
\notag
\end{align}}%

We specialize Proposition \ref{proposition:WeylKac} to $(p,j,\ell)=(2,1,2r)$ to have
\begin{align}
\chi_{2r}^{(2,5)}(z;q)
&=z^{-r}q^{-\frac{1}{8}+\frac{(2r+1)^2}{10}}
\frac{j(-q^{4r+12}z^{-5};q^{20})
-z^{2r+1}j(-q^{-4r+8}z^{-5};q^{20}) }
{j(z;q)}.
\label{equation:WeylKac25}
\end{align}
Comparing (\ref{equation:WeylKac25}) and (\ref{equation:newFourier25}) gives
{\allowdisplaybreaks \begin{align*}
z^{-r}
&\frac{j(-q^{4r+12}z^{-5};q^{20})
-z^{2r+1}j(-q^{-4r+8}z^{-5};q^{20}) }
{j(z;q)}\\
&=\mathcal{C}_{0,2r}^{(2,5)}(q)j(-zq^{2};q^{4})\\
&\qquad -\frac{1}{(q)_{\infty}^3}
q^{-r}j(-q^{2}z;q^{4}) j(q^{1+2r};q^{5})
\Big (m(-q,-q^{2}z^{-1};q^{4}) + m(-q,-q^{2}z;q^{4})\Big ).
\end{align*}}%
Solving for the string function results in
\begin{align}
\mathcal{C}_{0,2r}^{(2,5)}(q)
&=z^{-r}
\frac{j(-q^{4r+12}z^{-5};q^{20})
-z^{2r+1}j(-q^{-4r+8}z^{-5};q^{20}) }
{j(z;q)j(-zq^{2};q^{4})}
\label{equation:generalMockThetaConjLevel12}\\
&\qquad +\frac{1}{(q)_{\infty}^3}
q^{-r}   j(q^{1+2r};q^{5}) 
\Big (m(-q,-q^{2}z^{-1};q^{4}) + m(-q,-q^{2}z;q^{4})\Big ).
\notag
\end{align}


We prove identity (\ref{equation:mockThetaConj2502r-2ndA}).  In (\ref{equation:generalMockThetaConjLevel12}), we set $z=-1$ and use (\ref{equation:j-elliptic}) to get
\begin{align*}
\mathcal{C}_{0,2r}^{(2,5)}(q)
&=(-1)^{r}
2\frac{j(q^{4r+12};q^{20})}
{j(-1;q)j(q^{2};q^{4})}
 +2q^{-r}\frac{ j(q^{1+2r};q^{5}) }{(q)_{\infty}^3}
m(-q,q^{2};q^{4}).
\end{align*}
Elementary product rearrangements and (\ref{equation:2nd-A(q)}) give
\begin{equation}
(q)_{\infty}^3\mathcal{C}_{0,2r}^{(2,5)}(q)
=(-1)^{r}
\frac{J_{1}^4J_{4}}{J_{2}^4}j(q^{4r+12};q^{20})
 -2q^{-r} j(q^{1+2r};q^{5})  A(-q).
\end{equation}


We prove identity (\ref{equation:mockThetaConj2502r-2ndmu}).  In (\ref{equation:generalMockThetaConjLevel12}), we set $z=-q$ and use (\ref{equation:j-elliptic}) to get

\begin{align*}
\mathcal{C}_{0,2r}^{(2,5)}(q)
&=(-q)^{-r}
\frac{j(q^{4r+7};q^{20})
+q^{2r+1}j(q^{4r+17};q^{20}) }
{j(-1;q)j(q;q^{4})}\\
&\qquad +\frac{1}{(q)_{\infty}^3}
q^{-r}   j(q^{1+2r};q^{5}) 
\Big (m(-q,q;q^{4}) + m(-q,q^{3};q^{4})\Big ).
\end{align*}
Using (\ref{equation:jsplit-m2}) to combine the two theta functions gives us
{\allowdisplaybreaks \begin{align*}
\mathcal{C}_{0,2r}^{(2,5)}(q)
&=(-q)^{-r}
\frac{j(-q^{2r+1};-q^{5})}
{j(-1;q)j(q;q^{4})}\\
&\qquad +\frac{1}{(q)_{\infty}^3}
q^{-r}   j(q^{1+2r};q^{5}) 
\Big (m(-q,q;q^{4}) + m(-q,q^{3};q^{4})\Big ).
\end{align*}}%
From Corollary (\ref{corollary:mxqz-flip-xz}) and identity (\ref{equation:mxqz-fnq-z}), we have that
\begin{equation*}
m(-q,q^3;q^4)=m(-q,-q^{-4};q^4)=m(-q,-1;q^4).
\end{equation*}
Elementary product rearrangements and rewriting the Appell functions using (\ref{equation:2nd-mu(q)}) gives
\begin{equation}
(q)_{\infty}^3\mathcal{C}_{0,2r}^{(2,5)}(q)
=(-q)^{-r}\frac{1}{2}
\frac{J_{1}^3}{J_{2}J_{4}}j(-q^{2r+1};-q^{5})
+ q^{-r} \frac{1}{2}  j(q^{1+2r};q^{5}) \mu(q).
\end{equation}


\section{New mock theta conjecture-like identities for $1/3$-level string functions}
\label{section:mockTheta13-level}
We prove Theorem \ref{theorem:newMockThetaIdentitiespP37m0ell2r}.  We specialize Corollary \ref{corollary:polarFinite1p} to $p=3$
{\allowdisplaybreaks \begin{align*}
(q)_{\infty}^3\chi_{2r}^{(3,7)}&(z;q)-(q)_{\infty}^3C_{0,2r}^{(3,7)}(q)
j(-q^{3}z;q^{6})
 \\
&= -q^{-\frac{1}{8}+\frac{3(2r+1)^2}{28}+3-3r}
j(-q^{3}z;q^{6})
\sum_{m=1}^{2} 
(-1)^{m}q^{\binom{m+1}{2}+m(r-3)}
 \\
&\qquad \qquad \times \left ( j(-q^{7m+3(2r+1)};q^{42}) 
-q^{2m(3-r)}j(-q^{-7m+3(2r+1)};q^{42})\right )
 \\
&\qquad \qquad \qquad \times 
 \left ( m(-q^{m},-q^{3}z;q^{6})
+m(-q^{m},-q^{3}z^{-1};q^{6})\right ).
\end{align*}}%
We expand the sum over $m$ and change the string function notation with (\ref{equation:mathCalCtoStringC}).  Rearranging terms gives
\begin{align*}
(q)_{\infty}^3&\mathcal{C}_{0,2r}^{(3,7)}(q)
-\frac{(q)_{\infty}^3}{j(-q^{3}z;q^{6})}
q^{\frac{1}{8}-\frac{3(2r+1)^2}{28}}\chi_{2r}^{(3,7)}(z;q)
 \\
&= q^{3-3r}  \times \Big [-q^{-2+r}\left ( j(-q^{10+6r};q^{42}) 
-q^{6-2r}j(-q^{-4+6r};q^{42})\right )
 \\
&\qquad \qquad \qquad \times 
 \left ( m(-q,-q^{3}z;q^{6})
+m(-q,-q^{3}z^{-1};q^{6})\right )\\
&\qquad +  q^{-3+2r}  \left ( j(-q^{17+6r};q^{42}) 
-q^{12-4r}j(-q^{-11+6r};q^{42})\right )
 \\
&\qquad \qquad \qquad \times 
 \left ( m(-q^{2},-q^{3}z;q^{6})
+m(-q^{2},-q^{3}z^{-1};q^{6})\right )
\Big ]. 
\end{align*}
We use (\ref{equation:j-flip}) and (\ref{equation:j-elliptic}) to set up for the quintuple product identity (\ref{equation:quintuple}).  This gives
\begin{align*}
(q)_{\infty}^3&\mathcal{C}_{0,2r}^{(3,7)}(q)
-\frac{(q)_{\infty}^3}{j(-q^{3}z;q^{6})}
q^{\frac{1}{8}-\frac{3(2r+1)^2}{28}}\chi_{2r}^{(3,7)}(z;q)
 \\
&=  -q^{1-2r}\left ( j(-q^{32-6r};q^{42}) 
-q^{6-2r}j(-q^{46-6r};q^{42})\right )
 \\
&\qquad \qquad \qquad \times 
 \left ( m(-q,-q^{3}z;q^{6})
+m(-q,-q^{3}z^{-1};q^{6})\right )\\
&\qquad +  q^{-r}  \left ( j(-q^{17+6r};q^{42}) 
-q^{1+2r}j(-q^{31+6r};q^{42})\right )
 \\
&\qquad \qquad \qquad \times 
 \left ( m(-q^{2},-q^{3}z;q^{6})
+m(-q^{2},-q^{3}z^{-1};q^{6})\right ). 
\end{align*}
Applying the quintuple product identity (\ref{equation:quintuple}) brings us to
{\allowdisplaybreaks \begin{align*}
(q)_{\infty}^3&\mathcal{C}_{0,2r}^{(3,7)}(q)
-\frac{(q)_{\infty}^3}{j(-q^{3}z;q^{6})}
q^{\frac{1}{8}-\frac{3(2r+1)^2}{28}}\chi_{2r}^{(3,7)}(z;q)
 \\
&=  -q^{1-2r}\frac{j(q^{6-2r};q^{14})j(q^{26-4r};q^{28})}{J_{28}}
\times 
 \left ( m(-q,-q^{3}z;q^{6})
+m(-q,-q^{3}z^{-1};q^{6})\right )\\
&\qquad +  q^{-r} \frac{j(q^{1+2r};q^{14})j(q^{16+4r};q^{28})}{J_{28}} \times 
 \left ( m(-q^{2},-q^{3}z;q^{6})
+m(-q^{2},-q^{3}z^{-1};q^{6})\right ). 
\end{align*}}%
Setting $z=-q$ and rewriting the Appell functions using (\ref{equation:3rd-omega(q)}) and (\ref{equation:3rd-f(q)}) brings us to the result
{\allowdisplaybreaks \begin{align}
(q)_{\infty}^3&\mathcal{C}_{0,2r}^{(3,7)}(q)
-\frac{(q)_{\infty}^3}{J_{2}}
q^{\frac{1}{8}-\frac{3(2r+1)^2}{28}}\chi_{2r}^{(3,7)}(-q;q)
\label{equation:fourierFinal37} \\
&=  -q^{2-2r}\frac{j(q^{6-2r};q^{14})j(q^{26-4r};q^{28})}{J_{28}} \omega_3(-q)
+  \frac{q^{-r}}{2} \frac{j(q^{1+2r};q^{14})j(q^{16+4r};q^{28})}{J_{28}} 
f_{3}(q^2).\notag 
\end{align}}%

We take the appropriate specialization of Proposition \ref{proposition:WeylKac}:
\begin{equation*}
\chi_{2r}^{(3,7)}(z;q)
=z^{-r}q^{-\frac{1}{8}+3\frac{(2r+1)^2}{28}}
\frac{j(-q^{6r+24}z^{-7};q^{42})
-z^{2r+1}j(-q^{-6r+18}z^{-7};q^{42}) }
{j(z;q)}.
\end{equation*}
Setting $z=-q$ gives
\begin{equation*}
\chi_{2r}^{(3,7)}(-q;q)
=(-q)^{-r}q^{-\frac{1}{8}+3\frac{(2r+1)^2}{28}}
\frac{j(q^{6r+17};q^{42})
+q^{2r+1}j(q^{-6r+11};q^{42}) }
{j(-q;q)}.
\end{equation*}
Using (\ref{equation:j-flip}) and then the quintuple product identity (\ref{equation:quintuple}) gives
\begin{equation}
\chi_{2r}^{(3,7)}(-q;q)
=(-q)^{-r}q^{-\frac{1}{8}+3\frac{(2r+1)^2}{28}}
\frac{j(-q^{1+2r};q^{14})j(q^{16+4r};q^{28})}
{j(-1;q)J_{28}}.
\label{equation:weylKacFinal37}
\end{equation}

\noindent Inserting (\ref{equation:weylKacFinal37}) into (\ref{equation:fourierFinal37}) and isolating the string function gives the result.


\section{New mock theta conjecture-like identities for $2/3$-level string functions}
\label{section:mockTheta23-level}
In order to prove Theorems \ref{theorem:newMockThetaIdentitiespP38m0ell2r} and \ref{theorem:newMockThetaIdentitiespP38m2ell2r}, we will first need to take the appropriate specialization of our polar-finite decomposition Theorem \ref{theorem:generalPolarFinite}.  This will give us
\begin{proposition} \label{proposition:polarFinite23} We have
{\allowdisplaybreaks \begin{align*}
&\chi_{2r}^{(3,8)} (z;q)\\
&=C_{0,2r}^{(3,8)}(q)j(-z^{2}q^{6};q^{12})
+z^{-1}q^{\frac{3}{2}}C_{2,2r}^{(3,8)}(q)j(-z^{2};q^{12})\\
&\qquad +\frac{1}{(q)_{\infty}^3}
q^{-\frac{1}{8}+\frac{3(2r+1)^2}{32}}
j(-q^{6}z^{2};q^{12})\\
&\qquad \qquad  \times \Big (  q^{1-2r} 
\frac{j(q^{7-2r};q^{16})j(q^{30-4r};q^{32})}{J_{32}}
\Big (m(-q^{2},-q^{6}z^{-2};q^{12}) 
 + m(-q^{2},-q^{6}z^{2};q^{12})\Big )\\
&\qquad \qquad   - q^{-r} 
\frac{j(q^{1+2r};q^{16})j(q^{18+4r};q^{32})}{J_{32}}
\Big (m(-q^{4},-q^{6}z^{-2};q^{12}) 
 + m(-q^{4},-q^{6}z^{2};q^{12})\Big ) \Big ) \\
 &\qquad +\frac{1}{(q)_{\infty}^3}
q^{-\frac{1}{8}+\frac{3(2r+1)^2}{32}}z^{-1}
j(-z^{2};q^{12})\\
&\qquad \qquad     \Big ( q^{3-2r} 
\frac{j(q^{7-2r};q^{16})j(q^{30-4r};q^{32})}{J_{32}} 
\Big ( 
  m(-q^{8},-z^{2};q^{12})-m(-q^{4},-z^{2};q^{12})\Big )\\
 &\qquad \qquad   - q^{1-r} 
\frac{j(q^{1+2r};q^{16})j(q^{18+4r};q^{32})}{J_{32}}
\Big (
  m(-q^{10},-z^{2};q^{12})-m(-q^{2},-z^{2};q^{12}) \Big )\Big ).
\end{align*}}%
\end{proposition}
By choosing a value of $z$ so that the coefficient of a given family of string functions vanishes, we see that much more vanishes.  For $z=i$ half of the terms on the right-hand side vanish, and for $z=-iq^3$ the other half of the terms on the right-hand side vanish.  Indeed, we have

\begin{lemma}\label{lemma:polarFinite23m0AppellVanish} We have
\begin{gather*}
\lim_{z\to i}j(-z^2;q^{12})\left ( m(-q^{8},-z^2;q^{12})-m(-q^{4},-z^2;q^{12})\right )=0,\\ 
\lim_{z\to i}j(-z^2;q^{12})\left ( m(-q^{10},-z^2;q^{12})-m(-q^{2},-z^2;q^{12})\right )=0. 
\end{gather*}
\end{lemma}

\begin{lemma}\label{lemma:polarFinite23m2AppellVanish} We have
\begin{gather*}
\lim_{z\to iq^3} j(-z^2q^6;q^{12})\left ( m(-q^{2},-q^{6}z^{-2};q^{12})-m(-q^{2},-q^{6}z^2;q^{12})\right )=0,\\ 
\lim_{z\to iq^3} j(-z^2q^{6};q^{12})\left ( m(-q^{4},-q^{6}z^{-2};q^{12})-m(-q^{4},-q^{6}z^2;q^{12})\right )=0. 
\end{gather*}
\end{lemma}

For $z=i$ and $z=iq^3$, we then need to evaluate the character on the left-hand side.  Then it is straightforward to solve for the string function.
\begin{proposition}\label{proposition:weylKac23ell2rzVal} We have
\begin{gather}
\chi_{2r}^{(3,8)}(i;q)
=(-1)^{\kappa(r)}q^{-\frac{1}{8}+\frac{3(2r+1)^2}{32}}
\frac{J_2}{J_{1}J_{4}}j(-q^{27+6r};q^{48}),
 \label{equation:weylKac23ell2rzVal1}\\
 \chi_{2r}^{(3,8)}(iq^3;q)
=i(-1)^{\kappa(r)}q^{3-3r}q^{-\frac{1}{8}+\frac{3(2r+1)^2}{32}}
\frac{J_{2}}{J_{1}J_{4}}j(-q^{3+6r};q^{48}),
 \label{equation:weylKac23ell2rzVal2}
\end{gather}
 where  $\kappa(r):=\begin{cases} 0 & \textup{if} \ r=0,1,\\
1 & \textup{if} \ r=2,3. \end{cases}$
\end{proposition}

\begin{proof}[Proof of Proposition \ref{proposition:polarFinite23}]
Let us specialize Theorem \ref{theorem:generalPolarFinite}  to $p=3, j=2$:
{\allowdisplaybreaks \begin{align*}
&\chi_{2r}^{(3,8)} (z;q)\\
&=\sum_{s=0}^{1}z^{-s}q^{\frac{3}{2}s^2}C_{2s,2r}^{(3,8)}(q)j(-z^{2}q^{6-6s};q^{12})\\
&\qquad -\frac{1}{(q)_{\infty}^3}\sum_{s=0}^{1}
q^{-\frac{1}{8}+\frac{3(2r+1)^2}{32}}q^{3-3(r-s)}z^{-s}
j(-q^{6-6s}z^{2};q^{12})\\
&\qquad \qquad   \times \sum_{m=1}^{2}(-1)^{m}
q^{\binom{m+1}{2}+m(r-3)} 
\Big (  j(-q^{8m+3+6r};q^{48} ) 
-  q^{7m-2mr}j(-q^{-8m+3+6r};q^{48})\Big )\\
&\qquad \qquad \qquad     \times
\Big (q^{ms-6s}m(-q^{2m-6s},-q^{6+6s}z^{-2};q^{12}) 
 + q^{-ms}m(-q^{2m+6s},-q^{6-6s}z^{2};q^{12})\Big ).
\end{align*}}%
Expanding the sum over $s$ and then using (\ref{equation:mxqz-fnq-z}) gives
{\allowdisplaybreaks \begin{align*}
&\chi_{2r}^{(3,8)} (z;q)\\
&=C_{0,2r}^{(3,8)}(q)j(-z^{2}q^{6};q^{12})
+z^{-1}q^{\frac{3}{2}}C_{2,2r}^{(3,8)}(q)j(-z^{2};q^{12})\\
&\qquad -\frac{1}{(q)_{\infty}^3}q^{-\frac{1}{8}+\frac{3(2r+1)^2}{32}}q^{3-3r}
j(-q^{6}z^{2};q^{12})\\
&\qquad \qquad   \times \sum_{m=1}^{2}(-1)^{m}
q^{\binom{m+1}{2}+m(r-3)} 
\Big (  j(-q^{8m+3+6r};q^{48} )  
-  q^{7m-2mr}j(-q^{-8m+3+6r};q^{48})\Big )\\
&\qquad \qquad \qquad     \times
\Big (m(-q^{2m},-q^{6}z^{-2};q^{12}) 
 + m(-q^{2m},-q^{6}z^{2};q^{12})\Big )\\
 &\qquad -\frac{1}{(q)_{\infty}^3}q^{-\frac{1}{8}+\frac{3(2r+1)^2}{32}}q^{6-3r}z^{-1}
j(-z^{2};q^{12})\\
&\qquad \qquad   \times \sum_{m=1}^{2}(-1)^{m}
q^{\binom{m+1}{2}+m(r-3)} 
\Big (  j(-q^{8m+3+6r};q^{48} )  -  q^{7m-2mr}j(-q^{-8m+3+6r};q^{48})\Big )\\
&\qquad \qquad \qquad     \times
\Big (q^{m-6}m(-q^{2m-6},-z^{-2};q^{12}) 
 + q^{-m}m(-q^{2m+6},-z^{2};q^{12})\Big ).
\end{align*}}%
Expand the two sums over $m$ brings us to
{\allowdisplaybreaks \begin{align*}
\chi_{2r}^{(3,8)} (z;q)
&=C_{0,2r}^{(3,8)}(q)j(-z^{2}q^{6};q^{12})
+z^{-1}q^{\frac{3}{2}}C_{2,2r}^{(3,8)}(q)j(-z^{2};q^{12})\\
&\qquad -\frac{1}{(q)_{\infty}^3}
q^{-\frac{1}{8}+\frac{3(2r+1)^2}{32}}q^{3-3r}
j(-q^{6}z^{2};q^{12})\\
&\qquad \qquad  \Big (  -q^{-2+r} 
\Big (  j(-q^{11+6r};q^{48} )  -  q^{7-2r}j(-q^{-5+6r};q^{48})\Big )\\
&\qquad \qquad \qquad     \times
\Big (m(-q^{2},-q^{6}z^{-2};q^{12}) 
 + m(-q^{2},-q^{6}z^{2};q^{12})\Big )\\
&\qquad \qquad   + q^{-3+2r} 
\Big (  j(-q^{19+6r};q^{48} )  -  q^{14-4r}j(-q^{-13+6r};q^{48})\Big )\\
&\qquad \qquad \qquad     \times
\Big (m(-q^{4},-q^{6}z^{-2};q^{12}) 
 + m(-q^{4},-q^{6}z^{2};q^{12})\Big ) \Big ) \\
 &\qquad -\frac{1}{(q)_{\infty}^3}
q^{-\frac{1}{8}+\frac{3(2r+1)^2}{32}}q^{6-3r}z^{-1}
j(-z^{2};q^{12})\\
&\qquad \qquad     \Big ( -q^{-2+r} 
\Big (  j(-q^{11+6r};q^{48} )  -  q^{7-2r}j(-q^{-5+6r};q^{48})\Big )\\
&\qquad \qquad \qquad     \times
\Big (q^{-5}m(-q^{-4},-z^{-2};q^{12}) 
 + q^{-1}m(-q^{8},-z^{2};q^{12})\Big )\\
 &\qquad \qquad   + q^{-3+2r} 
\Big (  j(-q^{19+6r};q^{48} )  -  q^{14-4r}j(-q^{-13+6r};q^{48})\Big )\\
&\qquad \qquad \qquad     \times
\Big (q^{-4}m(-q^{-2},-z^{-2};q^{12}) 
 + q^{-2}m(-q^{10},-z^{2};q^{12})\Big )\Big ).
\end{align*}}%
We now use (\ref{equation:j-elliptic}) and (\ref{equation:j-flip}) to set up for the quintuple product identity.  
{\allowdisplaybreaks \begin{align*}
\chi_{2r}^{(3,8)} (z;q)
&=C_{0,2r}^{(3,8)}(q)j(-z^{2}q^{6};q^{12})
+z^{-1}q^{\frac{3}{2}}C_{2,2r}^{(3,8)}(q)j(-z^{2};q^{12})\\
&\qquad -\frac{1}{(q)_{\infty}^3}
q^{-\frac{1}{8}+\frac{3(2r+1)^2}{32}}q^{3-3r}
j(-q^{6}z^{2};q^{12})\\
&\qquad \qquad  \Big (  -q^{-2+r} 
\Big (  j(-q^{37-6r};q^{48} )  -  q^{7-2r}j(-q^{53-6r};q^{48})\Big )\\
&\qquad \qquad \qquad     \times
\Big (m(-q^{2},-q^{6}z^{-2};q^{12}) 
 + m(-q^{2},-q^{6}z^{2};q^{12})\Big )\\
&\qquad \qquad   + q^{-3+2r} 
\Big (  j(-q^{19+6r};q^{48} )  -  q^{1+2r}j(-q^{35+6r};q^{48})\Big )\\
&\qquad \qquad \qquad     \times
\Big (m(-q^{4},-q^{6}z^{-2};q^{12}) 
 + m(-q^{4},-q^{6}z^{2};q^{12})\Big ) \Big ) \\
 &\qquad -\frac{1}{(q)_{\infty}^3}
q^{-\frac{1}{8}+\frac{3(2r+1)^2}{32}}q^{6-3r}z^{-1}
j(-z^{2};q^{12})\\
&\qquad \qquad     \Big ( -q^{-2+r} 
\Big (  j(-q^{37-6r};q^{48} )  -  q^{7-2r}j(-q^{53-6r};q^{48})\Big )\\
&\qquad \qquad \qquad     \times
\Big (q^{-5}m(-q^{-4},-z^{-2};q^{12}) 
 + q^{-1}m(-q^{8},-z^{2};q^{12})\Big )\\
 &\qquad \qquad   + q^{-3+2r} 
\Big (  j(-q^{19+6r};q^{48} )  -  q^{1+2r}j(-q^{35+6r};q^{48})\Big )\\
&\qquad \qquad \qquad     \times
\Big (q^{-4}m(-q^{-2},-z^{-2};q^{12}) 
 + q^{-2}m(-q^{10},-z^{2};q^{12})\Big )\Big ).
\end{align*}}%

Using the quintuple product identity (\ref{equation:quintuple}) and then collecting the powers of $q$ in each of the summands gives
{\allowdisplaybreaks \begin{align*}
&\chi_{2r}^{(3,8)} (z;q)\\
&=C_{0,2r}^{(3,8)}(q)j(-z^{2}q^{6};q^{12})
+z^{-1}q^{\frac{3}{2}}C_{2,2r}^{(3,8)}(q)j(-z^{2};q^{12})\\
&\qquad +\frac{1}{(q)_{\infty}^3}
q^{-\frac{1}{8}+\frac{3(2r+1)^2}{32}}
j(-q^{6}z^{2};q^{12})\\
&\qquad \qquad  \Big (  q^{1-2r} 
\frac{j(q^{7-2r};q^{16})j(q^{30-4r};q^{32})}{J_{32}} 
\Big (m(-q^{2},-q^{6}z^{-2};q^{12}) 
 + m(-q^{2},-q^{6}z^{2};q^{12})\Big )\\
&\qquad \qquad   - q^{-r} 
\frac{j(q^{1+2r};q^{16})j(q^{18+4r};q^{32})}{J_{32}}
\Big (m(-q^{4},-q^{6}z^{-2};q^{12}) 
 + m(-q^{4},-q^{6}z^{2};q^{12})\Big ) \Big ) \\
 &\qquad +\frac{1}{(q)_{\infty}^3}
q^{-\frac{1}{8}+\frac{3(2r+1)^2}{32}}z^{-1}
j(-z^{2};q^{12})\\
&\qquad \qquad     \Big ( q^{3-2r} 
\frac{j(q^{7-2r};q^{16})j(q^{30-4r};q^{32})}{J_{32}} 
\Big (q^{-4}m(-q^{-4},-z^{-2};q^{12}) 
 + m(-q^{8},-z^{2};q^{12})\Big )\\
 &\qquad \qquad   - q^{1-r} 
\frac{j(q^{1+2r};q^{16})j(q^{18+4r};q^{32})}{J_{32}}
\Big (q^{-2}m(-q^{-2},-z^{-2};q^{12}) 
 + m(-q^{10},-z^{2};q^{12})\Big )\Big ).
\end{align*}}%
Using the Appell function property (\ref{equation:mxqz-flip}) gives us the result.
\end{proof}

\begin{proof}[Proof of Lemma \ref{lemma:polarFinite23m0AppellVanish}]
We consider the first limit.  The definition of our Appell function (\ref{equation:m-def}) gives 
\begin{align*}
\lim_{z\to i}j(-z^2;q^{12})&\left ( m(-q^{8},-z^2;q^{12})-m(-q^{4},-z^2;q^{12})\right )\\ 
&=\sum_{n\in\mathbb{Z}}\frac{(-1)^nq^{12\binom{n}{2}}}{1-q^{12(n-1)}(-q^8)}
-\sum_{n\in\mathbb{Z}}\frac{(-1)^nq^{12\binom{n}{2}}}{1-q^{12(n-1)}(-q^4)}.
\end{align*}
Setting $n \to n+1$ and then using the reciprocal formula (\ref{equation:jacobiThetaReciprocal}) yields
{\allowdisplaybreaks \begin{align*}
\lim_{z\to i}j(-z^2;q^{12})&\left ( m(-q^{8},-z^2;q^{12})-m(-q^{4},-z^2;q^{12})\right )\\ 
&=-\sum_{n}\frac{(-1)^nq^{12\binom{n+1}{2}}}{1+q^{12n+8}}
+\sum_{n}\frac{(-1)^nq^{12\binom{n+1}{2}}}{1+q^{12n+4}}\\
&=-\frac{J_{12}^3}{j(-q^{8};q^{12})}+\frac{J_{12}^3}{j(-q^{4};q^{12})}=0,
\end{align*}}%
where the last equality follows from (\ref{equation:j-flip}).  The argument for the second limit is similar and will be omitted.
\end{proof}

\begin{proof}[Proof of Lemma \ref{lemma:polarFinite23m2AppellVanish}]
We first rewrite the Appell functions.  Using the Appell function properties (\ref{equation:mxqz-flip}) and (\ref{equation:mxqz-fnq-z}) gives
\begin{align*}
m(-q^{2},-q^{6}z^{-2};q^{12})
=-q^{-2}m(-q^{-2},-q^{-6}z^{2};q^{12})
=-q^{-2}m(-q^{-2},-q^{6}z^{2};q^{12}).
\end{align*}
For the first limit, using the definition of our Appell function (\ref{equation:m-def}) immediately yields
\begin{align*}
\lim_{z\to iq^3} j(-z^2q^6;q^{12})
&\left ( m(-q^{2},-q^{6}z^{-2};q^{12})
-m(-q^{2},-q^{6}z^2;q^{12})\right )\\ 
&=\lim_{z\to iq^3} j(-z^2q^6;q^{12})
\left( q^{-2}m(-q^{-2},-q^{6}z^{2};q^{12})
-m(-q^{2},-q^{6}z^2;q^{12})\right)\\
&=-q^{-2}\sum_{n\in\mathbb{Z}}\frac{(-1)^{n}q^{12\binom{n}{2}}}{1-q^{12(n-1)}(-q^{-2})}
+\sum_{n\in\mathbb{Z}}\frac{(-1)^{n}q^{12\binom{n}{2}}}{1-q^{12(n-1)}(-q^{2})}.
\end{align*}
Letting $n\to n+1$ and then using the reciprocal formula (\ref{equation:jacobiThetaReciprocal}) gives
\begin{align*}
\lim_{z\to iq^3} j(-z^2q^6;q^{12})
&\left ( m(-q^{2},-q^{6}z^{-2};q^{12})
-m(-q^{2},-q^{6}z^2;q^{12})\right )\\ 
&=q^{-2}\frac{J_{12}^3}{j(-q^{-2};q^{12})}-\frac{J_{12}^3}{j(-q^{2};q^{12})}=0,
\end{align*}
where the last equality follows from applying (\ref{equation:j-flip}) and then (\ref{equation:j-elliptic}) to the theta function in the denominator of the first summand.  The proof of the second limit is analogous, so it will be omitted.
\end{proof}

\begin{proof}[Proof of Proposition \ref{proposition:weylKac23ell2rzVal}]
We prove (\ref{equation:weylKac23ell2rzVal1}).  Taking the appropriate specialization of Proposition \ref{proposition:WeylKac} gives
\begin{equation}
\chi_{2r}^{(3,8)}(z;q)
=z^{-r}q^{-\frac{1}{8}+\frac{3(2r+1)^2}{32}}
\frac{j(-q^{27+6r}z^{-8};q^{48})
-z^{2r+1}j(-q^{21-6r}z^{-8};q^{48}) }
{j(z;q)}. \label{equation:weylKac23genz}
\end{equation}
Specializing (\ref{equation:weylKac23genz}) to $z=i$ and using (\ref{equation:j-flip}) gives
\begin{equation*}
\chi_{2r}^{(3,8)}(i;q)
=i^{-r}q^{-\frac{1}{8}+\frac{3(2r+1)^2}{32}}
\frac{j(-q^{27+6r};q^{48})(1
-(-1)^ri)}
{j(i;q)}. 
\end{equation*}
Using the Jacobi triple product identity (\ref{equation:JTPid}) allows us to write
\begin{equation*}
\frac{1}{j(i;q)}=\frac{1}{(1-i)(iq)_{\infty}(-iq)_{\infty}(q)_{\infty}}
=\frac{1+i}{2}\frac{1}{(-q^2;q^2)_{\infty}J_{1}}
=\frac{1+i}{2}\frac{J_{2}}{J_{1}J_{4}}.
\end{equation*}
It is straightforward to show that
\begin{align*}
i^{-r}\frac{(1+i)(1-(-1)^ri)}{2}=(-1)^{\kappa(r)}, 
\qquad  \kappa(r):=\begin{cases} 0 & \textup{if} \ r=0,1,\\
1 & \textup{if} \ r=2,3, \end{cases}
\end{align*}
so the result follows.

We prove (\ref{equation:weylKac23ell2rzVal1}).   Specializing (\ref{equation:weylKac23genz}) to $z=iq^3$ gives
\begin{align*}
\chi_{2r}^{(3,8)}(iq^3;q)
=i^{-r}q^{-3r}q^{-\frac{1}{8}+\frac{3(2r+1)^2}{32}}
\frac{j(-q^{3+6r};q^{48})
-i(-1)^{r}q^{6r+3}j(-q^{-3-6r};q^{48}) }
{j(iq^3;q)}.
\end{align*}
Using \ref{equation:j-flip} and \ref{equation:j-elliptic} yields
\begin{align*}
\chi_{2r}^{(3,8)}(iq^3;q)
=i^{1-r}q^{3-3r}q^{-\frac{1}{8}+\frac{3(2r+1)^2}{32}}
\frac{j(-q^{3+6r};q^{48})
(1-(-1)^{r}i) }
{j(i;q)}.
\end{align*}
We then argue as before to obtain the result.
\end{proof}


\subsection{The $2/3$-level string functions with quantum number $m=0$}
We prove Theorem \ref{theorem:newMockThetaIdentitiespP38m0ell2r}.  We specialize Proposition \ref{proposition:polarFinite23} to $z=i$ and use Lemma \ref{lemma:polarFinite23m0AppellVanish} to produce the much smaller
 \begin{align*}
\chi_{2r}^{(3,8)} (i;q)
&=C_{0,2r}^{(3,8)}(q)j(q^{6};q^{12})\\
&\qquad +\frac{1}{(q)_{\infty}^3}
q^{-\frac{1}{8}+\frac{3(2r+1)^2}{32}}
j(q^{6};q^{12})\\
&\qquad \qquad  \times \Big (  q^{1-2r} 
\frac{j(q^{7-2r};q^{16})j(q^{30-4r};q^{32})}{J_{32}}
2m(-q^{2},q^{6};q^{12}) \\
&\qquad \qquad   - q^{-r} 
\frac{j(q^{1+2r};q^{16})j(q^{18+4r};q^{32})}{J_{32}}
2m(-q^{4},q^{6};q^{12})  \Big ).
\end{align*}

Substituting in the expression for the character found in (\ref{equation:weylKac23ell2rzVal1}) and the mock theta functions found in Proposition \ref{proposition:alternat3rdAppellForms2}, and then changing the string function notation with (\ref{equation:mathCalCtoStringC}) gives
 \begin{align*}
(-1)^{\kappa(r)}&\frac{J_{1}^3}{J_{6,12}}\frac{J_2}{J_{1}J_{4}}j(-q^{27+6r};q^{48})\\
&=(q)_{\infty}^3\mathcal{C}_{0,2r}^{(3,8)}(q)
+ q^{1-2r} 
\frac{j(q^{7-2r};q^{16})j(q^{30-4r};q^{32})}{J_{32}}
\left( q^2\omega_{3}(-q^2)
-q^2\frac{J_{24}^2J_{2}^4}{J_{4}^3J_{6}^2}\right )  \\
&\qquad \qquad   - q^{-r} 
\frac{j(q^{1+2r};q^{16})j(q^{18+4r};q^{32})}{J_{32}}
\left ( \frac{1}{2}f_{3}(q^4)
-\frac{1}{2}\frac{J_{2}^4J_{12}^6}{J_{4}^3J_{6}^4J_{24}^2}\right ) .
\end{align*}
Isolating the string function and regrouping terms then reads
 {\allowdisplaybreaks \begin{align*}
(q)_{\infty}^3&\mathcal{C}_{0,2r}^{(3,8)}(q)\\
&=(-1)^{\kappa(r)}\frac{J_{1}^3}{J_{6,12}}\frac{J_2}{J_{1}J_{4}}j(-q^{27+6r};q^{48})\\
&\qquad +q^{3-2r} 
\frac{j(q^{7-2r};q^{16})j(q^{30-4r};q^{32})}{J_{32}}
\frac{J_{24}^2J_{2}^4}{J_{4}^3J_{6}^2} 
- q^{-r} 
\frac{j(q^{1+2r};q^{16})j(q^{18+4r};q^{32})}{J_{32}}
\frac{1}{2}\frac{J_{2}^4J_{12}^6}{J_{4}^3J_{6}^4J_{24}^2}\\
&\qquad - q^{3-2r} 
\frac{j(q^{7-2r};q^{16})j(q^{30-4r};q^{32})}{J_{32}} \omega_{3}(-q^2) 
 + q^{-r} 
\frac{j(q^{1+2r};q^{16})j(q^{18+4r};q^{32})}{J_{32}} \frac{1}{2}f_{3}(q^4).
\end{align*}}%
The result then follows from Proposition \ref{proposition:masterThetaIdentitypP38m0ell2r}


\subsection{The $2/3$-level string functions with quantum number $m=2$}
We prove Theorem \ref{theorem:newMockThetaIdentitiespP38m2ell2r}.  We specialize Proposition \ref{proposition:polarFinite23} to $z=iq^3$ and use Lemma \ref{lemma:polarFinite23m2AppellVanish} to produce the much smaller
{\allowdisplaybreaks \begin{align*}
&\chi_{2r}^{(3,8)} (iq^3;q)\\
&=
-iq^{-3}q^{\frac{3}{2}}C_{2,2r}^{(3,8)}(q)j(q^6;q^{12})\\
&\qquad -i\frac{1}{(q)_{\infty}^3}
q^{-\frac{1}{8}+\frac{3(2r+1)^2}{32}}q^{-3}
j(q^{6};q^{12})\\
&\qquad \qquad     \times \Big ( q^{3-2r} 
\frac{j(q^{7-2r};q^{16})j(q^{30-4r};q^{32})}{J_{32}} 
\Big ( 
  m(-q^{8},q^{6};q^{12})-m(-q^{4},q^{6};q^{12})\Big )\\
 &\qquad \qquad   - q^{1-r} 
\frac{j(q^{1+2r};q^{16})j(q^{18+4r};q^{32})}{J_{32}}
\Big (
  m(-q^{10},q^{6};q^{12})-m(-q^{2},q^{6};q^{12}) \Big )\Big ).
\end{align*}}%
Substituting in the expression for the character found in (\ref{equation:weylKac23ell2rzVal2}) and the mock theta functions found in Proposition \ref{proposition:alternat3rdAppellForms2}, and then changing the string function notation with (\ref{equation:mathCalCtoStringC}) gives
{\allowdisplaybreaks \begin{align*}
i(-1)^{\kappa(r)}&q^{3-3r}
\frac{J_{2}}{J_{1}J_{4}}j(-q^{3+6r};q^{48})\\
&=
-iq^{-3}\mathcal{C}_{2,2r}^{(3,8)}(q)j(q^6;q^{12})\\
&\qquad -i\frac{1}{(q)_{\infty}^3}q^{-3}
j(q^{6};q^{12})\\
&\qquad \qquad     \times \Big ( q^{3-2r} 
\frac{j(q^{7-2r};q^{16})j(q^{30-4r};q^{32})}{J_{32}} 
\Big ( 1-\left ( \frac{1}{2}f_{3}(q^4)
-\frac{1}{2}\frac{J_{2}^4J_{12}^6}{J_{4}^3J_{6}^4J_{24}^2}\right ) \Big )\\
 &\qquad \qquad   - q^{1-r} 
\frac{j(q^{1+2r};q^{16})j(q^{18+4r};q^{32})}{J_{32}}
\Big (1-\left ( q^2\omega_{3}(-q^2)
-q^2\frac{J_{24}^2J_{2}^4}{J_{4}^3J_{6}^2}\right ) \Big )\Big ).
\end{align*}}%
Isolating the string function and rearranging terms gives
\begin{align*}
(q)_{\infty}^3\mathcal{C}_{2,2r}^{(3,8)}(q)
&=-(-1)^{\kappa(r)}q^{6-3r}
\frac{J_{1}^3}{J_{6,12}}\frac{J_{2}}{J_{1}J_{4}}j(-q^{3+6r};q^{48})\\
 &\qquad -    q^{3-2r} 
\frac{j(q^{7-2r};q^{16})j(q^{30-4r};q^{32})}{J_{32}} 
\frac{1}{2}\frac{J_{2}^4J_{12}^6}{J_{4}^3J_{6}^4J_{24}^2} \\
& \qquad   + q^{1-r} 
\frac{j(q^{1+2r};q^{16})j(q^{18+4r};q^{32})}{J_{32}}
q^2\frac{J_{24}^2J_{2}^4}{J_{4}^3J_{6}^2}\\
&\qquad -    q^{3-2r} 
\frac{j(q^{7-2r};q^{16})j(q^{30-4r};q^{32})}{J_{32}} 
\left ( 1- \frac{1}{2}f_{3}(q^4) \right )\\
& \qquad   + q^{1-r} 
\frac{j(q^{1+2r};q^{16})j(q^{18+4r};q^{32})}{J_{32}}
\left (1-q^2\omega_{3}(-q^2)\right ).
\end{align*}
The result then follows from Proposition \ref{proposition:masterThetaIdentitypP38m2ell2r}.


\section{New mock theta conjecture-like identities for $1/5$-level string functions}
\label{section:mockTheta15-level}
Theorem \ref{theorem:newMockThetaIdentitiespP511m0ell2r} will follow from the next proposition. 

\begin{proposition}\label{proposition:fourierExp511speczm1}
{\allowdisplaybreaks \begin{align*}
(q)_{\infty}^3\mathcal{C}_{0,2r}^{(5,11)}(q)
&=2(-1)^{r}\frac{(q)_{\infty}^3}{J_{5,10}}
\frac{j(q^{50-10r};q^{110})}{j(-1;q)}
\\
&\qquad 
-2q^{6-4r} \times \left ( j(-q^{16+10r};q^{110}) 
-q^{4+8r}j(-q^{6-10r};q^{110})\right )
 \times  m(-q,q^{5};q^{10})\\
& \qquad 
+2q^{3-3r}\times \left ( j(-q^{27+10r};q^{110}) 
-q^{3+6r}j(-q^{17-10r};q^{110})\right )
\times m(-q^{2},q^{5};q^{10})\\
&\qquad 
-2q^{1-2r}  \times \left ( j(-q^{38+10r};q^{110}) 
-q^{2+4r}j(-q^{28-10r};q^{110})\right )
\times m(-q^{3},q^{5};q^{10})\\
&\qquad  
+2q^{-r} \times \left ( j(-q^{49+10r};q^{110}) 
-q^{1+2r}j(-q^{39-10r};q^{110})\right )
\times m(-q^{4},q^{5};q^{10}).
\end{align*}}%
\end{proposition}

To prove Theorem \ref{theorem:newMockThetaIdentitiespP511m0ell2r}, we take the result of Proposition \ref{proposition:fourierExp511speczm1} and use Proposition \ref{proposition:alternat10thAppellForms} to rewrite the Appell functions in terms of the tenth-order mock theta functions.  This gives
{\allowdisplaybreaks \begin{align*}
(q)_{\infty}^3&\mathcal{C}_{0,2r}^{(5,11)}(q)\\
&=2(-1)^{r}\frac{(q)_{\infty}^3}{J_{5,10}}
\frac{J_{50,110}}{j(-1;q)}\\
&\qquad -
q^{6-4r} \times \left ( j(-q^{16+10r};q^{110}) 
-q^{4+8r}j(-q^{6-10r};q^{110})\right )
 \times q\phi_{10}(-q) \\
& \qquad +q^{3-3r}\times \left ( j(-q^{27+10r};q^{110}) 
-q^{3+6r}j(-q^{17-10r};q^{110})\right )
\times \chi_{10}(q^2) \\
&\qquad 
-q^{1-2r}  \times \left ( j(-q^{38+10r};q^{110}) 
-q^{2+4r}j(-q^{28-10r};q^{110})\right )
\times \left ( -\psi_{10}(-q)\right ) \\
&\qquad  
+q^{-r} \times \left ( j(-q^{49+10r};q^{110}) 
-q^{1+2r}j(-q^{39-10r};q^{110})\right )\times X_{10}(q^2)\\
&\qquad 
+q^{6-4r} \times \left ( j(-q^{16+10r};q^{110}) 
-q^{4+8r}j(-q^{6-10r};q^{110})\right )
 \times q\frac{J_{10}^2J_{3,10}}
{\overline{J}_{1,5}J_{2,10}}
\cdot \frac{J_{1}}{\overline{J}_{3,10}}   \\
& \qquad -q^{3-3r}\times \left ( j(-q^{27+10r};q^{110}) 
-q^{3+6r}j(-q^{17-10r};q^{110})\right )
\times q^{2}\frac{J_{10}^2J_{1,10}}{\overline{J}_{2,5}J_{4,10}}
\cdot \frac{J_{1}}
{\overline{J}_{4,10}}  \\
&\qquad 
+q^{1-2r}  \times \left ( j(-q^{38+10r};q^{110}) 
-q^{2+4r}j(-q^{28-10r};q^{110})\right )
\times q\frac{J_{10}^2 J_{1,10}}
{\overline{J}_{2,5}J_{4,10}}
\cdot \frac{J_{1}}
{\overline{J}_{1,10}} \\
&\qquad  
-q^{-r} \times \left ( j(-q^{49+10r};q^{110}) 
-q^{1+2r}j(-q^{39-10r};q^{110})\right )
\times \frac{J_{10}^2J_{3,10}}
{\overline{J}_{1,5}J_{2,10}}
\cdot \frac{J_{1}}
{\overline{J}_{2,10}}.
\end{align*}}%
Proposition \ref{proposition:masterThetaIdentitypP511m0ell2r} allows us to combine the five simple quotients of theta functions into a single simple quotient.  This gives the result.

\begin{proof}[Proof of Proposition \ref{proposition:fourierExp511speczm1}]
Specializing Corollary \ref{corollary:polarFinite1p} to $(p,p')=(5,11)$ yields
{\allowdisplaybreaks \begin{align}
\chi_{2r}^{(5,11)}(z;q)
&=C_{0,2r}^{(5,11)}(q)
j(-q^{5}z;q^{10})
\label{equation:fourierExppP511} \\
&\qquad -q^{-\frac{1}{8}+\frac{5(2r+1)^2}{44}+10-5r}\frac{j(-q^{5}z;q^{10})}{(q)_{\infty}^3}
\sum_{m=1}^{4} 
(-1)^{m}q^{\binom{m+1}{2}+m(r-5)}
\notag \\
&\qquad \qquad \times \left ( j(-q^{11m+5+10r};q^{110}) 
-q^{2m(5-r)}j(-q^{-11m+5+10r};q^{110})\right )
\notag \\
&\qquad \qquad \qquad \times 
 \left ( m(-q^{m},-q^{5}z;q^{10})
+m(-q^{m},-q^{5}z^{-1};q^{10})\right ).
\notag
\end{align}}%

We rewrite the character using Proposition \ref{proposition:WeylKac}.  This gives
{\allowdisplaybreaks \begin{equation}
\chi_{2r}^{(5,11)}(z;q)
=z^{-r}q^{-\frac{1}{8}+\frac{5(2r+1)^2}{44}}
\frac{j(-q^{10r+60}z^{-11};q^{110})
-z^{2r+1}j(-q^{-10r+50}z^{-11};q^{110}) }
{j(z;q)}.
\label{equation:charpP511}
\end{equation}}
Inserting (\ref{equation:charpP511}) into (\ref{equation:fourierExppP511}) gives
{\allowdisplaybreaks \begin{align*}
\frac{(q)_{\infty}^3}{j(-q^{5}z;q^{10})}
&z^{-r}q^{-\frac{1}{8}+\frac{5(2r+1)^2}{44}}
\frac{j(-q^{10r+60}z^{-11};q^{110})
-z^{2r+1}j(-q^{-10r+50}z^{-11};q^{110}) }
{j(z;q)}\\
&=(q)_{\infty}^3 C_{0,2r}^{(5,11)}(q)
 \\
&\qquad -q^{-\frac{1}{8}+\frac{5(2r+1)^2}{44}+10-5r}
\sum_{m=1}^{4} 
(-1)^{m}q^{\binom{m+1}{2}+m(r-5)}
 \\
&\qquad \qquad \times \left ( j(-q^{11m+5+10r};q^{110}) 
-q^{2m(5-r)}j(-q^{-11m+5+10r};q^{110})\right )
 \\
&\qquad \qquad \qquad \times 
 \left ( m(-q^{m},-q^{5}z;q^{10})
+m(-q^{m},-q^{5}z^{-1};q^{10})\right ).
\end{align*}}%
Changing the string function notation with (\ref{equation:mathCalCtoStringC}) and distributing the sum over $m$ gives
{\allowdisplaybreaks \begin{align*}
\frac{(q)_{\infty}^3}{j(-q^{5}z;q^{10})}
&z^{-r}
\frac{j(-q^{10r+60}z^{-11};q^{110})
-z^{2r+1}j(-q^{-10r+50}z^{-11};q^{110}) }
{j(z;q)}\\
&=(q)_{\infty}^3\mathcal{C}_{0,2r}^{(5,11)}(q)
 \\
&\qquad -q^{10-5r}
\Big ( 
-q^{-4+r} \times \left ( j(-q^{16+10r};q^{110}) 
-q^{10-2r}j(-q^{-6+10r};q^{110})\right )
 \\
&\qquad \qquad \qquad \times 
 \left ( m(-q,-q^{5}z;q^{10})
+m(-q,-q^{5}z^{-1};q^{10})\right )\\
& \qquad \qquad +q^{-7+2r} \times \left ( j(-q^{27+10r};q^{110}) 
-q^{20-4r}j(-q^{-17+10r};q^{110})\right )
 \\
&\qquad \qquad \qquad \times 
 \left ( m(-q^{2},-q^{5}z;q^{10})
+m(-q^{2},-q^{5}z^{-1};q^{10})\right )\\
&\qquad \qquad 
-q^{-9+3r}  \times \left ( j(-q^{38+10r};q^{110}) 
-q^{30-6r}j(-q^{-28+10r};q^{110})\right )
 \\
&\qquad \qquad \qquad \times 
 \left ( m(-q^{3},-q^{5}z;q^{10})
+m(-q^{3},-q^{5}z^{-1};q^{10})\right )\\
&\qquad \qquad 
+q^{-10+4r} \times \left ( j(-q^{49+10r};q^{110}) 
-q^{40-8r}j(-q^{-39+10r};q^{110})\right )
 \\
&\qquad \qquad \qquad \times 
 \left ( m(-q^{4},-q^{5}z;q^{10})
+m(-q^{4},-q^{5}z^{-1};q^{10})\right )
\Big ). 
\end{align*}}%
Specializing to $z=-1$ and rearranging terms gives the result.
\end{proof}


\section{On the positive admissible-level string functions}
\label{section:generalStringHeuristic}

In (\ref{equation:modStringFnHeckeForm}) we set $(m,\ell)\to (2k,2r)$, $(p,p^{\prime})\to(p,2p+j)$.  This gives
\begin{align*}
(q)_{\infty}^3\mathcal{C}_{2k,2r}^{(p,2p+j)}(q)
& = f_{1,2p+j,2p(2p+j)}(q^{1+k+r},-q^{p(2p+j+2r+1)};q)\\
 &\qquad \qquad -f_{1,2p+j,2p(2p+j)}(q^{k-r},-q^{p(2p+j-2r-1)};q).
\end{align*} 
We recall the definition (\ref{equation:mabc-def}) and note that the discriminant is $D:=b^2-ac=j(2p+j)$.   We will do our calculations mod theta, so we will replace the $z$-position with a $*$ and replace the equality symbol with a ``$\sim$''.  We have
{\allowdisplaybreaks \begin{align*}
&m_{1,2p+j,2p(2p+j)}(x,y,q,*,*)\\
&\qquad \sim\sum_{t=0}^{0}(-y)^tq^{2p(2p+j)\binom{t}{2}}j(q^{(2p+j)t}x;q)
m\left (-q^{\binom{2p+j+1}{2}-2p(2p+j)-tj(2p+j)} \frac{(-y)}{(-x)^{2p+j}},*;q^{j(2p+j)}\right )\\
&\qquad  \qquad +\sum_{t=0}^{2p(2p+j)-1}(-x)^tq^{\binom{t}{2}}j(q^{(2p+j)t}y;q^{2p(2p+j)})\\
&\qquad \qquad \qquad  \times m\left (-q^{2p(2p+j)\binom{2p+j+1}{2}-\binom{2p(2p+j)+1}{2}-tj(2p+j)}\frac{(-x)^{2p(2p+j)}}{(-y)^{2p+j}},*;q^{2pj(2p+j)^2}\right ).
\end{align*}}%
Hence for the first double-sum
{\allowdisplaybreaks \begin{align*}
&f_{1,2p+j,2p(2p+j)}(q^{1+k+r},-q^{p(2p+j+2r+1)};q)\\
&\qquad \sim\sum_{t=0}^{0}(q^{p(2p+j+2r+1)})^tq^{2p(2p+j)\binom{t}{2}}j(q^{(2p+j)t}q^{1+k+r};q)\\
&\qquad \qquad \times m\left ((-1)^{j+1}q^{\binom{j}{2}-k(2p+j)+j(p-r)-tj(2p+j)},*;q^{j(2p+j)}\right )\\
&\qquad  \qquad +\sum_{t=0}^{2p(2p+j)-1}(-1)^{t}q^{(1+k+r)t}q^{\binom{t}{2}}j(-q^{(2p+j)t}q^{p(2p+j+2r+1)};q^{2p(2p+j)})\\
&\qquad \qquad \qquad  \times m\left (-q^{p(2p+j)(j(2p+j)+2k)-tj(2p+j)},*;q^{2pj(2p+j)^2}\right ).
\end{align*}}%

\noindent The first sum over $t$ vanishes because $j(q^n;q)=0$ for all $n\in\mathbb{Z}$, so we get
\begin{align*}
&f_{1,2p+j,2p(2p+j)}(q^{1+k+r},-q^{p(2p+j+2r+1)};q)\\
&\qquad \sim \sum_{t=0}^{2p(2p+j)-1}(-1)^{t}q^{(1+k+r)t}q^{\binom{t}{2}}j(-q^{(2p+j)t}q^{p(2p+j+2r+1)};q^{2p(2p+j)})\\
&\qquad \qquad \qquad  \times m\left (-q^{p(2p+j)(j(2p+j)+2k)-tj(2p+j)},*;q^{2pj(2p+j)^2}\right ).
\end{align*}
In the remaining sum, we replace $t\to 2pi+m$, so we can rewrite the sum as
\begin{equation*}
\sum_{t=0}^{2p(2p+j)-1} \to \sum_{m=0}^{2p-1}\sum_{i=0}^{2p+j-1}.
\end{equation*}
This gives
{\allowdisplaybreaks \begin{align*}
&f_{1,2p+j,2p(2p+j)}(q^{1+k+r},-q^{p(2p+j+2r+1)};q)\\
&\qquad  \sim \sum_{m=0}^{2p-1}\sum_{i=0}^{2p+j-1}
(-1)^{2pi+m}q^{(1+k+r)(2pi+m)}q^{\binom{2pi+m}{2}}\\
&\qquad \qquad \times j(-q^{(2p+j)(2pi+m)}q^{p(2p+j+2r+1)};q^{2p(2p+j)})\\
&\qquad \qquad \qquad  \times m\Big ({-}q^{p(2p+j)(j(2p+j)+2k)-(2pi+m)j(2p+j)},*;q^{2pj(2p+j)^2}\Big ).
\end{align*}}%

We now rewrite our new expression using the quasi-elliptic transformation property for the theta function.  Let us pull the $i$ out of the theta function.  Using (\ref{equation:j-elliptic}) produces
{\allowdisplaybreaks \begin{align*}
&f_{1,2p+j,2p(2p+j)}(q^{1+k+r},-q^{p(2p+j+2r+1)};q)\\
&\qquad  \sim \sum_{m=0}^{2p-1}\sum_{i=0}^{2p+j-1}
(-1)^{2pi+m}q^{(1+k+r)(2pi+m)}q^{\binom{2pi+m}{2}}q^{-2p(2p+j)\binom{i}{2}-i(m(2p+j)+p(2p+j+2r+1))}\\
&\qquad \qquad  \times j(-q^{(2p+j)m+p(2p+j+2r+1)};q^{2p(2p+j)})\\
&\qquad \qquad \qquad  \times m\left (-q^{p(2p+j)(j(2p+j)+2k)-(2pi+m)j(2p+j)},*;q^{2pj(2p+j)^2}\right ).
\end{align*}}%
Simplifying and rewriting the exponents yields
{\allowdisplaybreaks  \begin{align*}
&f_{1,2p+j,2p(2p+j)}(q^{1+k+r},-q^{p(2p+j+2r+1)};q)\\
&\qquad  \sim \sum_{m=0}^{2p-1}(-1)^{m}q^{\binom{m+1}{2}+m(k+r)}j(-q^{(2p+j)m+p(2p+j+2r+1)};q^{2p(2p+j)})\\
&\qquad \qquad  \times \sum_{i=0}^{2p+j-1}q^{-2pj\binom{i+1}{2}+i(p(2k+j)-jm)}\\
&\qquad \qquad \qquad  \times m\left (-q^{2pj(\binom{2p+j}{2}-i(2p+j))}q^{(2p+j)(p(2k+j)-jm)},*;q^{2pj(2p+j)^2}\right ).
\end{align*}}%
We rewrite our expression in a form suitable for the Appell function property found in Theorem \ref{theorem:msplit-general-n}.  Our mod theta expression now reads
 {\allowdisplaybreaks \begin{align*}
&f_{1,2p+j,2p(2p+j)}(q^{1+k+r},-q^{p(2p+j+2r+1)};q)\\
&\qquad  \sim \sum_{m=0}^{2p-1}(-1)^{m}q^{\binom{m+1}{2}+m(k+r)}j(-q^{(2p+j)m+p(2p+j+2r+1)};q^{2p(2p+j)})\\
&\qquad \qquad  \times \sum_{i=0}^{2p+j-1}q^{-2pj\binom{i+1}{2}}(q^{p(2k+j)-jm})^{i}\\
&\qquad \qquad \qquad  \times m\left (-q^{2pj(\binom{2p+j}{2}-i(2p+j))}(q^{p(2k+j)-jm})^{(2p+j)},*;q^{2pj(2p+j)^2}\right ).
\end{align*}}%
Using Theorem \ref{theorem:msplit-general-n} gives
{\allowdisplaybreaks  \begin{align}
&f_{1,2p+j,2p(2p+j)}(q^{1+k+r},-q^{p(2p+j+2r+1)};q)
\label{equation:heuristicPartA}\\
&\qquad  \sim \sum_{m=0}^{2p-1}(-1)^{m}q^{\binom{m+1}{2}+m(k+r)}j(-q^{(2p+j)m+p(2p+j+2r+1)};q^{2p(2p+j)})
\notag\\
&\qquad \qquad  \times m(-q^{p(2k+j)-jm},*;q^{2pj}).
\notag
\end{align}}%

For the second double-sum, we use Theorem \ref{theorem:posDisc} to write modulo theta that
{\allowdisplaybreaks \begin{align*}
&f_{1,2p+j,2p(2p+j)}(q^{k-r},-q^{p(2p+j-2r-1)};q)\\
&\qquad \sim\sum_{t=0}^{0}q^{(p(2p+j-2r-1))t}q^{2p(2p+j)\binom{t}{2}}j(q^{(2p+j)t}q^{k-r};q)\\
&\qquad \qquad  \times m\Big ((-1)^{j+1}q^{\binom{j}{2}-k(2p+j)+j(p+r+1)-tj(2p+j)},*;q^{j(2p+j)}\Big )\\
&\qquad  \qquad \qquad +\sum_{t=0}^{2p(2p+j)-1}(-1)^{t}q^{t(k-r)}q^{\binom{t}{2}}j(-q^{(2p+j)t}q^{p(2p+j-2r-1)};q^{2p(2p+j)})\\
&\qquad \qquad \qquad \qquad  \times m\Big ({-}q^{p(2p+j)(j(2p+j)+2k)-tj(2p+j)},*;q^{2pj(2p+j)^2}\Big )\\
&\qquad \sim \sum_{t=0}^{2p(2p+j)-1}(-1)^{t}q^{t(k-r)}q^{\binom{t}{2}}j(-q^{(2p+j)t}q^{p(2p+j-2r-1)};q^{2p(2p+j)})\\
&\qquad \qquad \qquad \qquad  \times m\Big ({-}q^{p(2p+j)(j(2p+j)+2k)-tj(2p+j)},*;q^{2pj(2p+j)^2}\Big ).
\end{align*}}%

\noindent In the remaining sum, we replace $t\to 2pi+m$, so that we can rewrite the summation symbol as
\begin{equation*}
\sum_{t=0}^{2p(2p+j)-1} \to \sum_{m=0}^{2p-1}\sum_{i=0}^{2p+j-1}.
\end{equation*}
This gives
{\allowdisplaybreaks \begin{align*}
&f_{1,2p+j,2p(2p+j)}(q^{k-r},-q^{p(2p+j-2r-1)};q)\\
&\qquad \sim  \sum_{m=0}^{2p-1}\sum_{i=0}^{2p+j-1}(-1)^{2pi+m}q^{(2pi+m)(k-r)}
q^{\binom{2pi+m}{2}}j(-q^{(2p+j)(2pi+m)}q^{p(2p+j-2r-1)};q^{2p(2p+j)})\\
&\qquad \qquad \qquad  \times m\Big ({-}q^{p(2p+j)(j(2p+j)+2k)-(2pi+m)j(2p+j)},*;q^{2pj(2p+j)^2}\Big ).
\end{align*}}%

\noindent Again, let us pull the $i$ out of the theta function.  Using (\ref{equation:j-elliptic}), we get
{\allowdisplaybreaks \begin{align*}
&f_{1,2p+j,2p(2p+j)}(q^{k-r},-q^{p(2p+j-2r-1)};q)\\
&\qquad \sim  \sum_{m=0}^{2p-1}\sum_{i=0}^{2p+j-1}(-1)^{2pi+m}q^{(2pi+m)(k-r)}q^{\binom{2pi+m}{2}}\\
&\qquad \qquad  \times q^{-2p(2p+j)\binom{i}{2}-i(m(2p+j)+p(2p+j-2r-1))}j(-q^{(2p+j)m}q^{p(2p+j-2r-1)};q^{2p(2p+j)})\\
&\qquad \qquad \qquad  \times m\Big ({-}q^{p(2p+j)(j(2p+j)+2k)-(2pi+m)j(2p+j)},*;q^{2pj(2p+j)^2}\Big ).
\end{align*}}%

\noindent Rewriting the exponents, we get
{\allowdisplaybreaks \begin{align*}
&f_{1,2p+j,2p(2p+j)}(q^{k-r},-q^{p(2p+j-2r-1)};q)\\
&\qquad \sim  \sum_{m=0}^{2p-1}(-1)^{m}q^{\binom{m}{2}+m(k-r)}j(-q^{(2p+j)m}q^{p(2p+j-2r-1)};q^{2p(2p+j)})\\
&\qquad \qquad \qquad \times \sum_{i=0}^{2p+j-1}q^{-2pj\binom{i+1}{2}+i(p(2k+j)-jm)}\\
&\qquad \qquad \qquad \qquad \times  m\Big ({-}q^{p(2p+j)(j(2p+j)+2k)-(2pi+m)j(2p+j)},*;q^{2pj(2p+j)^2}\Big ).
\end{align*}}%

\noindent We again rewrite our expression in a form suitable for Theorem \ref{theorem:msplit-general-n}.  Again we obtain
{\allowdisplaybreaks \begin{align*}
&f_{1,2p+j,2p(2p+j)}(q^{k-r},-q^{p(2p+j-2r-1)};q)\\
&\qquad \sim  \sum_{m=0}^{2p-1}(-1)^{m}q^{\binom{m}{2}+m(k-r)}j(-q^{(2p+j)m}q^{p(2p+j-2r-1)};q^{2p(2p+j)})\\
&\qquad  \qquad \times \sum_{i=0}^{2p+j-1}q^{-2pj\binom{i+1}{2}}(q^{p(2k+j)-jm})^{i}\\
&\qquad  \qquad \qquad \times  m\Big ({-}q^{2pj(\binom{2p+j}{2}-i(2p+j))}(q^{p(2k+j)-jm})^{2p+j},*;q^{2pj(2p+j)^2}\Big ).
\end{align*}}%
Using Theorem \ref{theorem:msplit-general-n} yields
{\allowdisplaybreaks \begin{align}
&f_{1,2p+j,2p(2p+j)}(q^{k-r},-q^{p(2p+j-2r-1)};q)
\label{equation:heuristicPartB}\\
&\qquad \sim  \sum_{m=0}^{2p-1}(-1)^{m}q^{\binom{m}{2}+m(k-r)}j(-q^{(2p+j)m}q^{p(2p+j-2r-1)};q^{2p(2p+j)})
\notag\\
&\qquad \qquad \qquad \times   m(-q^{p(2k+j)-jm},*;q^{2pj}).
\notag
\end{align}}%

Now let us combine the two double-sums (\ref{equation:heuristicPartA}) and (\ref{equation:heuristicPartB}).  This gives
{\allowdisplaybreaks \begin{align*}
(q)_{\infty}^3&\mathcal{C}_{2k,2r}^{(p,2p+j)}(q)\\
&=f_{1,2p+j,2p(2p+j)}(q^{1+k+r},-q^{p(2p+j+2r+1)};q)
-f_{1,2p+j,2p(2p+j)}(q^{k-r},-q^{p(2p+j-2r-1)};q)\\
&\qquad  \sim \sum_{m=0}^{2p-1}(-1)^{m}q^{\binom{m+1}{2}+m(k+r)}j(-q^{(2p+j)m+p(2p+j+2r+1)};q^{2p(2p+j)})\\
&\qquad \qquad  \times m(-q^{p(2k+j)-jm},*;q^{2pj})\\
&\qquad \qquad \qquad - \sum_{m=0}^{2p-1}(-1)^{m}q^{\binom{m}{2}+m(k-r)}j(-q^{(2p+j)m}q^{p(2p+j-2r-1)};q^{2p(2p+j)})\\
&\qquad \qquad \qquad \qquad \times   m(-q^{p(2k+j)-jm},*;q^{2pj}).
\end{align*}}%

We take advantage of symmetries to rewrite the above expression.  The two $m=0$ terms cancel.  Indeed, if we use (\ref{equation:j-flip}), if follows that the coefficient of
\begin{equation*}
m(-q^{p(2k+j)},*;q^{2pj})
\end{equation*}
then evaluates to
{\allowdisplaybreaks \begin{align*}
&j(-q^{p(2p+j+2r+1)};q^{2p(2p+j)}) - j(-q^{p(2p+j-2r-1)};q^{2p(2p+j)}) \\
&\qquad =j(-q^{p(2p+j+2r+1)};q^{2p(2p+j)}) - j(-q^{2p(2p+j)-p(2p+j-2r-1)};q^{2p(2p+j)})\\
&\qquad =j(-q^{p(2p+j+2r+1)};q^{2p(2p+j)}) - j(-q^{p(4p+2j-2p-j+2r+1)};q^{2p(2p+j)})
=0.
\end{align*}}%
It is also true that the two $m=p$ terms cancel.  The coefficient of 
\begin{equation*}
m(-q^{p(2k+j)-jp},*;q^{2pj})=m(-q^{2pk},*;q^{2pj})
\end{equation*}
reduces to zero.  We pull out a common factor, rewrite an exponent, and then use (\ref{equation:j-elliptic}).  This reads
{\allowdisplaybreaks \begin{align*}
&(-1)^{p}q^{\binom{p+1}{2}+p(k+r)}j(-q^{(2p+j)p+p(2p+j+2r+1)};q^{2p(2p+j)})\\
&\qquad \qquad -(-1)^{p}q^{\binom{p}{2}+p(k-r)}j(-q^{(2p+j)p+p(2p+j-2r-1)};q^{2p(2p+j)})\\
&\qquad =(-1)^{p}q^{\binom{p}{2}+p(k-r)}\Big ( q^{p(2r+1)}j(-q^{(2p+j)2p+p(2r+1)};q^{2p(2p+j)})
-j(-q^{p(2r+1)};q^{2p(2p+j)})\Big )
= 0.
\end{align*}}%

Breaking up the two sums over $m$ then gives
{\allowdisplaybreaks  \begin{align}
(q)_{\infty}^3\mathcal{C}_{2k,2r}^{(p,2p+j)}(q)
&  \sim \sum_{m=1}^{p-1}(-1)^{m}q^{\binom{m+1}{2}+m(k+r)}j(-q^{(2p+j)m+p(2p+j)+p(2r+1)};q^{2p(2p+j)})
\label{equation:heuristicFourSums}\\
&\qquad   \times m(-q^{p(2k+j)-jm},*;q^{2pj})
\notag\\
&\qquad + \sum_{m=p+1}^{2p-1}(-1)^{m}q^{\binom{m+1}{2}+m(k+r)}j(-q^{(2p+j)m+p(2p+j)+q(2r+1)};q^{2p(2p+j)})
\notag\\
&\qquad   \qquad \times m(-q^{p(2k+j)-jm},*;q^{2pj})
\notag\\
&\qquad  - \sum_{m=1}^{p-1}(-1)^{m}q^{\binom{m}{2}+m(k-r)}j(-q^{(2p+j)m+p(2p+j)-p(2r+1)};q^{2p(2p+j)})
\notag\\
&\qquad  \qquad \times   m(-q^{p(2k+j)-jm},*;q^{2pj})
\notag\\
&\qquad  - \sum_{m=p+1}^{2p-1}(-1)^{m}q^{\binom{m}{2}+m(k-r)}j(-q^{(2p+j)m+p(2p+j)-p(2r+1)};q^{2p(2p+j)})
\notag\\
&\qquad  \qquad \times   m(-q^{p(2k+j)-jm},*;q^{2pj}).
\notag
\end{align}}%
In (\ref{equation:heuristicFourSums}), we replace $m$ with $p-m$ in the first and third sums.  In the second and fourth sums replace $m$ with $m+p$.  This brings us to 
{\allowdisplaybreaks  \begin{align*}
(q)_{\infty}^3&\mathcal{C}_{2k,2r}^{(p,2p+j)}(q)\\
&\sim \sum_{m=1}^{p-1}(-1)^{p-m}q^{\binom{p-m+1}{2}+(p-m)(k+r)}j(-q^{2p(2p+j)-m(2p+j)+p(2r+1)};q^{2p(2p+j)})\\
& \qquad  \times m(-q^{p(2k+j)-j(p-m)},*;q^{2pj})\\
& \qquad + \sum_{m=1}^{p-1}(-1)^{m+p}q^{\binom{m+p+1}{2}+(m+p)(k+r)}j(-q^{(2p+j)m+2p(2p+j)+p(2r+1)};q^{2p(2p+j)})\\
&\qquad   \qquad \times m(-q^{p(2k+j)-j(m+p)},*;q^{2pj})\\
&\qquad  - \sum_{m=1}^{p-1}(-1)^{p-m}q^{\binom{p-m}{2}+(p-m)(k-r)}j(-q^{2p(2p+j)-m(2p+j)-p(2r+1)};q^{2p(2p+j)})\\
&\qquad  \qquad \times   m(-q^{p(2k+j)-j(p-m)},*;q^{2pj})\\
&\qquad  - \sum_{m=1}^{p-1}(-1)^{m+p}q^{\binom{m+p}{2}+(m+p)(k-r)}j(-q^{(2p+j)m+2p(2p+j)-p(2r+1)};q^{2p(2p+j)})\\
&\qquad  \qquad \times   m(-q^{p(2k+j)-j(m+p)},*;q^{2pj}).
\end{align*}}%
Rewriting the theta functions using (\ref{equation:j-elliptic}) and (\ref{equation:j-flip}), and tweaking the second and fourth Appell function with (\ref{equation:mxqz-flip}), we have
{\allowdisplaybreaks \begin{align*}
(q)_{\infty}^3&\mathcal{C}_{2k,2r}^{(p,2p+j)}(q)\\
&  \sim \sum_{m=1}^{p-1}(-1)^{p-m}q^{\binom{p-m+1}{2}+(p-m)(k+r)}q^{m(2p+j)-p(2r+1)}\\
& \qquad \times j(-q^{-m(2p+j)+p(2r+1)};q^{2p(2p+j)})
 m(-q^{2pk+jm},*;q^{2pj})\\
& \qquad - \sum_{m=1}^{p-1}(-1)^{m+p}q^{\binom{m+p+1}{2}+(m+p)(k+r)}q^{-(2p+j)m-p(2r+1)}\\
& \qquad \qquad \times j(-q^{(2p+j)m+p(2r+1)};q^{2p(2p+j)})
 q^{-2pk+jm}m(-q^{-2pk+jm},*;q^{2pj})\\
& \qquad - \sum_{m=1}^{p-1}(-1)^{p-m}q^{\binom{p-m}{2}+(p-m)(k-r)}\\
& \qquad \qquad \times j(-q^{m(2p+j)+p(2r+1)};q^{2p(2p+j)})
   m(-q^{2pk+jm},*;q^{2pj})\\
& \qquad + \sum_{m=1}^{p-1}(-1)^{m+p}q^{\binom{m+p}{2}+(m+p)(k-r)}\\
& \qquad \qquad \times j(-q^{-m(2p+j)+p(2r+1)};q^{2p(2p+j)})
   q^{jm-2pk}m(-q^{-2pk+jm},*;q^{2pj}).
\end{align*}}%

\noindent Rewriting the leading $q$-exponents, we have
{\allowdisplaybreaks \begin{align*}
(q)_{\infty}^3&\mathcal{C}_{2k,2r}^{(p,2p+j)}(q)\\
&  \sim (-1)^{p}q^{\binom{p}{2}-p(r-k)}\sum_{m=1}^{p-1}(-1)^{m}q^{\binom{m+1}{2}+m(r-k-p)}\\
& \qquad \qquad \times q^{m(2p+j-(2r+1))}j(-q^{-m(2p+j)+p(2r+1)};q^{2p(2p+j)})
 m(-q^{2pk+jm},*;q^{2pj})\\
&\qquad  - (-1)^{p}q^{\binom{p}{2}-p(r+k)}\sum_{m=1}^{p-1}(-1)^{m}q^{\binom{m+1}{2}+m(r+k-p)}\\
&\qquad  \qquad \times j(-q^{(2p+j)m+p(2r+1)};q^{2p(2p+j)}) 
m(-q^{-2pk+jm},*;q^{2pj})\\
&\qquad  - (-1)^{p}q^{\binom{p}{2}-p(r-k)}\sum_{m=1}^{p-1}(-1)^{m}q^{\binom{m+1}{2}+m(r-k-p)}\\
&\qquad \qquad \times j(-q^{m(2p+j)+p(2r+1)};q^{2p(2p+j)})
   m(-q^{2pk+jm},*;q^{2pj})\\
&\qquad  + (-1)^{p}q^{\binom{p}{2}-p(r+k)}\sum_{m=1}^{p-1}(-1)^{m}q^{\binom{m+1}{2}+m(r+k-p)}\\
&\qquad  \qquad \times q^{m(2p+j-(2r+1))} j(-q^{-m(2p+j)+p(2r+1)};q^{2p(2p+j)})
m(-q^{-2pk+jm},*;q^{2pj}).
\end{align*}}%

\noindent Combining like terms, we have modulo a theta function that
{\allowdisplaybreaks \begin{align}
&(q)_{\infty}^{3}\mathcal{C}_{2k,2r}^{(p,2p+j)}(q)\\
& \sim 
-(-1)^{p}q^{\binom{p}{2}-p(r+k)}\sum_{m=1}^{p-1}(-1)^{m}q^{\binom{m+1}{2}+m(r+k-p)}
\notag\\
&  \qquad \times\left (  j(-q^{(2p+j)m+p(2r+1)};q^{2p(2p+j)}) -
q^{m(2p+j-(2r+1))} j(-q^{-m(2p+j)+p(2r+1)};q^{2p(2p+j)})\right )
\notag \\
&  \qquad \qquad \times 
m(-q^{-2pk+jm},*;q^{2pj})
\notag \\
&  \qquad - (-1)^{p}q^{\binom{p}{2}-p(r-k)}\sum_{m=1}^{p-1}(-1)^{m}q^{\binom{m+1}{2}+m(r-k-p)}
\notag\\
&  \qquad \qquad \times \left ( j(-q^{m(2p+j)+p(2r+1)};q^{2p(2p+j)})
- q^{m(2p+j-(2r+1))}j(-q^{-m(2p+j)+p(2r+1)};q^{2p(2p+j)})\right )
\notag \\
&  \qquad \qquad \qquad \times m(-q^{2pk+jm},*;q^{2pj}).\notag
\end{align}}%
We combine the sums for the final result.


\section{Acknowledgments}

The study was carried out with the financial support of the Ministry of Science and Higher Education of the Russian Federation in the framework of a scientific project under agreement No. 075-15-2025-013.  We would also like to thank Frank Garvan for his help with his Maple packages.  Finally we would like to thank the speakers of the seminar ``Modularity, geometry, and physics'' held at  the Euler International Mathematical Institute on July-August 2024 for productive discussions on \cite{DMZ} and related topics.

\end{document}